\documentclass[11pt]{article}
\usepackage[a4paper,margin=1in]{geometry}
\usepackage{amsmath,amsthm,amssymb,mathtools}
\usepackage{graphicx}
\usepackage{subcaption}
\usepackage{authblk}

\usepackage{hyperref}
\hypersetup{colorlinks=true,linkcolor=blue,citecolor=blue,urlcolor=blue}
\usepackage{xcolor}
\usepackage{booktabs}
\usepackage{mathabx}
\usepackage{bbm}

\usepackage{algorithm}
\usepackage{algpseudocode}

\usepackage{wrapfig}
\usepackage{caption}
\usepackage{adjustbox}

\usepackage{todonotes}

\usepackage{url}

\usepackage{cite}

\usepackage{tikz}
\usetikzlibrary{calc}
\usetikzlibrary{cd, arrows, shapes}
\usepackage[all]{xy}

\newtheorem{theorem}{Theorem}[section]
\newtheorem{lemma}[theorem]{Lemma}
\newtheorem{proposition}[theorem]{Proposition}
\newtheorem{corollary}[theorem]{Corollary}
\theoremstyle{definition}
\newtheorem{definition}[theorem]{Definition}
\newtheorem{remark}[theorem]{Remark}

\newcommand{\RR}{\mathbb R}

\newcommand{\CalM}{\mathcal M}

\newcommand{\init}{u_0} 

\newcommand{\figwidth}{0.75\linewidth}

\usepackage{enumerate}

\newcommand{\subjclass}[2][2020]{%
  \par\medskip
  \noindent\textbf{#1 Mathematics Subject Classification.} #2\par
}

\newcommand{\keywords}[1]{%
  \par\medskip
  \noindent\textbf{Keywords and phrases.} #1\par
}

\title{Reduced Order Modeling of One-Dimensional Conservative PDEs via the Cumulative Distribution Transform
\thanks{HA, AS, and SS are partially supported by the Office of Naval Research (ONR) under Award NO: N00014-24-1-2147, NSF grant DMS-2408877, and the Air Force Office of Scientific Research (AFOSR) under Award NO: FA9550-25-1-0231. GR, IM, KP are partially supported by NIH Award GM130825, and ONR Award N000142212505.}}

\author[1]{Harbir Antil}
\author[2]{Rocío Díaz Martín}
\author[3]{Ivan V. Medri}
\author[3]{Kristofor E. Pas}
\author[3]{Gustavo K. Rohde}
\author[1]{Aryan Saxena}
\author[1]{Sarswati Shah}

\affil[1]{Center for Mathematics and Artificial Intelligence and Department of Mathematical Sciences, George Mason University,  Fairfax, Virginia 22030}
\affil[2]{Department of Mathematics, Florida State University, FL 32306}
\affil[3]{University of Virginia,
Charlottesville, VA 22908}

\date{}

\begin{document}

\maketitle

\vspace{-1cm}

\begin{abstract}We propose a reduced order modeling (ROM) framework for one-dimensional conservative PDEs based on the cumulative distribution transform (CDT). The CDT maps nonnegative, equal-mass states into a Hilbert space in which onde-dimensional Wasserstein distances become weighted $L^2$ distances and translations become affine shifts. This makes the transform especially suited for transport-dominated dynamics, where Eulerian linear-subspace ROMs often suffer from slow decay of Kolmogorov widths. \par
We study this phenomenon for scalar conservative dynamics by analyzing the solution manifold in CDT coordinates. For linear transport, the transformed solution manifold is contained in the two-dimensional space spanned by the transformed initial datum and the constant function, and has zero Kolmogorov $2$-width. For nonlinear hyperbolic conservation laws, we prove two complementary types of estimates: robust $O(n^{-1})$ bounds that rely only on the conservative transport structure and remain meaningful after shock formation, and sharper $O(n^{-2})$ bounds in smooth pre-shock regimes. For conservative advection-diffusion, we show that the CDT trajectory remains within distance $O(\sqrt{DT})$ of the pure-transport plane, and we also obtain sharper $O(D^2T^2)$ estimates under additional regularity or away from initial layers. In both cases, the zero 2-width behavior of linear transport is recovered as the diffusion coefficient tends to zero. \par
Motivated by these estimates, we develop a CDT-POD numerical scheme: snapshots are mapped to CDT space, Proper Orthogonal Decomposition (POD) is performed in transformed coordinates, and the inverse CDT is used to reconstruct physical states. Numerical experiments for several transport-dominated dynamics show that CDT-POD can capture solution manifolds with substantially fewer modes than Eulerian POD. 
\end{abstract}



\clearpage
\tableofcontents

\section{Introduction and Related Work}\label{sec: intro}

    Conservative partial differential equations (PDEs) are central models in fluid dynamics, kinetic theory, Hamiltonian systems, and many other areas of continuum science \cite{evans2010partial,dafermos2005hyberbolic,cercignani1988boltzmann,morrison1998hamiltonian}. Their numerical approximation often requires high-dimensional simulations or measurements, yet the resulting solution families may contain exploitable low-dimensional structure once they are represented in appropriate coordinates. This observation is the starting point of reduced order modeling (ROM), whose goal is to replace high-fidelity models by lower-dimensional surrogates that retain the dominant behavior of the dynamics \cite{FBlack_PSchulze_BUnger_2020a,BPeherstorfer_2020a,amsallem2012nonlinear}. Such reductions are useful for data compression, flow analysis and control, forecasting, and data-driven discovery of governing equations \cite{rowley2017model,brunton2016discovering,gonnella2026nonlinear}.
    
    A standard ROM pipeline begins with a collection of solution snapshots and seeks a low-dimensional linear space that approximates them well. Singular Value Decomposition (SVD), Proper Orthogonal Decomposition (POD) \cite{berkooz1993proper,HAntil_MHeinkenschloss_DCSorensen_2013a}, and Dynamic Mode Decomposition (DMD) \cite{schmid2022dynamic,tu2013dynamic} are fundamental tools of this type. Their performance is closely related to the approximability of the solution manifold by finite-dimensional linear spaces. Kolmogorov $n$-widths quantify the smallest possible worst-case error achievable by an $n$-dimensional linear approximation, and they therefore provide a natural theoretical lens for understanding when linear ROMs can succeed. This point of view is classical in reduced-basis and projection-based ROMs, appears in discussions of their limitations \cite{algoritmy,KCarlberg_MBarone_HAntil_2017a}, and has also been made precise for linear dynamical systems through connections with Hankel singular values \cite{unger2019kolmogorov,HAntil_MHeinkenschloss_RHoppe_DCSorensen_2010b,HAntil_MHeinkenschloss_RHoppe_2010c}.

    The main obstruction addressed in this work is that many mass-preserving scalar conservative PDEs generate moving profiles. This includes, for example, first-order hyperbolic conservation laws, where contacts, shocks, rarefactions, and traveling fronts arise, as well as conservative advection-diffusion equations, where advective motion is coupled with parabolic smoothing. In Eulerian coordinates, such moving features are poorly aligned across time. Consequently, even simple translated profiles may require many linear modes. This phenomenon is often described as a Kolmogorov barrier: for translated profiles with jump discontinuities, the Kolmogorov $n$-width decays only at the rate $n^{-1/2}$, and analogous barriers arise for wave problems with discontinuous data \cite{greif2019decay,arbes2025kolmogorov}; see also \cite[Section 5.1]{algoritmy}. At the discrete level, this slow decay is often reflected in the snapshot singular-value spectrum.
    
    Several approaches have been developed to mitigate this transport barrier. One class replaces fixed linear spaces by nonlinear manifolds, adaptive bases, dynamically transformed modes, or other geometric parametrizations \cite{peherstorfer2022breaking,hesthaven2026nonlinear,wotte2026model}. Another class seeks to align transported features before applying linear compression, using registration, displacement interpolation, monotone rearrangement, transported subspaces, or transport maps \cite{cagniart2018model,rim2018model,rim2018displacement,rim2023manifold,taddei2020registration,welper2017interpolation,nonino2023overcoming}. A closely related viewpoint is Lagrangian: if the coordinates move with the transported structure, then much of the apparent Eulerian complexity can disappear. This idea appears classically in Lagrangian formulations of advection-diffusion, where advection is absorbed into the coordinates and the remaining evolution is governed by an effective diffusion operator \cite{thiffeault2003advection}; it also underlies recent Lagrangian ROM and physics-informed approaches for advection-dominated dynamics \cite{lu2020lagrangian,mojgani2017lagrangian,mojgani2023kolmogorov}. These methods can be highly effective, but they typically require choosing, estimating, or learning deformation maps, which can become a delicate task. 
    
    In this work, we use the Cumulative Distribution Transform (CDT) \cite{park2018cumulative}, denoted by $u\mapsto\widehat{u}$, as a transport-based coordinate representation and perform POD in CDT space. For one-dimensional  nonnegative, equal-mass states, translations become additive shifts under the CDT. Thus, a family of translated profiles that is poorly approximated by Eulerian linear spaces becomes an affine line in CDT space. This indicates that solution manifolds dominated by advective motion may become significantly more compressible after transformation. Moreover, a central advantage of using the CDT for ROM is that it maps probability densities into a Hilbert space in which one-dimensional Wasserstein distances become $L^2$ distances. This realization of the Wasserstein metric as an $L^2$ norm permits the use of structural properties of Wasserstein geometry (that characterize mass displacement in a physically intuitive way)  in tandem with the computational power of Euclidean space methods. 

    The CDT was introduced in \cite{park2018cumulative} and has since been used in image and signal processing,  classification, parametric  estimation, and data-driven modeling \cite{rubaiyat2020parametric,rubaiyat2022nearest,rubaiyat2024end,rubaiyat2024data}; implementations are available through PyTransKit \cite{pytranskit}. In this work, we use the notation of the CDT and inverse-CDT framework, including its measure formulation, closely aligned with the optimal transport viewpoint, as developed in \cite{akram22sign,ivan2024data}. Our aim is to build on the insights provided by the Radon-CDT model-reduction framework of \cite{ren2021model,long2025reduced} by developing and testing a CDT-POD approach for scalar conservative PDEs.
    
    Closely related work \cite{liu2022neural,jin2025parameterized,zuo2026numerical} represents evolving densities as push-forwards of a fixed reference measure through parameterized families of transport maps. Our CDT-POD approach shares this Lagrangian transport-map viewpoint, but learns a data-adapted POD subspace from the exact one-dimensional monotone optimal transport maps.
    
    We consider scalar evolution problems for a state $u:(a,b)\times[0,T]\to\mathbb R$, with the convention $a=-\infty$ and $b=+\infty$ when the dynamics are posed on the whole real line. ROMs may be used either to represent the solution compactly or to infer the governing evolution operator from data. In the present work we focus on the first task: we seek accurate (linear) low-dimensional compression and reconstruction of the solution manifold
    \begin{equation}\label{eq: sol manifold}
         \mathcal M_{[t_0,T]}:=\{u(\cdot,t):\, t\in[t_0,T]\},
        \qquad \text{for some }t_0\ge0.
    \end{equation}
    
    The evolution equations considered in this article are one-dimensional scalar conservative PDEs, including first-order hyperbolic conservation laws such as linear advection, inviscid Burgers' equation, traffic flow, and Buckley--Leverett dynamics, as well as conservative advection-diffusion. This conservative structure ensures that, under suitable boundary or decay conditions, each snapshot can be interpreted as a nonnegative density with common total mass, making the CDT applicable.

    The CDT-POD procedure is simple to describe: given a collection of discretized time snapshots $ \{u(\cdot, t_j)\}$ (as in Figure \ref{fig: discretization}),  transform them by applying the CDT $\{\widehat{u}(\cdot,t_j)\}$ (as in Figure \ref{fig: HFM and Transform}), perform POD in transform space, and then invert the transform. The objective of this work, however, is not only to apply this procedure numerically, but also to move toward a better understanding of why it works, when it should be expected to work, and where it may fail. Since any improvement over native POD depends on the geometry of the transformed solution manifold, we study the transformed dynamics more explicitly and use this analysis to interpret the behavior observed in the numerical experiments.
    
    \begin{figure}[h!]
        \centering
        \includegraphics[width=1\linewidth]{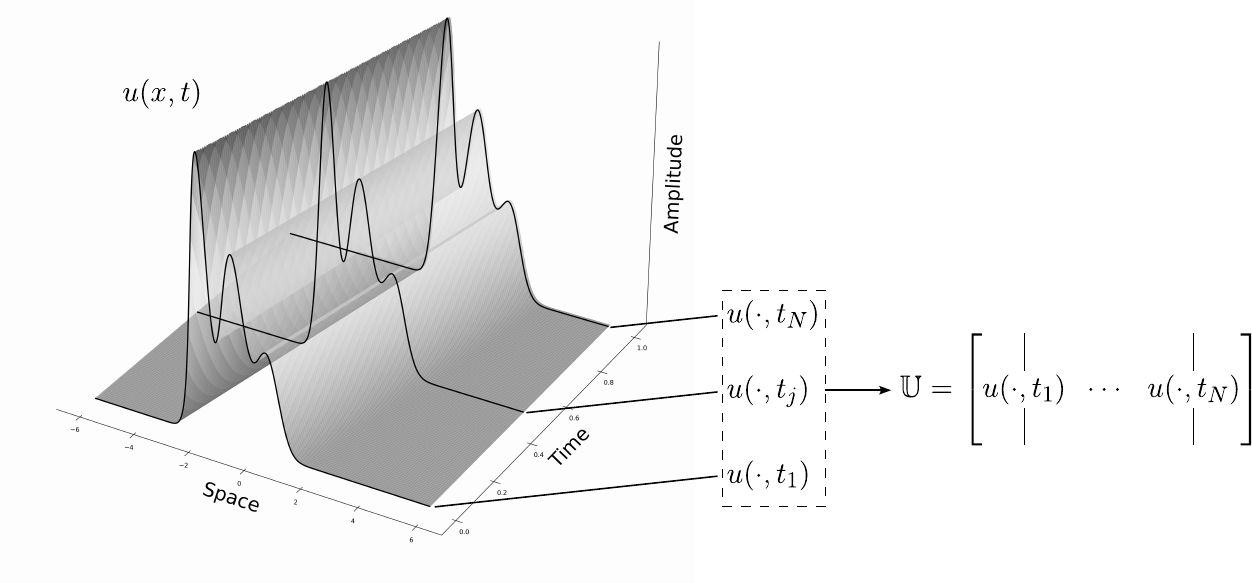}
        \vspace{-0.4in}
        \caption{Sketch of a solution $u(x,t)$ of a conservative evolution equation and its discretized snapshot matrix $\mathbb U$.}
        \label{fig: discretization}
    \end{figure}
    
    \begin{figure}[h!]
        \centering
        \includegraphics[width=1\linewidth]{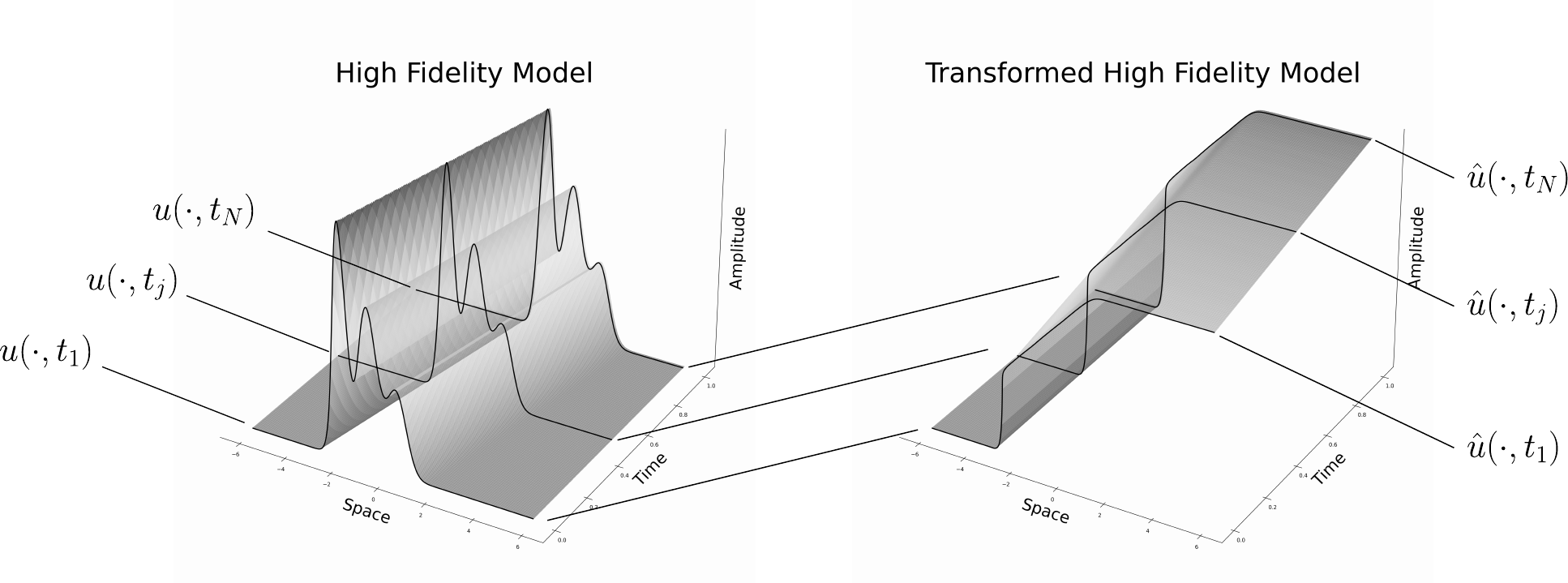}
        \vspace{-0.4in}
        \caption{Sketch of solution snapshots in native space (left) and CDT space (right), and corresponding snapshots $\{u(\cdot,t_j)\}_{j=1}^N$ and $\{\widehat u(\cdot,t_j)\}_{j=1}^N$.}
        \label{fig: HFM and Transform}
    \end{figure}
    
    A useful way to think about the CDT representation is that it gives two complementary pictures of the same evolution. First, the transformed solution can be viewed as a curve $t \mapsto \widehat u(\cdot,t)$ in a weighted $L^2$ space. Second, for each mass label $\xi$, the curve $t \mapsto \widehat u(\xi,t)$ tracks the position of a fixed amount of accumulated mass. In this sense, the transformed variables behave like characteristic-type coordinates. The important difference is that these cumulative-mass trajectories can still be defined in situations where classical characteristic curves are no longer the right global description, for instance after shocks form or for non-hyperbolic equations. This dual point of view is useful because different parts of the analysis require different types of information. Sometimes it is enough to work with the CDT trajectory as a curve $t\mapsto \widehat u(\cdot,t)$ in a weighted $L^2$ space and use general approximation results for Hilbert-valued curves. In other situations, especially for sharper estimates, it is useful to understand the pointwise behavior of $t\mapsto \widehat u(\xi,t)$. This is why the paper studies both $L^2$-type regularity and pointwise regularity of the transformed dynamics.
     
    The results obtained here should be understood in this spirit. For linear advection, the CDT representation makes the solution manifold exactly low-dimensional. For smooth hyperbolic equations before shocks, the transformed dynamics are closely related to classical characteristics, and this leads to Kolmogorov $n$-width estimates of order $\mathcal O(n^{-2})$ under suitable regularity assumptions. For entropy solutions, where classical characteristics may break down, one can still obtain estimates of order $\mathcal O(n^{-1})$ using the conservative structure of the equation and the transport interpretation of the CDT. For advection-diffusion, the CDT trajectory can be compared with the purely advective one, and the effect of diffusion can be bounded in a way that is consistent with the advective case as the diffusion coefficient tends to zero.
    
    In the numerical part of the paper, we do not compute Kolmogorov widths directly, since this is not feasible in practice. Instead, we compare native POD and CDT-POD using singular-value decay and reconstruction errors. These quantities are not the same as the theoretical $n$-widths, but they are the standard computable diagnostics for the low-rank behavior of the snapshot matrices. The purpose of the theory is therefore not to predict the numerical errors exactly, but to explain why the transformed snapshots are often more compressible and why this behavior depends on the underlying mass motion.
    
\paragraph{Contributions.}
    This work proposes and studies CDT-POD for one-dimensional scalar conservative PDEs, with the goal of understanding when the transport-based coordinates improve linear compressibility of solution manifolds. The main contributions are as follows.
    
    \begin{enumerate}[(i)]
        \item \textbf{A CDT-POD pipeline for conservative PDEs.} We apply the transform-reduce-invert strategy to nonnegative equal-mass snapshots: the snapshots are mapped to CDT space, POD is performed in the transformed coordinates, and the inverse CDT is used to reconstruct the physical states. This gives a nonlinear approximation in Eulerian variables while retaining a linear reduced representation in transform space. See an illustration in Figure \ref{fig: CDT_pipeline} (bottom). When the reduced CDT map remains monotone, the reconstruction is a valid
        mass-preserving density; the possible loss of monotonicity under linear
        truncation is identified as a limitation and a motivation for constrained
        extensions.
        
        \item \textbf{Equations in Eulerian, semi-Lagrangian, and CDT coordinates.} We derive formal transformed-coordinate representations of the conservative dynamics, see Table \ref{tab:equations-in-CDT-space}. These formulas clarify how cumulative-mass coordinates relate to Eulerian variables and to characteristic-type descriptions of transport. 
        
        \item \textbf{A geometric interpretation of transformed solution manifolds}. We view the CDT solution manifold both as a curve ($t\mapsto \widehat u(\cdot,t)$) in a weighted $L^2$ Hilbert space and as a family of cumulative-mass trajectories ($t\mapsto \widehat u(\xi,t)$). This helps explain
        why transport-dominated dynamics may become simpler in CDT space and why both $L^2$ and pointwise regularity estimates appear naturally.
        
        \item \textbf{Kolmogorov-width estimates in CDT space.} 
        For hyperbolic conservation laws, we prove robust $O(n^{-1})$ estimates for
        entropy solutions and sharper $O(n^{-2})$ estimates in smooth pre-shock
        regimes. For conservative advection-diffusion, we prove bounds showing that
        the transformed dynamics remain close to the pure-transport plane, with
        sharper estimates under additional regularity or away from initial layers.
        
        \item \textbf{Numerical comparison with Eulerian POD.} We compare native POD and CDT-POD through singular-value decay and
        reconstruction errors. The experiments show improved low-rank structure in
        CDT space for transport-dominated problems, while also revealing the need
        for monotonicity-aware reduced approximations.
    \end{enumerate}

    Our main results on quantitative Kolmogorov $n$-width estimates are Theorems \ref{thm:hyperbolic-cdt-lipschitz} and \ref{thm:hyperbolic-cdt-width}  for the case of  hyperbolic conservation laws, and Theorems \ref{thm:ADR-CDT-widths} and \ref{thm:ADR-CDT-widths-R} for advection-diffusion. The estimates in Theorems \ref{thm:hyperbolic-cdt-lipschitz} and
    \ref{thm:ADR-CDT-widths} should be read as robust transport-metric estimates, rather than as regularity estimates as in the case of Theorems \ref{thm:hyperbolic-cdt-width} and \ref{thm:ADR-CDT-widths-R}. That is, Theorems \ref{thm:hyperbolic-cdt-lipschitz} and \ref{thm:ADR-CDT-widths} do not rely on a smooth parametrization of the solution by classical characteristics, nor on differentiating the CDT trajectory. Instead, they use only the conservative structure of the
    equation, the interpretation of the dynamics as motion of mass, and the CDT isometry between $W_2$ and a weighted $L^2$ space. In the hyperbolic case, Theorem \ref{thm:hyperbolic-cdt-lipschitz}  gives a global-in-time Lipschitz bound for entropy solutions, and hence a general $\mathcal O(n^{-1})$ width estimate, even in regimes where shocks may form and the classical characteristic map is no longer invertible. In the advection-diffusion case (i.e., Theorem \ref{thm:ADR-CDT-widths}), the solution is compared with the purely advected density: diffusion only moves the CDT trajectory away from the rank-two transport plane by a distance of order $\sqrt{DT}$, where $D>0$ is the diffusion coefficient. These results are therefore complementary to the higher-regularity estimates obtained in Theorems \ref{thm:hyperbolic-cdt-width} and \ref{thm:ADR-CDT-widths-R}. Theorem \ref{thm:hyperbolic-cdt-width} provides sharper $\mathcal O(n^{-2})$ bounds, but only under smoothness, positivity, and pre-shock. For advection-diffusion, Theorem \ref{thm:ADR-CDT-widths-R} yields Kolmogorov $2$-width bound of order $D^2T^2$ away from initial layers or under regularity conditions on the initial condition. Thus the low-regularity theorems explain the stability and robustness of the CDT representation for conservative dynamics, whereas the higher-regularity results explain the improved compressibility observed in smooth regimes.
    
    Our analysis combines tools from optimal-transport and its dynamic formulation based on the continuity-equation to ensure mass-preservation \cite{benamou2000computational,otto2001geometry,ambrosio2005gradient,Santambrogio-OTAM, villani2003topics}, Hilbert-valued curve approximation and standard facts about Bochner integration \cite{hytonen2016analysis}, and general theory on scalar conservation laws \cite{evans2010partial}.
    
\paragraph{Organization.}
    Section \ref{sec: our systems} introduces the class of scalar conservative PDEs and the mass-preserving setting, and motivates the need for the CDT by relating it to characteristic curves. Section \ref{sec: CDT} recalls the CDT, its inverse, its isometric relation with the Wasserstein distance (earth mover's distance), and its properties such as translation and dilation equivariance. Section \ref{sec: examples x2} shows two examples, advection-diffusion with Dirac initial data and the linear transport equation, where dimensionality reduction with CDT can be shown analytically. Section \ref{sec:kolmogorov} recalls the notion of Kolmogorov widths as a quantifier for dimensionality reduction and the SVD method as a practical tool to find linear subspace approximations. Section \ref{sec:CDT-POD-method} presents the CDT-POD algorithm in practice and explains how to interpret the errors shown in the numerical example figures of Sections \ref{sec: numerics hyperbolic case} and \ref{sec: num adv-diff}. 
    
    Section \ref{sec: theory for conservative eq in general} begins the analysis of general scalar conservative equations of  form \eqref{eq: IBVP} in transformed coordinates (CDT-space). Section \ref{sec: lagrangian form} shows formal derivations of the evolution of PDEs in hybrid settings (using simultaneously native and transformed coordinates), and in fully transformed coordinates. Sections \ref{sec: reg theory} and \ref{sec: extra reg theory} introduce general regularity results needed to make the formal derivations valid for certain equations. The objective of both sections is the same but the techniques and applicability differ, and that is why we split the discussion in two parts. 
    
    Sections \ref{sec: hyperbolic case} and \ref{sec:adr-widths} have the main theorems and experiments of the paper. Sections \ref{sec: kolm hyp} and \ref{sec: kol adr} study CDT-space Kolmogorov width estimates for the family of hyperbolic equations \eqref{eq: hyp} and the particular case of advection-diffusion \eqref{eq:AD}, respectively. Numerical experiments are presented alongside the relevant PDE classes studied in Sections \ref{sec: hyperbolic case} and \ref{sec:adr-widths}, in Sections \ref{sec: numerics hyperbolic case} and \ref{sec: num adv-diff}. The code to reproduce the experiments of this paper, together with a tutorial, is available online at \path{https://github.com/rohdelab/CDT_POD}.
    
    Section \ref{sec: conclusions} summarizes conclusions and open directions.

\section{Background and the CDT-POD Method}

    This section motivates and introduces the main representation studied in the paper. The goal is to give the reader different examples, connections, and intuitions for why the CDT is a useful tool for representing the evolution of certain PDEs. We start with the scalar conservative equations considered in this work and motivate the use of mass-based coordinates by relating them to characteristic curves. We then recall the Cumulative Distribution Transform (CDT), its inverse, its Wasserstein interpretation, and some of its basic linearization properties. After that, we present two canonical examples where the CDT reveals simple linear low-dimensional structure. Finally, we introduce Kolmogorov widths and the CDT-POD procedure, emphasizing how the resulting reduced approximations should be compared in native and transformed coordinates, and how to interpret the error curves and figures reported later.

\subsection{Scalar Conservative PDEs and Motivation of This Work}\label{sec: our systems}

    In this work we consider the class of problems given by scalar evolution equations in conservative form,
    \begin{equation}\label{eq: IBVP}
        u_t+\partial_x (f[u])=0, 
        \qquad x\in (a,b),\quad t>0,
        \qquad u(x,0)=\init(x),
    \end{equation}
    where the function $u(x,t)$ is unknown, the \textit{flux} $f[u](x,t):=f(u(x,t),u_x(x,t),\dots)$ is determined by a known function $f$, and the \textit{initial state} $\init$ is  given. In the cases where $f[u](x,t) = f(u(x,t))$ we simply write $f(u)$. We have included a few examples in Table \ref{tab:scalar_conservation_laws}, which will be later examined in our experiments.

    \begin{wraptable}[9]{r}{0.55\textwidth}
    \vspace{-0.2in}
        \centering
        \caption{Examples of 1D scalar conservative PDEs.}
        \vspace{-0.15in}
        \begin{tabular}{ll}
            \toprule
            \textbf{Name} & \textbf{Equation} \\
            \midrule
            Advection--Diffusion 
            & $\displaystyle u_t + \partial_x(Au - Du_x) = 0$ \\[3pt]
    
            Inviscid Burgers 
            & $\displaystyle u_t + \partial_x\left(\tfrac{u^2}{2}\right) = 0$ \\[8pt]
    
            Traffic Flow 
            & $\displaystyle u_t + \partial_x\bigl(u(1-u)\bigr) = 0$ \\[4pt]
    
            Buckley-Leverett 
            & $\displaystyle u_t + \partial_x\left(\tfrac{u^2}{u^2+(1-u)^2}\right) = 0$ \\
            \bottomrule
        \end{tabular}
        \label{tab:scalar_conservation_laws}
    \end{wraptable}
    Throughout the paper, whenever the CDT is applied to a time-dependent solution, we assume that the snapshots are nonnegative and have common total mass. This is ensured either by zero-flux boundary conditions $f[u](a,t)=f[u](b,t)=0$ on a bounded interval $I = (a,b)$ or by sufficient decay of the flux at infinity on $\mathbb R$. \\

    When $f$ depends only on $u$, the problem becomes hyperbolic, and \eqref{eq: IBVP} can be thought of as a generalization of the transport problem
    \begin{equation}\label{eq: pure advection}
        \begin{cases}
            u_t + A \, u_x = 0, 
            \qquad  
            u(x,0)=\init(x),\\
            u(a,t)=0  \text{ if } A>0,  \text{ or }   
            u(b,t)=0 \ \text{ if } A<0 \quad \text{for } t>0,
        \end{cases}
    \end{equation}
    in the sense that the solutions are obtained by domain deformations (at least before shocks) that keep the values of $u$ in correspondence with the values of $u_0$ (see for example Figure \ref{fig: hyperbolic_conservation_laws_intro}).
    
    \begin{figure}[ht!]
        \centering
        \includegraphics[width=1\linewidth]{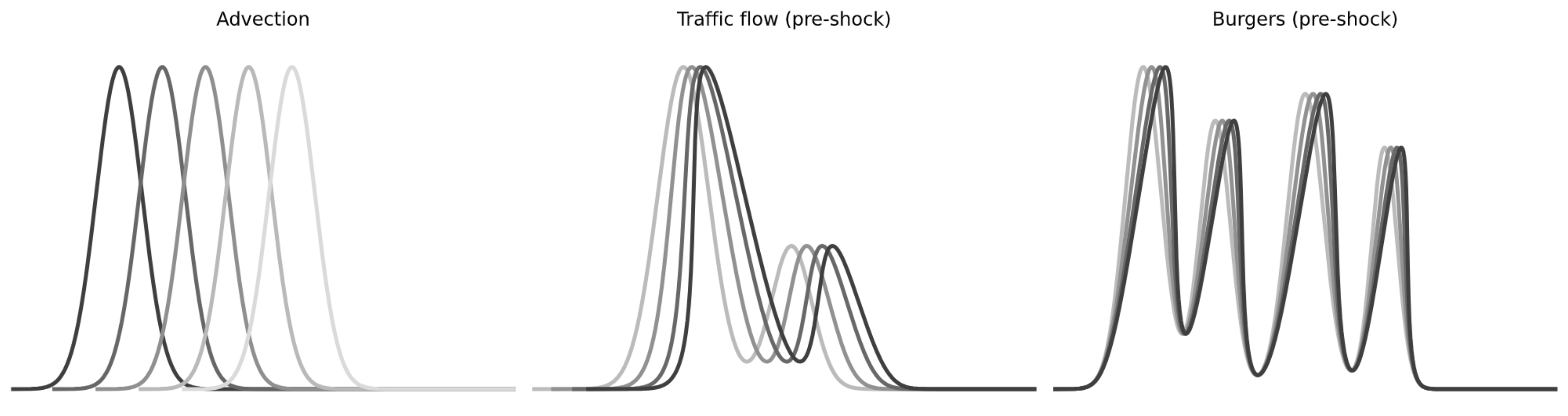}
        \vspace{-0.3in}
        \caption{Examples of hyperbolic conservation laws.}
        \label{fig: hyperbolic_conservation_laws_intro}
    \end{figure}

    As mentioned in Section \ref{sec: intro}, one objective of ROM is to find low-dimensional representations of the solution manifold $\mathcal M_{[t_0,T]}$. This could be naively done by performing SVD to the snapshot matrix $\mathbb U$ shown in Figure \ref{fig: discretization}, in order to approximate it by a rank-$k$ representation given by its truncated SVD.
    This approach implicitly assumes that $\mathcal{M}_{[t_0,T]}$ is close to a low-dimensional subspace. Nevertheless, even for the simple transport equation \eqref{eq: pure advection}, where $\mathcal M_{[0,T]} = \{\init(\cdot-At): \, t\in[0,T]\}$ consists of translations of a fixed initial profile $u_0$, that is not the case (see the left-most panel in Figure \ref{fig: hyperbolic_conservation_laws_intro}). In fact, for $u_0$ compactly supported, the family of translates $\mathcal{M}_{[0,T]}$ is not contained in any finite-dimensional linear subspace of functions and the decay of the singular values of the snapshot matrix $\mathbb{U}$ might be very slow. This is a reason why ROMs, such as POD, are known to have problems with advective behavior (translations) \cite{reiss2018shifted} (see also \cite{algoritmy,greif2019decay,arbes2025kolmogorov,peherstorfer2022breaking}).
    Nevertheless, as a manifold, $\mathcal{M}_{[t_0,T]}$ is parametrized only by the parameter $t$ and thus is expected to have a low-dimensional structure under the appropriate representation. \\
    
    This opens up the idea of using \textit{transforms}, which at their core are invertible maps between spaces of functions, and can be thought of as changes of coordinates. Of course, not every transform will provide a  linear low-dimensional representation of the solution manifold of a PDE. For instance, by applying the classical  \emph{Fourier transform} on the space variable to the linear transport dynamic \eqref{eq: pure advection} on the real line, the new transformed equation reads as 
        \begin{equation*}
        \mathcal F{u}_t(\xi,t) + iA\xi\,\mathcal F{u}(\xi,t)=0,
        \qquad \mathcal F {u}(\xi,0)=\mathcal F{\init}(\xi).
        \end{equation*}
    and so, for each $\xi\in\mathbb{R}$, $\mathcal F{u}(\xi,t)=\mathcal F{\init}(\xi)\,e^{-iA\xi t}= \int_{\RR} u_0(x) \, e^{-iA\xi t} \, e^{-i\xi x} \, dx $. Thus, the Fourier transform diagonalizes the dynamics. However, for generic initial data,  the transformed manifold $\mathcal{F}\left(\mathcal{M}_{[0,T]}\right) :=\{\mathcal{F}u(\cdot,t)\}$ intuitively lives in an infinite-dimensional space whose `basis' functions are the complex exponentials $\phi_\xi(x) :=e^{-i\xi x}$ with $\xi \in \mathbb{R}$. \\

    For a hyperbolic equation $u_t + (f(u))_x = 0$ with initial condition $u_0$, one transform that provides a linear low-dimensional representation of the solution manifold could be defined by the use of the method of characteristics.  Let $\Phi(\xi,t)$ be the characteristic flow satisfying 
    \begin{equation}\label{eq: height_characteristic_flow_ODE}
        \begin{cases}
            \Phi_t(\xi,t) = f'(u_0(\xi)),\\
            \Phi(\xi,0) = \xi.
        \end{cases}
    \end{equation}
    It is known that the curve $t\to \Phi(\xi,t)$, called characteristic curves,  tracks the position in space where $u(\cdot,t)$ has the same value (height) as the initial condition $u_0$ at the labeled initial position $\xi$ (see for example Fig. \ref{fig: charac-adv}). That is, $\Phi(\xi,t)$ can be defined implicitly by the \textit{registration} formula
    \begin{equation}\label{eq: implicit characteristic flow}
        u(\Phi(\xi,t),t) = u_0(\xi),
    \end{equation}
    which provides a unique correspondence between $u(\cdot,t)$ and $\Phi(\cdot,t)$ (if we restrict our search for $\Phi$ to increasing functions of $\xi$). It is easy to show from $\eqref{eq: height_characteristic_flow_ODE}$ that for $t$ before shock times, 
    \begin{equation*}
        \Phi(\xi,t) = \xi + tf'(u_0(\xi)).
    \end{equation*}
    By proposing a  transformation map of the form $u(x,t)\longmapsto \Phi(\xi,t)$, we have that in the new  space the transformed solution manifold is contained in a linear space with dimension at most two:
    $$\{\Phi(\cdot,t)\}_{t\in[0,T]} \subset \mathrm{span}\{ \mathrm{Id}, f'\circ \init \}.$$
    
    This seems to have solved our linear low-dimensional representation problem. Nevertheless, several difficulties arise if one tries to build a transform directly from characteristic curves. The function $\Phi(\xi,t)$ may be difficult to infer from data; the representation is valid only as long as the characteristic map $\xi\mapsto\Phi(\xi,t)$ remains invertible; after shocks or rarefactions form, one must work with entropy weak solutions and generalized notions of characteristics; and the approach does not extend directly to non-hyperbolic equations, such as viscous conservation laws that are important to better understand entropy solutions.\\

    \begin{wrapfigure}[8]{r}{0.6\linewidth}
        \centering
        \vspace{-0.19in}
        \includegraphics[width=0.6\linewidth]{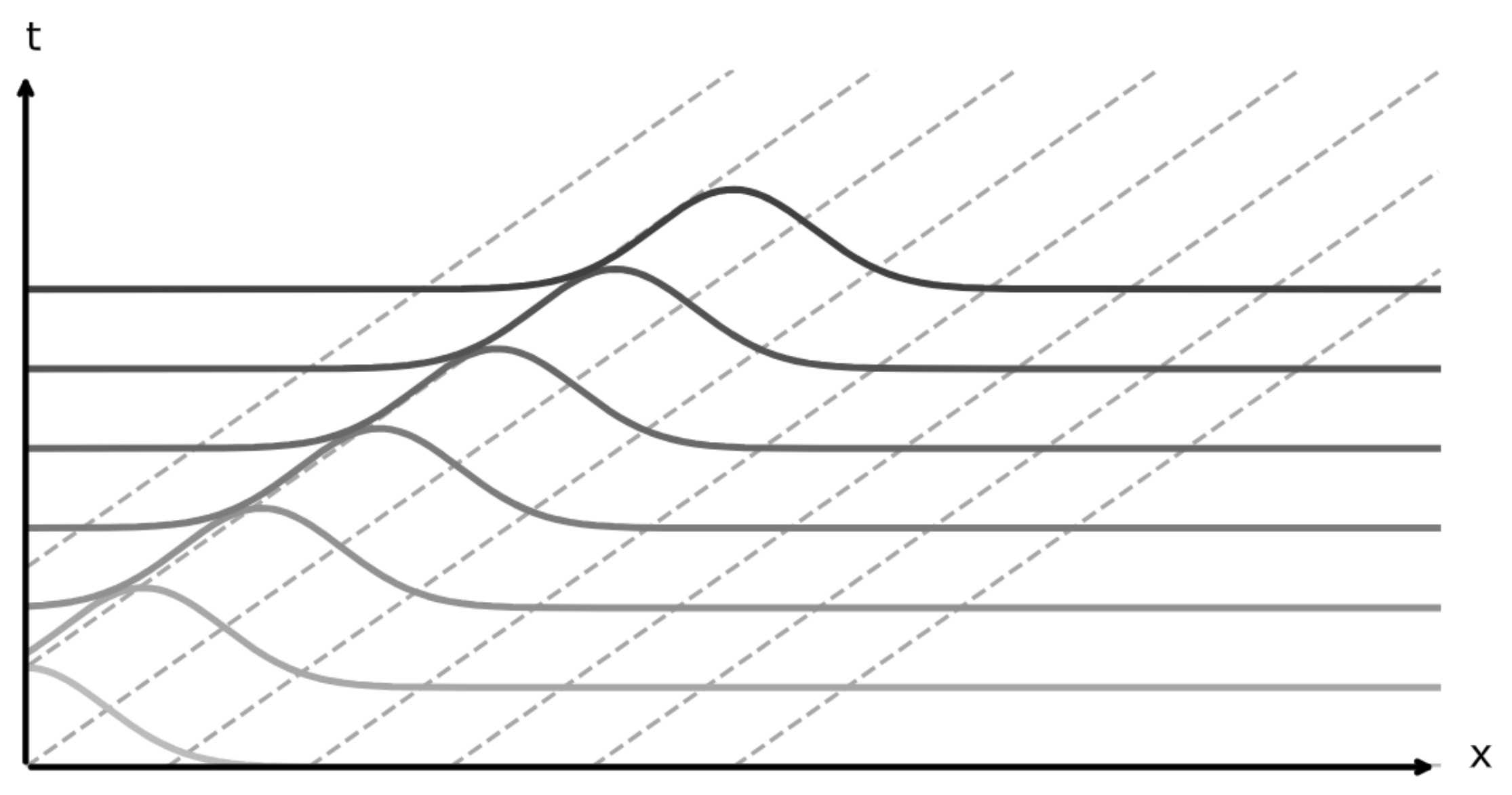}
        \vspace{-0.12in}
        \caption{Characteristics of pure advection equation.}
        \label{fig: charac-adv}
        \vspace{0.5in}
    \end{wrapfigure}

    Motivated by the Lagrangian viewpoint of the characteristics and the mass conservation nature of our equations of interest, and in a similar fashion to tracking the level curves of $u$, we replace the height-based labels of the characteristic curves by mass-based labels. Specifically, instead of tracking the point $\Phi(\xi,t)$ that originated from the spatial location $\xi$ and satisfies \eqref{eq: implicit characteristic flow}, we track the point $\widehat u(\xi,t)$ at which the evolving density $u(\cdot,t)$ has accumulated the same total mass that a reference density $r$ has accumulated up to $\xi$ (see Figure \ref{fig: mass_preservation}). That is, our transform is implicitly defined by
    \begin{equation}\label{eq: mass preserv}
        \int_{-\infty}^{\widehat u(\xi,t)}u(x,t)\,dx
        =
        \int_{-\infty}^{\xi}r(x)\,dx.
    \end{equation}

    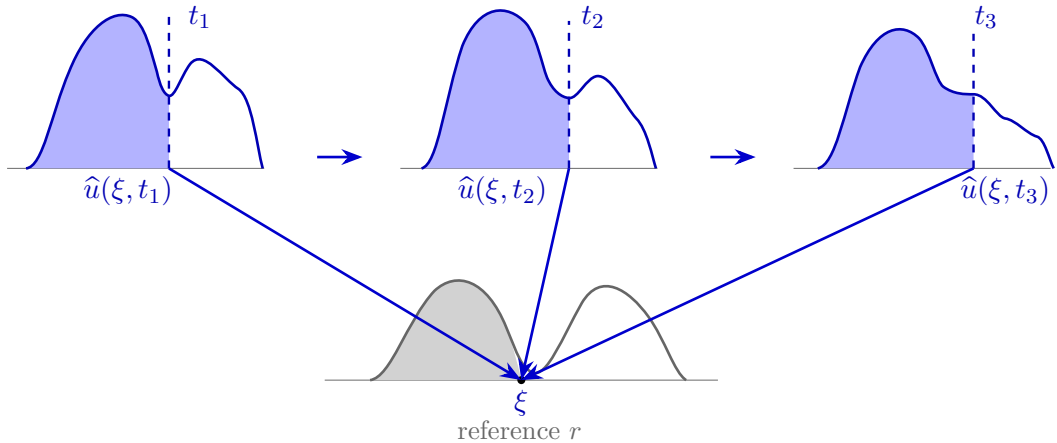
\begin{figure}[ht!]
        \centering
        \begin{tikzpicture}[
        x=1cm,y=1cm,
        >={Stealth[length=2.8mm]},
        bluecurve/.style={draw=blue!70!black, line width=1pt},
        refcurve/.style={draw=black!60, line width=1pt},
        massfill/.style={fill=blue!30, draw=none},
        refmass/.style={fill=gray!35, draw=none},
        maparrow/.style={->, draw=blue!80!black, line width=1pt},
        flowarrow/.style={->, draw=blue!80!black, line width=1pt},
        pickline/.style={draw=blue!70!black, dashed, line width=0.9pt}
        ]

        \begin{scope}[shift={(-5.2,2.8)}]
            \draw[black!55] (-1.6,0) -- (1.8,0);
        
            \fill[massfill]
            (-1.35,0)
              .. controls (-1.12,0.02) and (-1.02,0.72) .. (-0.78,1.28)
              .. controls (-0.55,1.86) and (-0.20,2.12) .. (0.05,2.00)
              .. controls (0.22,1.92) and (0.30,1.55) .. (0.38,1.28)
              .. controls (0.42,1.10) and (0.47,0.98) .. (0.54,0.96)
              -- (0.54,0) -- cycle;
        
            \draw[bluecurve]
            (-1.35,0)
              .. controls (-1.12,0.02) and (-1.02,0.72) .. (-0.78,1.28)
              .. controls (-0.55,1.86) and (-0.20,2.12) .. (0.05,2.00)
              .. controls (0.22,1.92) and (0.30,1.55) .. (0.38,1.28)
              .. controls (0.42,1.10) and (0.47,0.98) .. (0.54,0.96)
              .. controls (0.62,0.94) and (0.70,1.22) .. (0.82,1.38)
              .. controls (1.00,1.60) and (1.26,1.18) .. (1.46,1.05)
              .. controls (1.64,0.92) and (1.70,0.30) .. (1.78,0);
        
            \draw[pickline] (0.54,0) -- (0.54,2.02);
        
            \node[blue!70!black] at (0.95,2) {$t_1$};
            \node[blue!70!black] at (0,-0.3) {$\widehat{u}(\xi,t_1)$};
        \end{scope}
        
        \begin{scope}[shift={(0,2.8)}]
            \draw[black!55] (-1.6,0) -- (1.8,0);
        
            \fill[massfill]
            (-1.32,0)
              .. controls (-1.10,0.05) and (-1.02,0.90) .. (-0.78,1.62)
              .. controls (-0.55,2.18) and (-0.18,2.18) .. (0.06,1.92)
              .. controls (0.20,1.76) and (0.27,1.38) .. (0.34,1.20)
              .. controls (0.42,1.02) and (0.53,0.94) .. (0.63,0.93)
              -- (0.63,0) -- cycle;
        
            \draw[bluecurve]
            (-1.32,0)
              .. controls (-1.10,0.05) and (-1.02,0.90) .. (-0.78,1.62)
              .. controls (-0.55,2.18) and (-0.18,2.18) .. (0.06,1.92)
              .. controls (0.20,1.76) and (0.27,1.38) .. (0.34,1.20)
              .. controls (0.42,1.02) and (0.53,0.94) .. (0.63,0.93)
              .. controls (0.75,0.92) and (0.86,1.20) .. (1.00,1.22)
              .. controls (1.18,1.24) and (1.36,0.84) .. (1.52,0.66)
              .. controls (1.66,0.52) and (1.73,0.16) .. (1.78,0);
        
            \draw[pickline] (0.63,0) -- (0.63,1.95);
        
            \node[blue!70!black] at (0.95,2) {$t_2$};
            \node[blue!70!black] at (-0.25,-0.3) {$\widehat{u}(\xi,t_2)$};
        \end{scope}
        
        \begin{scope}[shift={(5.2,2.8)}]
            \draw[black!55] (-1.6,0) -- (1.8,0);
        
            \fill[massfill]
            (-1.28,0)
              .. controls (-1.04,0.05) and (-0.94,0.74) .. (-0.72,1.34)
              .. controls (-0.48,1.88) and (-0.16,1.94) .. (0.06,1.72)
              .. controls (0.20,1.56) and (0.28,1.16) .. (0.38,1.08)
              .. controls (0.52,0.98) and (0.66,0.98) .. (0.78,0.98)
              -- (0.78,0) -- cycle;
        
            \draw[bluecurve]
            (-1.28,0)
              .. controls (-1.04,0.05) and (-0.94,0.74) .. (-0.72,1.34)
              .. controls (-0.48,1.88) and (-0.16,1.94) .. (0.06,1.72)
              .. controls (0.20,1.56) and (0.28,1.16) .. (0.38,1.08)
              .. controls (0.52,0.98) and (0.66,0.98) .. (0.78,0.98)
              .. controls (0.96,0.98) and (1.08,0.68) .. (1.20,0.66)
              .. controls (1.34,0.64) and (1.48,0.46) .. (1.62,0.42)
              .. controls (1.72,0.39) and (1.78,0.12) .. (1.84,0);
        
            \draw[pickline] (0.78,0) -- (0.78,1.78);
        
            \node[blue!70!black] at (0.95,2) {$t_3$};
            \node[blue!70!black] at (1.2,-0.3) {$\widehat{u}(\xi,t_3)$};
        \end{scope}
        
        \begin{scope}[shift={(0,0)}]
            \draw[black!55] (-2.6,0) -- (2.6,0);
        
            \fill[refmass]
            (-2.00,0)
              .. controls (-1.72,0.04) and (-1.52,0.78) .. (-1.15,1.18)
              .. controls (-0.86,1.45) and (-0.46,1.34) .. (-0.20,0.70)
              .. controls (-0.10,0.42) and (-0.04,0.18) .. (0,0)
              -- cycle;
        
            \draw[refcurve]
            (-2.00,0)
              .. controls (-1.72,0.04) and (-1.52,0.78) .. (-1.15,1.18)
              .. controls (-0.86,1.45) and (-0.46,1.34) .. (-0.20,0.70)
              .. controls (-0.05,0.34) and (0.00,0.16) .. (0.15,0.10)
              .. controls (0.34,0.03) and (0.54,0.56) .. (0.80,1.02)
              .. controls (1.05,1.44) and (1.42,1.20) .. (1.68,0.84)
              .. controls (1.90,0.54) and (2.02,0.10) .. (2.18,0);
        
            \fill (0,0) circle (1.5pt);
            \node[blue!70!black] at (0,-0.30) {$\xi$};
            \node[black!60] at (-0.02,-0.65) {\small reference $r$};
        \end{scope}
        
        \draw[flowarrow] (-2.7,2.95) -- (-2.1,2.95);
        \draw[flowarrow] ( 2.5,2.95) -- ( 3.1,2.95);
        
        \draw[maparrow] (-4.66,2.8) -- (0,0);
        \draw[maparrow] ( 0.63,2.8) -- (0,0);
        \draw[maparrow] ( 5.98,2.8) -- (0,0);
        
        \end{tikzpicture}
        \vspace{-0.1in}
        \caption{For each $t$, we define $\widehat{u}(\xi,t)$ as the (unique) point such that the identity \eqref{eq: mass preserv} is satisfied (shared areas). The blue curves in the top panel represent $u(\cdot,t)$ at different times $t=t_1,t_2,t_3$ and the gray curve on the bottom represents a fixed reference density $r$.}
        \label{fig: mass_preservation}
\end{figure}

    This $\widehat u(\cdot,t)$ coincides with what in optimal transport is called the optimal transport map from the reference density $r$ to the snapshot $u(\cdot,t)$, and some extra definitions and properties of this transform are given in the following Section \ref{sec: CDT}. Although a priori there are no guarantees that with this transform the manifold $\mathcal{M}_{[t_0,T]}$ will be represented in transformed space $\widehat{\mathcal{M}}_{[t_0,T]}$ as a low-dimensional linear space, we will show that in several cases it does. In particular, for the advection equation \eqref{eq: pure advection}, the curves $t\mapsto \widehat{u}(\xi,t)$ coincide with the characteristics. This says that we can interpret these trajectories as a generalization of characteristic curves, with the advantage that they can be defined even after shocks and rarefactions, and for non-hyperbolic conservative equations. Moreover, this transform has a closed form expression, see Definition \ref{def: CDT}, that is  easy to compute from data and has a closed form inverse transform, given by \eqref{eq: inv cdt}, that allows us to reconstruct the original manifold after applying reduction techniques in transform space. The trade-off for this generality is that the guaranteed two-dimensional representation that we obtained with characteristics is lost, but it still exhibits better compression than working directly in native space. The implicit formulation \eqref{eq: mass preserv} is also not fortuitous, it is in fact an implicit form to characterize the manifold of densities with the same total mass as the reference (modulo some other properties),  and the CDT is a global coordinate chart in that space. Coincidentally, the solution manifold $\mathcal{M}_{[t_0,T]}$ of conservative equations lives entirely inside the space of densities with the same total mass.

\subsection{The Cumulative Distribution Transform (CDT)}
\label{sec: CDT}

    We briefly recall the one-dimensional cumulative distribution transform(CDT) and its basic properties. The CDT represents a nonnegative density of fixed mass by the monotone transport map from a prescribed reference density to that density. This transport-based representation is particularly useful for advective dynamics: translations become affine shifts in transform space, and one-dimensional Wasserstein distances become weighted $L^2$ distances. These properties will be used below to explain why transport-dominated solution manifolds can become more compressible after applying the CDT.
    
    The CDT was introduced in \cite{park2018cumulative} (see \cite{antil2026discrete} for a notion of discrete CDT) and has since been used in signal processing, imaging, classification, and data analysis; see, for example, \cite{rubaiyat2020parametric,rubaiyat2022nearest,rubaiyat2024end,rubaiyat2024data}. Many of these methods are implemented in PyTransKit \cite{pytranskit}. In this work, we use the measure and transport-map viewpoint developed in \cite{akram22sign,ivan2024data}.

    \begin{definition}
        Let $I\subseteq \mathbb{R}$ be a (possibly unbounded) interval, and $u:I\to[0,\infty)$ a probability density. We define its cumulative distribution function (CDF), denoted $F_u:I\to [0,1]$ by
        \begin{equation}
            F_u(x) := \int_{I\cap(-\infty,x]} u(s) \, ds, 
        \end{equation}
        and its quantile $F_u^\dagger:(0,1)\to I$ by 
        \begin{equation}\label{eq: quantile}
            F_u^\dagger(p) := \inf\{x\in I: \, F_u(x)\geq p\}.
        \end{equation}
    \end{definition}
    
    \begin{remark}\label{rmk:quantile}
        The quantile function \eqref{eq: quantile} is a generalization of the inverse of $F_u$, since they coincide whenever the inverse exists. Nevertheless, it is not the only possible extension of the definition of the inverse. For example, one could replace the condition inside the infimum in \eqref{eq: quantile} with $F_u(x)>p$, and this could still be considered a generalized inverse (see, for example, \cite{akram22sign}). Moreover, the quantile function is defined on the open interval $(0,1)$ because, in some cases, defining it at the endpoints of the interval would result in values of $\pm\infty$, even if $u$ is compactly supported. Since these issues could lead to certain artifacts and to different possible definitions of the Cumulative Distribution Transform introduced below, we sometimes work with its definition in an almost-everywhere sense. In practice, however, we choose the particular representative given in terms of \eqref{eq: quantile} and work mostly with bounded domains.   
    \end{remark}

    \begin{definition}[CDT]\label{def: CDT}
        Let $K,I\subseteq \mathbb{R}$ (possibly unbounded) intervals. Let 
        $r:K\to[0,\infty)$ be a fixed probability density that we will call the reference. For any probability density $u:I\to[0,\infty)$ (with finite second moment if $I$ is unbounded) we define its Cumulative Distribution Transform (CDT) with respect to $r$, denoted $\widehat{u}:K\to I$, by 
        \begin{equation}\label{eq: cdt}
               \widehat{u}(\xi)\;:=\;F_u^{\dagger}(F_r(\xi)), 
        \end{equation}
        where $F_u^\dagger$ is the quantile of $u$ and $F_r$ is the CDF of $r$ and formula \eqref{eq: cdt} is interpreted in an $r(\xi)d\xi$-almost-everywhere sense. 
    \end{definition}
    
    \begin{remark}
        The CDT can also be implicitly defined, in an $r(\xi)d\xi$-almost-everywhere sense, as the function $\widehat{u}$ satisfying $F_u(\widehat{u}(\xi))=F_r(\xi)$ or, equivalently, in integral form,
        \begin{equation*}
        \int_{I\cap(-\infty,\widehat{u}(\xi)]}u(s)ds
        =
        \int_{K\cap(-\infty,\xi]}r(s)ds.
        \end{equation*}
    \end{remark}

    In most of this work, we assume that the reference $r$ and all the signals $u$ involved are defined on a common interval $I$, without distinguishing between the domain of the reference and those of the signals. Without loss of generality, one may also assume nonnegative signals of fixed total mass, not necessarily unit mass. 

    In general, given two probability densities $u_1,u_2$, we say that a map $T:\RR\to\RR$ pushes $u_1$ to $u_2$, denoted $T_\#(u_1(\xi)d\xi) = u_2(x)dx$, if 
    \begin{equation}\label{eq: def pushforward}
        \int_\RR \varphi(x) \, u_2(x) \, dx=\int_{\RR}\varphi(T(x)) \, u_1(x) \, dx   
    \end{equation}
    for every bounded Borel test function $\varphi$. Thus, from a measure theory perspective,  $\widehat u$ from Definition \ref{def: CDT} is determined by being a monotone map with the property
    \begin{equation}\label{eq: cdt-in}
        \widehat u_{\#}\big(r(\xi)\,d\xi\big)=u(x)\,dx.
    \end{equation}
    Moreover, this pushforward identity \eqref{eq: cdt-in} naturally describes the inverse transform. Indeed, given $\widehat u$, the measure with probability density function $u$ is recovered by $\widehat u_{\#}\big(r(\xi)\,d\xi\big)$. In particular, if $\widehat{u}$ is of class $C^1$ and strictly monotone with $\widehat{u}'>0$, then \eqref{eq: cdt-in} reads as
    $u(\widehat{u}(\xi))\,\widehat{u}'(\xi)=r(\xi)$, and the \emph{inverse CDT} (iCDT) admits a closed-form formula:
    \begin{equation}\label{eq: inv cdt}
        u(x) = (\widehat{u} \, ^\dagger)'(x) \, r(\widehat{u} \, ^\dagger(x)).
    \end{equation}
    If, in addition, the reference is taken as the constant function $r\equiv 1$ defined on $(0,1)$, the CDT can be written in \emph{quantile coordinates} $p=F_r(\xi)=\xi$ as
    $\widehat u(p)=F_u^\dagger(p)$. Thus, the CDT and iCDT simplify to: 
    \begin{equation*}
        u
        \qquad \underset{CDT}{\longrightarrow} \qquad \widehat{u} := F_u^\dagger  
        \qquad \underset{iCDT}{\longrightarrow} 
        \qquad u := (\widehat{u} \, ^\dagger)'.
    \end{equation*}

    We now recall the Wasserstein distance $W_2$ defined in the space $\mathcal{P}_2(\RR)$ of probability measures with finite second moment. For simplicity, we work with the case where $u_1,u_2$ are probability density functions inducing measures in $\mathcal{P}_2(\RR)$, which by abuse of notation, we denote as $u_1,u_2\in \mathcal{P}_2(\mathbb{R})$. The $2$-Wasserstein distance is  
    \begin{align*}
        W_2(u_1,u_2)^2:&=\min\left\{\int_{\RR\times \RR}|x-y|^2 \, d\pi(x,y) : \, \pi \text{ coupling of } u_1(x)dx \text{ and } u_2(y)dy\right\}\\
        &=\min\left\{\int_{\RR}|x-T(x)|^2\, u_1(x)\, dx : \, T_\# (u_1(x)dx)= u_2(x)dx\right\},
    \end{align*}
    where a \textit{coupling} of the measures $u_1(x)dx$ and $u_2(y)dy$ is a probability measure on $\RR\times\RR$ with marginal densities $u_1,u_2$. The equality of the two minimization problems defining $W_2$ holds because we are considering density functions.  Moreover, since we are working with one dimensional densities, the \textit{optimal transport map} $T^*$ that realizes the distance is characterized as the unique non-decreasing map (defined $u_1(x)dx$-a.e.) that pushes $u_1$ to $u_2$ \cite[Theorems 2.5 and 2.9]{Santambrogio-OTAM} (see also \cite{villani2003topics}). 
    
    The CDT provides an \emph{isometric} embedding of the Wasserstein space into the weighted $L^2(dr)$. Precisely, if $u\in \mathcal P_2(\RR)$, using the substitution $x=\widehat{u}(\xi)$  (i.e., using \eqref{eq: cdt-in} and the definition of pushforward \eqref{eq: def pushforward}), we have 
    \begin{equation}
        \|\widehat{u}\|^2_{L^2(dr)}:=\int|\widehat{u}(\xi)|^2 \, r(\xi) \, d\xi=\int x^2 \, u(x) \, dx<\infty, \label{eq: 2nd moment 1}
    \end{equation}
    that is,  $\widehat{u}\in L^2(dr)$ (the $L^2$-space weighted by the reference $r$). Now, to show the isometry, we use the substitution $x=\widehat{u}_1(\xi)$ and compute,
    \begin{equation}\label{eq:cdt-isometry}
        \|\widehat u_1-\widehat u_2\|_{L^2(dr)}^2 \;=\int |x-\widehat{u}_2(\widehat{u}_1^\dagger(x))|^2 \, u_1(x) \, dx =\; W_2(u_1,u_2)^2 \qquad \forall u_1,u_2\in \mathcal{P}_2(\RR),
    \end{equation}
    where for the last equality one uses that $\widehat{u}_2 \circ \widehat{u}_1^\dagger$ is a monotone map that pushes $u_1$ to $u_2$, thus it is the optimal transport map. 

    Transport theory describes translations and more general rearrangements in terms of mass displacement, with the Wasserstein distance providing a natural metric for quantifying the cost of such deformations. This geometry is physically meaningful but intrinsically nonlinear and therefore more difficult to analyze computationally. By contrast, $L^2$-methods operate in a linear Hilbert space but compare states pointwise and may fail to capture transport-dominated variability efficiently. The CDT bridges these perspectives by isometrically embedding the Wasserstein geometry into the $L^2(dr)$ space, preserving the transport-aware interpretation while enabling the use of linear and Euclidean-space techniques.\\
       
    We next record two basic composition properties  that explain why domain deformations become simple operations in CDT space.

    \begin{lemma}[Composition property of the CDT] \cite[Proposition 3.1]{akram22sign}
    \label{lem:cdt-composition}
        Let $I_1,I_2\subseteq\mathbb R$ be intervals, and let
        $g:I_2\to I_1$ be a strictly increasing $C^1$ diffeomorphism with
        $g'>0$. Let $u_1,u_2$ be probability densities on $I_1,I_2$ (respectively)  related by
        \begin{equation*}
            u_2(x)=g'(x)\,u_1(g(x)).
        \end{equation*}
        Let $\widehat u_1, \widehat u_2$ denote the CDTs of $u_1, u_2$  with respect to a fixed reference density $r$. Then
        \begin{equation*}
            \widehat u_2 = g^{-1}\circ \widehat u_1 .
        \end{equation*}
    \end{lemma}
        
    \begin{proof}[Proof sketch]
        Since $g$ is increasing and maps $I_2$ onto $I_1$, the change of variables
        $z=g(s)$ gives
        \begin{equation*}
            F_{u_2}(x)
            =
            \int_{I_2\cap(-\infty,x]} g'(s)u_1(g(s))\,ds
            =
            \int_{I_1\cap(-\infty,g(x)]} u_1(z)\,dz
            =
            F_{u_1}(g(x)).
        \end{equation*}
        Therefore, for $p\in(0,1)$, $F_{u_2}^\dagger(p)
            =
            g^{-1}\big(F_{u_1}^\dagger(p)\big)$.
        Composing with $F_r$ gives
        $$
            \widehat u_2
            =
            F_{u_2}^\dagger\circ F_r
            =
            g^{-1}\circ F_{u_1}^\dagger\circ F_r
            =
            g^{-1}\circ \widehat u_1.
        $$
    \end{proof}

    Two important special cases are translation and dilation.
    
    \begin{corollary}[Translation and dilation equivariance] \cite[Corollaries 2.3 and 2.4]{akram22sign}\label{cor:cdt-translation-dilation}
        Let $r$ be a fixed reference density, and let all CDTs below be computed with respect to this same reference. For a probability density $u\in\mathcal{P}_2(\mathbb R)$, define, for $\tau\in\mathbb R$ and $\beta>0$,
        \begin{equation*}
            (\mathcal T_\tau u)(x):=u(x-\tau),
            \qquad
            (\mathcal D_\beta u)(x):=\beta\,u(\beta x).
        \end{equation*}
        Then
        \begin{equation*}
            \widehat{\mathcal T_\tau u}(\xi)
            =
            \widehat u(\xi)+\tau,
            \qquad
            \widehat{\mathcal D_\beta u}(\xi)
            =
            \tfrac{1}{\beta} \, \widehat u(\xi).
        \end{equation*}
    \end{corollary}

    In this work, for a time-dependent solution of a PDE $u(\cdot,t)$ with nonnegative equal-mass snapshots, we apply the CDT at each time:
    \begin{equation} 
        \widehat u(\xi,t)
        :=
        F_{u(\cdot,t)}^\dagger(F_r(\xi)).
    \end{equation}
    In particular, it satisfies
    \begin{equation}\label{eq: push-forward id}
        \widehat u(\cdot,t)_{\#}\big(r(\xi)\,d\xi\big)
        =
        u(x,t)\,dx.
    \end{equation}
    Thus the transformed solution may be viewed either as a curve $t\mapsto \widehat u(\cdot,t)$ in the Hilbert space $L^2(dr)$, or as a family of cumulative-mass trajectories indexed by the reference coordinate $\xi$. This dual interpretation is central to the analysis below.

\subsection{Two Analytical Examples of Dimensionality Reduction Using CDT}\label{sec: examples x2}

\paragraph{Advection-Diffusion with Dirac Initial Datum.}\label{sec: example}
    As a foundational example,  let us consider the advection-diffusion equation on $\RR$ with an impulse initial condition 
    \begin{align}\label{eq: conv diff}
        \begin{cases}
            u_t = D u_{xx} - A u_x\\
            u(x,0) = \delta_0(x).
        \end{cases}
    \end{align}
    For $t > 0$, the solution is a Gaussian density
    \begin{equation*}
        u(x,t) = \frac{1}{\sqrt{4\pi D t}}e^{\frac{-(x-At)^2}{4Dt}},
    \end{equation*}
    which can be obtained using the Fourier transform. In this form, it is not immediately apparent that the solution manifold $\mathcal M_{[t_0,T]}:=\{u(\cdot,t):\, t\in[t_0,T]\}$ can be recast as a low-dimensional linear space. In native coordinates, this family combines translation and diffusion-driven spreading, and its low-dimensional linear structure is not apparent. In CDT coordinates, however, the structure becomes explicit. With the aid of the CDT we can transform it into a two-dimensional linear space. Consider the functions
    \begin{equation*}
        h(x) := \tfrac{1}{\sqrt{4\pi D}}e^{-x^2/(4D)}, \qquad g(x;t) =  \frac{x-At}{\sqrt{t}}, \quad \text{with } g'(x;t) = 1/\sqrt{t}.
    \end{equation*}
    Then, the solution of \eqref{eq: conv diff} is $u(x,t) = g'(x;t) h(g(x;t))$. Taking $t>0$ as a fixed parameter, transforming $u$ using the CDT with respect to $x$, and applying the composition property of Lemma \ref{lem:cdt-composition}, we get 
    \begin{equation*}
        \widehat{u}(\xi,t) = \sqrt{t} \, \widehat{h}(\xi) + At \,  \mathbf{1}(\xi),
    \end{equation*}
    where $\mathbf{1}$ denotes the constant function one. This expression tells us that the evolution of the full system can be reduced to a two-dimensional system. That is, the CDT transform provides a representation of the solution in a new space generated by the basis $\mathcal B:=\{\widehat{h},\mathbf{1}\}$; and the functions $\sqrt{t}$ and $At$ are the coordinates at each time $t$ that describe the evolution of our state in the subspace generated by $\mathcal B$.

\paragraph{The Linear Transport Equation.}\label{sec:linear-transport-eq}
    To isolate the effect of transport on linear approximation spaces, we consider the linear advection equation \eqref{eq: pure advection} on the real line with constant velocity $A\in\mathbb R$ and initial condition $u(x,0)=\init(x)$. We assume that $\init$ is a nonnegative probability density with finite second moment and that $\init\in L^2(\mathbb R)$. Thus the native solution manifold can be viewed in $L^2(\mathbb R)$, while the CDT is well defined in $L^2(dr)$.
    
    The solution is given explicitly by translation of the initial datum,
    \begin{equation*}
        u(x,t)=\init(x-At), \qquad t\in[0,T].
    \end{equation*}
    On the whole real line no boundary condition is imposed; equivalently, for sufficiently decaying data, the flux contribution at infinity vanishes. Then the associated native solution manifold is
    \begin{equation*}
        \mathcal{M}_{[0,T]}
        :=
        \{\,\init(\cdot-At): \, t\in[0,T] \, \}
        \subset L^2(\mathbb R).
    \end{equation*}
    
    Let $\widehat{\init}$ denote the CDT of $\init$ with respect to the fixed
    reference density $r$. By the translation equivariance property of the CDT
    (Corollary~\ref{cor:cdt-translation-dilation}), we have
    \begin{equation*}
        \widehat{\init(\cdot-At)}
        =
        \widehat{\init}+At,
        \qquad t\in[0,T].
    \end{equation*}
    Consequently, the transformed solution manifold is the affine line segment
    \begin{equation}
    \label{eq:transf-manifold-transport}
        \widehat{\mathcal M}_{[0,T]}
        =
        \big\{\,\widehat{\init}+t\,A\mathbf 1: \, t\in[0,T]\,\big\}
        \subset
        \widehat{\init}+\operatorname{span}\{\mathbf 1\}
        \subset
        \operatorname{span}\{\widehat{\init},\mathbf 1\}.
    \end{equation}
    
    Thus, while the family of translated copies of a nontrivial profile is not,
    in general, contained in any finite-dimensional linear subspace of
    $L^2(\mathbb R)$, its CDT image is contained in an at most two-dimensional
    linear subspace of $L^2(dr)$. 

\subsection{Beyond Exact Dimensionality Reduction: Kolmogorov Width and SVD}
\label{sec:kolmogorov}

The examples in the previous section are special cases in which the
transformed solution manifold is contained in a low-dimensional linear space
whose dimension can be computed exactly. This exact finite-dimensional
structure does not hold in general. We therefore need tools to quantify how
well a solution manifold can be approximated by low-dimensional linear
spaces, and practical methods for constructing such approximating spaces
from snapshot data.

Following standard practice in reduced order modeling, we combine two
complementary diagnostics. Kolmogorov $n$-widths quantify the best achievable
worst-case error of $n$-dimensional linear approximation
\cite{pinkus2012n,algoritmy,unger2019kolmogorov}, while POD/SVD singular
value decay and snapshot reconstruction errors provide computable measures
of data compressibility and approximation accuracy
\cite{berkooz1993proper,benner2015survey,rowley2017model}.

\paragraph{Kolmogorov Width.}
The Kolmogorov $n$-width of a set measures the smallest possible worst-case
error when approximating that set by an $n$-dimensional linear subspace.
Although this notion is defined through an optimal linear subspace, such a
space is generally not available explicitly and may not be constructible in
practice. Nevertheless, upper bounds on Kolmogorov widths provide useful
information about whether a solution manifold is amenable to low-dimensional
linear approximation.

\begin{definition}[Kolmogorov $n$-width]
Let $\mathcal H$ be a Hilbert space. For a subset
$\mathcal S\subset \mathcal H$, the Kolmogorov $n$-width of $\mathcal S$ in
$\mathcal H$ is defined by
\begin{equation*}
    d_n(\mathcal S;\mathcal H)
    :=
    \inf_{\substack{V\subset \mathcal H\\ \dim V\le n}}
    \sup_{u\in\mathcal S}
    \underbrace{\inf_{h\in V}\|u-h\|_{\mathcal H}}_{\mathrm{dist}(u,V)} .
\end{equation*}
In particular, if $\mathcal S$ is contained in a linear subspace of
dimension $d$, then
\begin{equation*}
    d_n(\mathcal S;\mathcal H)=0
    \qquad\text{for all } n\ge d.
\end{equation*}
\end{definition}

\begin{remark}
With this definition, the dimensionality reduction of the solution manifold
of the linear transport equation \eqref{eq: pure advection} in CDT space can
be written as
\begin{equation}
\label{eq:kol-2-adv}
    d_n\big(\widehat{\mathcal M}_{[0,T]};L^2(dr)\big)=0
    \qquad\text{for every } n\ge2,
\end{equation}
where $\widehat{\mathcal M}_{[0,T]}$ is given by
\eqref{eq:transf-manifold-transport}.
\end{remark}

Thus, Kolmogorov widths quantify how well a set can be approximated by
low-dimensional linear spaces, which is a central question in reduced order
modeling. In this work, we focus on upper bounds for Kolmogorov widths of
solution manifolds after transformation by the CDT. The transformed solution
manifold, denoted by $\widehat{\mathcal M}_{[t_0,T]}$, can be viewed as the
image of a curve
\begin{equation*}
    t\mapsto \widehat u(\cdot,t)
\end{equation*}
in the Hilbert space $L^2(dr)$. We will use the following elementary
Hilbert-valued curve estimates.

\begin{lemma}[Kolmogorov-width bounds for Hilbert-valued curves]
\label{lem:curve-width}
Let $\mathcal H$ be a Hilbert space and let
$\gamma:[0,T]\to\mathcal H$ be a curve. Let $0\le t_0<T$ and set
$L:=T-t_0$.
\begin{enumerate}
\item  \label{item:1-curve} If $\gamma\in C^1([0,T];\mathcal H)$ and
\begin{equation*}
    M_1:=\sup_{t\in[t_0,T]}\|\gamma'(t)\|_{\mathcal H},
\end{equation*}
then
\begin{equation}
\label{eq:bound-M1}
    d_n\big(\gamma([t_0,T]);\mathcal H\big)
    \le
    \frac{L M_1}{2n}
    \qquad\text{for every } n\ge1.
\end{equation}

\item  \label{item:2-curve} If $\gamma\in C^2([0,T];\mathcal H)$ and
\begin{equation*}
    M_2:=\sup_{t\in[t_0,T]}\|\gamma''(t)\|_{\mathcal H},
\end{equation*}
then
\begin{equation}
\label{eq:bound-M2}
    d_n\big(\gamma([t_0,T]);\mathcal H\big)
    \le
    \frac{L^2 M_2}{8(n-1)^2}
    \qquad\text{for every } n\ge2.
\end{equation}
\end{enumerate}
\end{lemma}

For completeness, a proof is provided in Appendix~\ref{app: vector-val}.
The preceding lemma should be viewed as a low-regularity instance of a
standard interpolation principle. If the Hilbert-valued curve $\gamma$ has
higher temporal regularity, then improved Kolmogorov-width estimates can be
obtained using higher-order piecewise polynomial interpolation instead of
piecewise linear interpolation.

\paragraph{SVD.}
In general, it is not possible to explicitly find the best linear subspace
realizing the infimum in the definition of the Kolmogorov width. This is
especially true in numerical settings, where the solution manifold is only
observed through a finite collection of snapshots. A standard practical
proxy is therefore the singular value decomposition of a snapshot matrix.

Let $\mathcal S$ be a family of functions sampled on a spatial grid
$\{x_i\}_{i=1}^M$ and a temporal grid $\{t_j\}_{j=1}^N$. We denote the
sampled states by $s(\cdot,t_j)$ and form the snapshot matrix
$\mathbb S\in\mathbb R^{M\times N}$ by
\begin{equation*}
    \mathbb S_{ij}:=s(x_i,t_j).
\end{equation*}
Thus the $j$-th column $s_j$ of $\mathbb S$ represents the discrete state
$s(\cdot,t_j)$. The POD approximation is obtained from the singular value
decomposition
$
    \mathbb S = W\Sigma V^T.
$
The columns of $W$ are the POD modes associated with the chosen
representation. Retaining the first $k$ modes gives the rank-$k$
approximation
$
    \mathbb S_k
    =
    W_k\Sigma_k V_k^T
    =
    W_k W_k^T \mathbb S,
$
in the Euclidean inner product. Equivalently, each snapshot is approximated
by
\begin{equation*}
    s(\cdot,t_j)
    \approx
    \sum_{\ell=1}^k a_\ell(t_j)w_\ell,
    \qquad
    a_\ell(t_j)=w_\ell^T s_j.
\end{equation*}

When comparing with continuous Hilbert-space norms, the SVD should be
interpreted with the quadrature weights associated with the relevant norm.
For native-space $L^2$ errors, this corresponds to weighting by the spatial
quadrature weights. In CDT space, the natural norm is $L^2(dr)$; equivalently,
when using quantile coordinates $p\in(0,1)$ on a uniform grid, the standard
Euclidean SVD approximates the desired $L^2$ norm. In weighted coordinates,
one may either compute a weighted SVD or absorb the square roots of the
quadrature weights into the snapshot matrix before applying the usual SVD.

The decay of the singular values provides a finite-dimensional diagnostic
for the linear compressibility of the sampled manifold. Rapid singular-value
decay indicates that few modes are sufficient to approximate the snapshots
well in the discrete norm associated with the snapshot matrix, whereas slow
decay signals poor linear compressibility. Throughout the paper, a
\emph{mode} refers to a left singular vector of the relevant snapshot matrix.

\begin{remark}
The Kolmogorov-width estimates and the numerical study of singular-value
decay are complementary. Kolmogorov widths quantify the best possible linear
approximation of a solution manifold in a prescribed Hilbert-space norm,
whereas singular values measure the low-rank structure observed in a finite
snapshot matrix and in the corresponding discrete norm. Since the optimal
subspace realizing, or nearly realizing, the Kolmogorov width is typically
not known explicitly, POD provides a concrete data-driven procedure for
constructing an approximating subspace from available snapshots.
Appendix~\ref{sec: kolmogorov width and SVD} briefly recalls the connection
between Kolmogorov widths and singular-value decay.
\end{remark}

\subsection{The CDT-POD Method and the Structure of Numerical Examples}
\label{sec:CDT-POD-method}

\paragraph{CDT-POD method.}
We now describe the numerical procedure used to evaluate dimension reduction
in CDT space. We first compute a high-fidelity solution of the PDE at time
instances $t_1,\ldots,t_N$. The corresponding native snapshot matrix is
\begin{equation}
\label{eq:U}
    \mathbb U
    =
    \begin{bmatrix}
        u(\cdot,t_1) & u(\cdot,t_2) & \cdots & u(\cdot,t_N)
    \end{bmatrix}.
\end{equation}
We then apply the CDT to each snapshot, using a fixed reference density $r$,
and form the transformed snapshot matrix
\begin{equation}
\label{eq:hat-U}
    \widehat{\mathbb U}
    =
    \begin{bmatrix}
        \widehat{u}(\cdot,t_1) &
        \widehat{u}(\cdot,t_2) &
        \cdots &
        \widehat{u}(\cdot,t_N)
    \end{bmatrix}.
\end{equation}

We compute low-rank approximations in both coordinate systems. In native
variables, $\mathbb U$ is approximated by its rank-$k$ POD/SVD truncation,
denoted by $\mathbb U_k$. The native POD approximation at time $t_j$ is the
$j$-th column of $\mathbb U_k$, denoted by
$\widetilde u_E(\cdot,t_j)$. In transformed coordinates,
$\widehat{\mathbb U}$ is approximated by its rank-$k$ POD/SVD truncation,
denoted by $\widehat{\mathbb U}_k$. The $j$-th column of
$\widehat{\mathbb U}_k$ gives an approximate CDT map, denoted by
$\widetilde{\widehat{u}}(\cdot,t_j)$. Applying the inverse CDT column-wise,
whenever the reduced map is a valid monotone transport map, gives the
corresponding physical-space approximation
$\widetilde u_T(\cdot,t_j)$.

Thus, $\widetilde u_E$ is obtained by reducing directly in Eulerian
coordinates, whereas $\widetilde u_T$ is obtained by transforming the data,
reducing in CDT coordinates, and mapping the reduced representation back to
physical space. This procedure is summarized in
Algorithm~\ref{alg:cdt-pod} and illustrated in Figure~\ref{fig: CDT_pipeline}.

        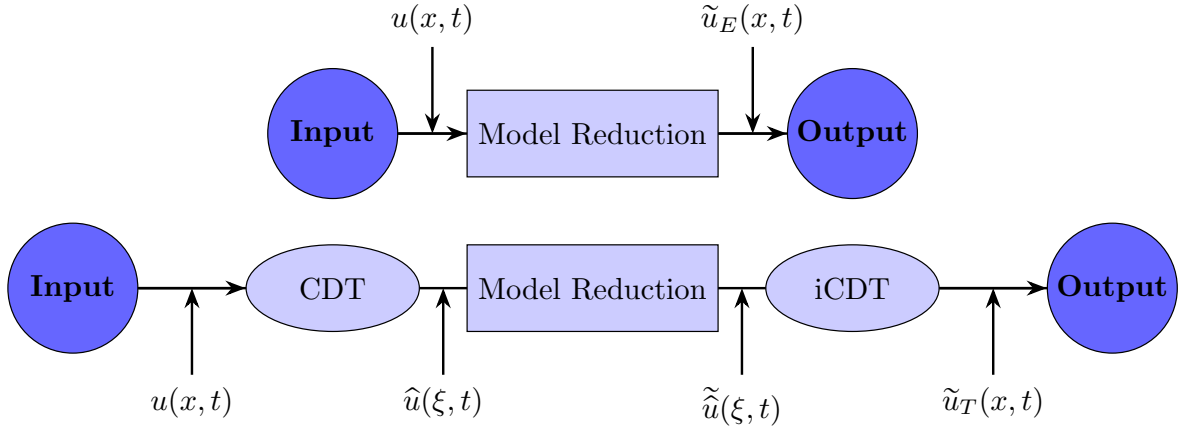
\begin{figure}[!ht] 
        \centering
        \resizebox{\textwidth}{!}{
            \begin{tikzpicture}
            \node (top){
            \begin{tikzpicture}[>=Stealth, node distance=1.5cm, every node/.style={font=\small}]
    
                \draw[->, thick] (3,0) -- (9,0);
                
                \node (Input) at (3,0) [circle, draw, fill=blue!60, minimum size=1.5cm, inner sep=0pt, anchor = center] {\textbf{Input}};
                
                \node[draw, rectangle, fill=blue!20, minimum width=2cm, minimum height=1cm, anchor=center] (B1) at (6,0) {Model Reduction};
    
                \node (Output) at (9,0) [circle, draw, fill=blue!60, minimum size=1.5cm, inner sep=0pt, anchor = center] {\textbf{Output}};

                \draw[->, thick] (Input.east) -- (B1.west);
                
                \coordinate (I0) at ($(Input.east)!0.5!(B1.west)$);
                \coordinate (I3) at ($(B1.east)!0.5!(Output.west)$);
                
                \draw[<-, thick] (I0) -- ++(0,1) node[above] {$u(x,t)$};
                \draw[<-, thick] (I3) -- ++(0,1) node[above] {$\widetilde{u}_E(x,t)$};
    
                \draw[->,thick] (B1.east) -- (Output.west);
    
            \end{tikzpicture}};
    
            \node[below=0pt of top, anchor=north] (bottom){            
            \begin{tikzpicture}[>=Stealth, node distance=1.5cm, every node/.style={font=\small}]
    
                \draw[->, thick] (0,0) -- (12,0);
                
                \node (Input) at (0,0) [circle, draw, fill=blue!60, minimum size=1.5cm, inner sep=0pt, anchor = center] {\textbf{Input}};
                
                \node[draw, ellipse, fill=blue!20, minimum width=2cm, minimum height=1cm, anchor=center] (B1) at (3,0) {CDT};
                \node[draw, rectangle, fill=blue!20, minimum width=2cm, minimum height=1cm, anchor=center] (B2) at (6,0) {Model Reduction};
                \node[draw, ellipse, fill=blue!20, minimum width=2cm, minimum height=1cm, anchor=center] (B3) at (9,0) {iCDT};
    
                \node (Output) at (12,0) [circle, draw, fill=blue!60, minimum size=1.5cm, inner sep=0pt, anchor = center] {\textbf{Output}};

                \draw[->, thick] (Input.east) -- (B1.west);
                
                \coordinate (I0) at ($(Input.east)!0.5!(B1.west)$);
                \coordinate (I1) at ($(B1.east)!0.5!(B2.west)$);
                \coordinate (I2) at ($(B2.east)!0.5!(B3.west)$);
                \coordinate (I3) at ($(B3.east)!0.5!(Output.west)$);
                
                \draw[<-, thick] (I0) -- ++(0,-1) node[below] {$u(x,t)$};
                \draw[<-, thick] (I1) -- ++(0,-1) node[below] {$\widehat{u}(\xi,t)$};
                \draw[<-, thick] (I2) -- ++(0,-1) node[below] {$\widetilde{\widehat{u}}(\xi,t)$};
                \draw[<-, thick] (I3) -- ++(0,-1) node[below] {$\widetilde{u}_T(x,t)$};
    
                \draw[->,thick] (B3.east) -- (Output.west);
    
            \end{tikzpicture}};
            \end{tikzpicture}}
            \vspace{-0.4in}
            \caption{Comparison between the standard model reduction (top) and the CDT-based model reduction (bottom) pipelines.}
            \label{fig: CDT_pipeline}
        \end{figure}

The comparison between these two pipelines must be interpreted with care.
The SVD truncation of $\mathbb U$ gives the best rank-$k$ approximation of
the native snapshot matrix in the discrete norm used to compute the SVD.
In contrast, the CDT-based approximation is not constructed to minimize the
native-space reconstruction error directly. Instead, the low-rank
approximation is computed in CDT coordinates, where the CDT isometry
\eqref{eq:cdt-isometry} identifies the $L^2(dr)$ error with a
one-dimensional Wasserstein error between snapshots. Singular-value decay of
$\widehat{\mathbb U}$ should therefore be interpreted as evidence of
low-dimensional structure in transform space, not as a direct native-space
error estimate after applying the inverse transform.

When assessing reconstruction accuracy in physical variables, we compare
the high-fidelity solution with both the native POD reconstruction
$\widetilde u_E$ and the CDT-based reconstruction $\widetilde u_T$ after
the inverse transform has been applied. This puts both approximations in the
same native coordinate system. The numerical examples below therefore report
both global reconstruction errors in physical variables and time-dependent,
or column-wise, errors along the evolution.

There is one further asymmetry between the two approaches. The CDT is
defined for nonnegative densities with fixed mass, whereas a linear POD
approximation in native variables may introduce negative values or change
the total mass. Therefore, it is not always meaningful to transform the
native POD approximation $\widetilde u_E$ and measure its error in CDT space.
On the other hand, a linear truncation in CDT space need not preserve the
monotonicity required of a valid transport map. Consequently, the inverse
CDT step is valid directly only when the reduced CDT approximation remains
monotone; otherwise one must apply a monotone rearrangement, a projection, or
a constrained reduced approximation. This limitation is important in
interpreting the numerical results and motivates monotonicity-aware CDT-ROM
methods.

\begin{algorithm}[H]
\caption{CDT-POD}
\label{alg:cdt-pod}
\begin{algorithmic}[1]
\Require Snapshots $\{u(\cdot,t_j)\}_{j=1}^N$ with common unit mass
\Require Reference density $r$ with the same mass
\Require Number of modes $k$

\State Form the native snapshot matrix
\[
    \mathbb U
    =
    \big[u(\cdot,t_1)\ \cdots\ u(\cdot,t_N)\big].
\]

\For{$j=1,\ldots,N$}
    \State Compute the CDT snapshot
    \[
        \widehat u(\cdot,t_j)
        =
        F_{u(\cdot,t_j)}^\dagger\circ F_r .
    \]
\EndFor

\State Form the transformed snapshot matrix
\[
    \widehat{\mathbb U}
    =
    \big[\widehat u(\cdot,t_1)\ \cdots\ \widehat u(\cdot,t_N)\big].
\]

\State Compute a rank-$k$ POD approximation in CDT space,
\[
    \widehat{\mathbb U}_k
    =
    W_k W_k^\top \widehat{\mathbb U},
\]
or use the corresponding weighted projection when quadrature weights are
included.

\For{$j=1,\ldots,N$}
    \State Let
    \[
        \widetilde{\widehat u}(\cdot,t_j)
        :=
        (\widehat{\mathbb U}_k)_{:,j}.
    \]
    \State If $\widetilde{\widehat u}(\cdot,t_j)$ is monotone nondecreasing,
    reconstruct
    \[
        \widetilde u_T(\cdot,t_j)
        =
        \operatorname{iCDT}
        \big(\widetilde{\widehat u}(\cdot,t_j)\big).
    \]
    \State Otherwise, apply a monotone projection or rearrangement before
    applying the inverse CDT.
\EndFor

\State \Return $\{\widetilde u_T(\cdot,t_j)\}_{j=1}^N$
\end{algorithmic}
\end{algorithm}

\paragraph{Structure of numerical examples.}\label{par: numerical_examples_structure}Taking all these into account, in Sections~\ref{sec: numerics hyperbolic case} and~\ref{sec: num adv-diff}, we use a common set of plots to compare the native and CDT-based reduced approximations for different equations. The first group of figures, Figures~\ref{fig: burgers_three_rows}, \ref{fig: traffic_three_rows}, \ref{fig: buc_lev_three_rows}, and~\ref{fig: adv_diff_three_rows}, focuses on the solution of one problem each with fixed parameters and initial condition. Each figure is organized as follows:
\begin{enumerate}[(a)]
    \item
        We compare the high-fidelity solution with the native and CDT-based rank-$k$ reconstructions at selected times. These plots show where the approximation error occurs in physical space and help distinguish errors related to displacement, diffusion, shocks, rarefactions, or loss of amplitude.

    \item
        We report time-dependent, or columnwise, reconstruction errors for fixed ranks, typically $k=2$ and $k=5$. These curves show how the approximation quality changes along the evolution and help identify whether the error increases near dynamically important events, such as shock formation, rarefaction regions, or strong transport.

    \item
        We compare the singular value decay of $\mathbb U$ and $\widehat{\mathbb U}$, together with the global native-space reconstruction error as a function of the number of retained modes. The singular value decay indicates how low-dimensional the snapshot matrix appears in each coordinate system, while the reconstruction error measures the quality of the recovered solution in physical variables. These plots should be read together: faster singular value decay in $\widehat{\mathbb U}$ indicates better low-rank structure in CDT space, but it does not necessarily imply the same decay for the physical-space reconstruction error, since the CDT-based approximation must be mapped back through the inverse transform. For this reason, we also report the transform/inverse-transform error floor, obtained by applying the CDT followed by its inverse without rank truncation. This gives a reference level for the smallest native-space error that can be expected from the CDT-ROM pipeline at the given discretization.
\end{enumerate}

The second group of figures, Figures~\ref{fig: burgers_more_transport}, \ref{fig: traffic_more_transport}, \ref{fig: buc_lev_more_transport}, and~\ref{fig: adv_diff_more_transport}, repeats the comparison across different parameters and initial conditions. These experiments are designed to vary the strength of the transport effects, for example by increasing characteristic speeds or considering more displacement-dominated regimes. For each regime, we compare:
\begin{enumerate}[(a)]
    \item
        The high-fidelity solution with the native and CDT-based rank-$k$ reconstructions at representative times.

    \item
        The corresponding singular value decays and rank-dependent native-space reconstruction errors.
\end{enumerate}
This allows us to assess how the relative performance of the native and CDT-based approximations changes as the dynamics become more transport-dominated.

\section{Conservative Equations in CDT Space}\label{sec: theory for conservative eq in general}

Let $I\subseteq\mathbb R$ be an interval (possibly unbounded) with endpoints $a,b$ with $-\infty \leq a<b \leq \infty$. Throughout this section, we assume that $u$ is a nonnegative mass-preserving solution of \eqref{eq: IBVP} on $I\times[0,T]$. Thus, for each $t\in[0,T]$, the snapshot $u(\cdot,t)$ has the same total mass as the reference density $r$. When $I$ is unbounded, we also assume finite second moment whenever the CDT is viewed as an element of $L^2(dr)$.

We begin in Section~\ref{sec: lagrangian form} with formal derivations of
semi-Lagrangian and Lagrangian representations of scalar conservative PDEs
in CDT coordinates. Here, ``Lagrangian'' refers to equations written only in
terms of the mass label, or transformed coordinate, and the CDT map. We then
develop the technical preliminaries in Sections~\ref{sec: reg theory}
and~\ref{sec: extra reg theory}, which will be used in
Sections~\ref{sec: kolm hyp} and~\ref{sec: kol adr} to establish bounds on
the Kolmogorov $n$-width of
$\widehat{\mathcal M}_{[t_0,T]}$ in $L^2(dr)$.

   \subsection{Formal Derivation of Lagrangian and Semi-Lagrangian Forms}\label{sec: lagrangian form}

The goal of this section is to derive formal evolution equations for the
CDT map $\widehat{u}$. We first work with a general reference density $r$.
For this formal calculation, assume that $u$, $r$, and $\widehat{u}$ are
smooth enough, that $r>0$, and that $\widehat{u}_\xi>0$.

By definition of the CDT,
\begin{equation}
\label{eq:hat-equivalent-r-general}
    \int_a^{\widehat{u}(\xi,t)} u(x,t)\,dx
    =
    F_r(\xi).
\end{equation}
Differentiating \eqref{eq:hat-equivalent-r-general} with respect to $\xi$
and rearranging gives
\begin{equation}
    \label{eq:u-in-cdt-space-r-general-ratio}
    u(\widehat{u}(\xi,t),t)=\frac{r(\xi)}{\widehat{u}_{\xi}(\xi,t)}.
\end{equation}
Differentiating once more with respect to $\xi$ yields
\begin{equation}
\label{eq:ux-in-cdt-space-r-general}
    u_x(\widehat{u}(\xi,t),t)
    =
    \frac{
        r'(\xi)\widehat{u}_\xi(\xi,t)
        -
        r(\xi)\widehat{u}_{\xi\xi}(\xi,t)
    }{
        \widehat{u}_\xi(\xi,t)^3
    }.
\end{equation}
Higher Eulerian derivatives of $u$ evaluated at $x=\widehat u(\xi,t)$ can be
obtained similarly by repeated differentiation with respect to $\xi$.

Next, differentiating \eqref{eq:hat-equivalent-r-general} with respect to
$t$ gives
\begin{equation}
\label{eq:integral-dyn}
    \widehat u_t(\xi,t)
    =
    -\frac{1}{u(\widehat u(\xi,t),t)}
    \int_a^{\widehat u(\xi,t)}u_t(x,t)\,dx .
\end{equation}
Since the equation has conservative form
\begin{equation*}
    u_t+\partial_x f[u]=0,
\end{equation*}
we obtain
\begin{equation*}
    \int_a^{\widehat u(\xi,t)}u_t(x,t)\,dx
    =
    - f[u](\widehat u(\xi,t),t)+f[u](a,t).
\end{equation*}
Assuming zero flux at the lower endpoint, or the corresponding decay
condition if $a = -\infty$, this becomes
\begin{equation}
\label{eq:semilag}
    \widehat u_t(\xi,t)
    =
    \frac{f[u](\widehat u(\xi,t),t)}
    {u(\widehat u(\xi,t),t)}.
\end{equation}
Using \eqref{eq:u-in-cdt-space-r-general-ratio} and
\eqref{eq:ux-in-cdt-space-r-general}, this can be written formally as
\begin{equation}
\label{eq:ut-hat-general-r}
\begin{aligned}
    \widehat u_t(\xi,t)
    &=
    \frac{\widehat u_\xi(\xi,t)}{r(\xi)}
    f\left(
        \frac{r(\xi)}{\widehat u_\xi(\xi,t)},
        \frac{
            r'(\xi)\widehat u_\xi(\xi,t)
            -
            r(\xi)\widehat u_{\xi\xi}(\xi,t)
        }{
            \widehat u_\xi(\xi,t)^3
        },
        \ldots
    \right).
\end{aligned}
\end{equation}

Equivalently, whenever the flux can be written in the form
\begin{equation*}
    f[u]=u\,v,
\end{equation*}
with $v$ interpreted as an Eulerian velocity field, the CDT map satisfies
the semi-Lagrangian equation
\begin{equation}
\label{eq:ut-hat-v-form}
    \widehat u_t(\xi,t)
    =
    v(\widehat u(\xi,t),t).
\end{equation}
Thus, for each fixed mass label $\xi$, the point
$x=\widehat u(\xi,t)$ moves with the Eulerian velocity $v(x,t)$. This is
semi-Lagrangian because the unknown $\widehat u$ is Lagrangian, while the
velocity field is still described in Eulerian coordinates and evaluated
along the CDT map.

The fully transformed formulas are simplest in \textit{quantile coordinates $p$}, which means taking the reference $r:(0,1)\to\mathbb{R}$ to be the constant function $r\equiv 1$, and defining $p:=F_r(\xi)= \xi$. In these coordinates, the CDT of a signal is just its quantile function, and formula \eqref{eq:ut-hat-general-r} simplifies to 
\begin{equation}\label{eq:ut-hat-quantile-coordinates}
    \widehat{u}_t(p,t) =\widehat{u}_p(p,t)f(1/\widehat{u}_p(p,t), -\widehat{u}_{pp}(p,t)/\widehat{u}^3_p(p,t),\dots),
\end{equation}
since $r\equiv 1$ and $r'\equiv 0$. Notice that formula \eqref{eq:ut-hat-v-form} is independent of the chosen reference so it remains unchanged in quantile coordinates.

Using \eqref{eq:ut-hat-general-r}, \eqref{eq:ut-hat-v-form} and \eqref{eq:ut-hat-quantile-coordinates}, Table~\ref{tab:equations-in-CDT-space} presents the equations in Table~\ref{tab:scalar_conservation_laws} written in Eulerian,  \textit{Hybrid} Eulerian-CDT, and  fully CDT coordinates. 

\begin{table}[H]
\centering
\caption{Conservative equations in Eulerian, Hybrid, and CDT coordinates.
The final column is written in quantile coordinates using \eqref{eq:ut-hat-quantile-coordinates}.}
\vspace{-0.1in}
\begin{adjustbox}{max width=\linewidth}
\begin{tabular}{llll}
\toprule
\textbf{Equation}
&
\textbf{Native}
&
\textbf{Hybrid}
&
\textbf{CDT} \\
&
\textbf{(Eulerian)}
&
\textbf{(Semi-Lagrangian)}
&
\textbf{(Lagrangian)} \\
\midrule

Advection--Diffusion
&
$u_t+(Au-Du_x)_x=0$
&
$\widehat u_t(\xi,t)
 =
 \left.\dfrac{Au-Du_x}{u}\right|_{(\widehat u(\xi,t),t)}$
&
$\widehat u_t
 =
 A+D\dfrac{\widehat u_{pp}}{\widehat u_p^2}$
\\[0.9em]

Burgers
&
$u_t+\left(\dfrac{u^2}{2}\right)_x=0$
&
$\widehat u_t(\xi,t)
 =
 \left.\dfrac{u}{2}\right|_{(\widehat u(\xi,t),t)}$
&
$\widehat u_t
 =
 \dfrac{1}{2\widehat u_p}$
\\[0.9em]

Traffic
&
$u_t+\bigl(u(1-u)\bigr)_x=0$
&
$\widehat u_t(\xi,t)
 =
 \left.(1-u)\right|_{(\widehat u(\xi,t),t)}$
&
$\widehat u_t
 =
 \dfrac{\widehat u_p-1}{\widehat u_p}$
\\[0.9em]

Buckley--Leverett
&
$u_t+
\left(\dfrac{u^2}{u^2+(1-u)^2}\right)_x=0$
&
$\widehat u_t(\xi,t)
 =
 \left.\dfrac{u}{u^2+(1-u)^2}\right|_{(\widehat u(\xi,t),t)}$
&
$\widehat u_t
 =
 \dfrac{\widehat u_p}{1+(\widehat u_p-1)^2}$
\\

\bottomrule
\end{tabular}
\end{adjustbox}
\label{tab:equations-in-CDT-space}
\end{table}

Differentiating \eqref{eq:ut-hat-v-form} in time, we formally obtain
\begin{equation*}
\label{eq:utt-hat-a}
    \widehat u_{tt}(\xi,t)
    =
    \mathfrak a(\widehat u(\xi,t),t),
    \qquad
    \mathfrak a(x,t):=v_t(x,t)+v_x(x,t)v(x,t).
\end{equation*}
Finally, let us formally write the $L^2(dr)$-norms of $\widehat{u}(\cdot,t)$, $\widehat{u}_t(\cdot,t)$, and $\widehat{u}_{tt}(\cdot,t)$. Notice that the pushforward identity \eqref{eq: push-forward id} gives,
for suitable Borel functions $\phi$,
\begin{equation*}
    \int \phi(\widehat u(\xi,t))\,r(\xi)\,d\xi
    =
    \int \phi(x)\,u(x,t)\,dx .
\end{equation*}
Taking $\phi(x)=|x|^2$, $\phi(x)=|v(x,t)|^2$, and
$\phi(x)=|\mathfrak a(x,t)|^2$, respectively, one obtains
\begin{align}
    \|\widehat u(\cdot,t)\|_{L^2(dr)}^2
    &=
    \int |\widehat u(\xi,t)|^2 \, r(\xi)\,d\xi
    =
    \int_I x^2 \, u(x,t)\,dx,
    \label{eq:norm-hat-u}
    \\
    \|\widehat u_t(\cdot,t)\|_{L^2(dr)}^2
    &=
    \int |v(\widehat u(\xi,t),t)|^2 \, r(\xi)\,d\xi
    =
    \int_I |v(x,t)|^2 \, u(x,t)\,dx,
    \label{eq:norm-v}
    \\
    \|\widehat u_{tt}(\cdot,t)\|_{L^2(dr)}^2
    &=
    \int |\mathfrak a(\widehat u(\xi,t),t)|^2 \, r(\xi)\,d\xi
    =
    \int_I |\mathfrak a(x,t)|^2 \, u(x,t)\,dx.
    \label{eq:norm-a}
\end{align}
The first identity is finite when $u(\cdot,t)$ has finite second moment. The second and third identities require additional integrability of the Eulerian velocity $v$ and acceleration $\mathfrak a$ with respect to the measure $u(x,t)\,dx$. The following sections impose hypotheses ensuring precisely these properties.

\subsection{General Regularity Theory}\label{sec: reg theory}

\begin{lemma}
\label{lem: only cont}
Let $\gamma:[0,T]\to L^2(dr)$, and assume that there exists a representative $\widetilde\gamma:I\times[0,T]\to\mathbb R$
such that $\widetilde\gamma(\cdot,t)=\gamma(t)$ in $L^2(dr)$ for every
$t\in[0,T]$, and such that, for $r$-a.e. $\xi\in I$, the map
$t\mapsto \widetilde\gamma(\xi,t)$ is continuous on $[0,T]$. Let
$h:\mathbb R\times[0,T]\to\mathbb R$ be bounded and continuous. Define $H(t)(\xi):=h(\widetilde\gamma(\xi,t),t)$. Then $H\in C([0,T];L^2(dr))$.
\end{lemma}

\begin{proof}
Fix $s\in[0,T]$. For $r$-a.e. $\xi$, the map
$t\mapsto \widetilde\gamma(\xi,t)$ is continuous. Hence
\begin{equation*}
    (\widetilde\gamma(\xi,t),t)
    \to
    (\widetilde\gamma(\xi,s),s)
    \qquad\text{as }t\to s.
\end{equation*}
Since $h$ is continuous on $\mathbb R\times[0,T]$, we obtain
\begin{equation*}
    h(\widetilde\gamma(\xi,t),t)
    \to
    h(\widetilde\gamma(\xi,s),s)
    \qquad\text{as }t\to s
\end{equation*}
for $r$-a.e. $\xi$. Moreover, if $M:=\|h\|_{L^\infty(\mathbb R\times[0,T])}$,
then $|h(\widetilde\gamma(\xi,t),t)
      -h(\widetilde\gamma(\xi,s),s)|^2
    \le 4M^2$.
Since $dr$ is a probability measure, $4M^2$ is $dr$-integrable. Therefore,
by the dominated convergence theorem,
\begin{equation*}
    \|H(t)-H(s)\|_{L^2(dr)}^2
    =
    \int_I
    |h(\widetilde\gamma(\xi,t),t)
      -h(\widetilde\gamma(\xi,s),s)|^2\,dr(\xi)
    \to 0
    \qquad\text{as }t\to s.
\end{equation*}
Thus $H\in C([0,T];L^2(dr))$.
\end{proof}

The following result is an application of the Bochner version of the fundamental theorem of calculus. 
For its proof, we refer the reader to Appendix \ref{sec: scalar conservative laws}.
We also refer to the textbook \cite{hytonen2016analysis}.

\begin{lemma}\label{lem: Bochner}
    Given the curve $\gamma \in C([0,T]; L^2(dr))$, let  $\Gamma(t) := \int_0^t \gamma(s) \, ds$, where the integral is in the Bochner sense. Then, $\Gamma \in C^1([0,T]; L^2(dr))$ and $\Gamma'(t) = \gamma(t)$ as a derivative in the $L^2(dr)$ sense. 
\end{lemma}

The following lemma is partially based on \cite[Theorem 4.4]{Santambrogio-OTAM}.

\begin{lemma}\label{lem:santambrogio-lagrangian-first-derivative}
    Let $u$ be a nonnegative unit-mass solution of $u_t+(uv)_x=0$
    on $I\times [0,T]$ with initial condition $\init$.
    Assume that $ \widehat{\init}\in L^2(dr),$
    where $\widehat{\init}$ is the CDT of $\init$ with respect to a given reference $r$.
    Assume that $v:I\times[0,T]\to\mathbb R$ satisfies the following hypotheses:
    \begin{enumerate}[(a)]
        \item $v(x,t)$  is measurable in $t$
         for each fixed $x$; and $v(\cdot,t):I\to \RR$ is Lipschitz continuous in $x$, uniformly in $t$, and uniformly bounded.
        
        \item The flow generated by $v$ leaves $I$ invariant.\footnote{This is obvious for $I=\RR$ and, for example, if $I=(a,b)$, this is guaranteed by the no-outflow condition $v(a,t)=v(b,t)=0$
        for a.e. $t\in[0,T]$.}
    \end{enumerate}
    Let $\Phi^{(t)}:I\to I$ denote the associated Carathéodory flow, namely
    \begin{equation}\label{eq:flow-definition}
    \frac{d}{dt}\Phi^{(t)}(x)=v(\Phi^{(t)}(x),t)
    \qquad\text{for a.e. }t\in(0,T),
    \qquad
    \Phi_0(x)=x.
    \end{equation}
    Then,
    
    \begin{enumerate}[(a)]
        \item $\widehat u(\cdot,t)=\Phi^{(t)}\circ\widehat \init$ $r\text{-a.e.}$;
        \item \label{item: u_hat abs cont rep} $r$-a.e. $\xi \in I$, the function $t\to \widehat u(\xi,t)$ is absolutely continuous.  
        \item $\widehat u\in W^{1,\infty}([0,T]; L^2(dr))$; 
        \item $\widehat u_t(\cdot,t) = v(\widehat u(\cdot,t),t)$, with $\partial_t$ in the $L^2(dr)$ sense for a.e. $t\in(0,T)$;
        \item $\widehat u\in C([0,T];L^2(dr))$.
    \end{enumerate}
\end{lemma}

\begin{proof}
By the assumptions, the ordinary differential equation \eqref{eq:flow-definition} admits a unique global flow $\Phi^{(t)}:I\to I$. 
Because the problem is one-dimensional and $x\mapsto v(x,t)$ is Lipschitz, the flow is order preserving. Indeed, if $x<y$, then uniqueness of the ODE prevents the trajectories $t\mapsto\Phi^{(t)}(x)$ and $t\mapsto\Phi^{(t)}(y)$ from crossing. Equivalently, Gronwall's inequality gives
\begin{equation*}\label{eq:order-preserving-flow}
\Phi^{(t)}(y)-\Phi^{(t)}(x)\ge e^{-Lt}(y-x)>0 \qquad \text{(where $L$ is the Lipschitz constant of $v$)}.
\end{equation*}
In addition, the pushforward measure $(\Phi^{(t)})_{\#}(\init(x)dx)$
solves, in the weak sense, the same continuity equation $u_t+(uv)_x=0$ as $u(\cdot,t)$,
with the same initial datum $\init$. Indeed, for $t=0$, $(\Phi_0)_\#(\init(x)dx)=\init(x)dx$ since $\Phi_0$ is the identity map, and for every
$\varphi\in C_c^\infty(\RR)$,
\begin{align*}
\frac{d}{d t}\int_{\RR}\varphi(\Phi^{(t)}(x)) \, \init(x)d x
&=
\int_{\RR}\varphi'(\Phi^{(t)}(x))v(\Phi^{(t)}(x),t)\, \init(x) d x.
\end{align*}
Then,
\begin{equation*}
    (\Phi^{(t)})_\#\left((\widehat{\init})_\#(r(\xi)d\xi)\right)=(\Phi^{(t)})_\#(\init(x)dx)=u(x,t)dx,
\end{equation*}
where the first equality holds from the CDT's definition and the second identity is a consequence of \cite[Theorem 4.4]{Santambrogio-OTAM} which guarantees that continuity equation admits
a unique solution. Finally, 
by uniqueness of the monotone transport map in one dimension, it follows that
\begin{equation}\label{eq:uhat-flow-proof}
\widehat u(\cdot,t)=\Phi^{(t)}\circ\widehat{\init}
\qquad r\text{-a.e.}
\end{equation}
since, by definition of the CDT,  $\widehat u(\cdot,t)$ is monotonically increasing and satisfies \eqref{eq: push-forward id}.

Now, using \eqref{eq:flow-definition}, the above identity \eqref{eq:uhat-flow-proof} implies that, for  $r\text{-a.e.}$ $\xi$,
\begin{equation*}
    \partial_t\widehat{u}(\xi,t)=\frac{d}{dt}\Phi^{(t)}(\widehat{\init}(\xi))=v(\Phi^{(t)}(\widehat{\init}(\xi)),t)=v(\widehat{u}(\xi,t),t)  \quad \text{with } \widehat{u}(\xi,0)=\Phi_0(\widehat{\init}(\xi))=\widehat{\init}(\xi).
\end{equation*}
Since, for $r$-a.e. $\xi$, the curve
$t\mapsto \widehat u(\xi,t)=\Phi^{(t)}(\widehat\init(\xi))$
is absolutely continuous and satisfies
\begin{equation*}
\widehat u_t(\xi,t)
=
v(\widehat u(\xi,t),t)
\quad\text{for a.e. }t
\end{equation*}
(solutions of Carathéodory ODEs are absolutely continuous in time \cite{CoddingtonLevinson1955}), then 
the fundamental theorem of calculus for absolutely continuous functions gives,
for $0\le s\le t\le T$,
\begin{equation}\label{eq:pointwise-integral-curve}
\widehat u(\xi,t)-\widehat u(\xi,s)
=
\int_s^t v(\widehat u(\xi,\tau),\tau)\,d\tau
\qquad r\text{-a.e. }\xi.
\end{equation}
Let $\mathbf{v}(t):=v(\widehat u(\cdot,t),t)$.
By hypothesis,
let
$
M:=\operatorname*{ess\,sup}_{t\in[0,T]}\|v(\cdot,t)\|_{L^\infty(I)}<\infty$.
Then
\begin{equation*}
\|\mathbf{v}(t)\|_{L^2(dr)}
\le M\|\mathbf{1}\|_{L^2(dr)}
=M,
\end{equation*}
because $r$ is a probability density. Hence
$\mathbf{v}\in L^\infty([0,T];L^2(dr))\subset L^1([0,T];L^2(dr))$.
Equation \eqref{eq:pointwise-integral-curve} may therefore be interpreted as a Bochner integral identity in $L^2(dr)$:
\begin{equation}\label{eq:bochner-integral-first-proof}
\widehat u(\cdot,t)-\widehat u(\cdot,s)
=
\int_s^t \mathbf{v}(\tau)\,d\tau.
\end{equation}
In particular,
\begin{equation}\label{eq: boch}
    \widehat u(\cdot,t)
=
\widehat \init+\int_0^t \mathbf{v}(\tau)\,d\tau
\qquad\text{in }L^2(dr) 
\end{equation}
and so $\widehat u\in W^{1,\infty}([0,T];L^2(dr))\subset W^{1,1}([0,T];L^2(dr))$ with $\widehat u_t(\cdot,t)
=
\mathbf{v}(t)
=
v(\widehat u(\cdot,t),t)$
in $L^2(dr)$ for a.e. $t\in(0,T)$, where weak and almost everywhere derivatives coincide (see \cite[Proposition 2.5.9]{hytonen2016analysis}).
Moreover, from \eqref{eq:bochner-integral-first-proof}, for all $t,s\in [0,T]$,
\begin{equation*}
    \|\widehat u(\cdot,t)-\widehat u(\cdot,s)\|_{L^2(dr)}
=
\left\|\int_s^t \mathbf{v}(\tau)\,d\tau\right\|_{L^2(dr)}\leq \int_s^t \|\mathbf{v}(\tau)\|_{L^2(dr)}d\tau\leq M|t-s|,
\end{equation*}
i.e., the $L^2(dr)$-valued curve $t\to \widehat u(\cdot,t)$ is Lipschitz continuous, 
in particular $\widehat u\in C([0,T];L^2(dr))$.

\end{proof}

\begin{corollary}
    \label{lem: hyp v cont}
    If, in addition to the hypothesis from Lemma \ref{lem:santambrogio-lagrangian-first-derivative}, we have that $v\in C(I\times[0,T])$, then the map
\begin{equation*}
{\mathbf{v}}:[0,T]\to L^2(dr),
\qquad
\mathbf{v}(t):=v(\widehat u(\cdot,t),t),
\end{equation*}
belongs to $C([0,T];L^2(dr))$. Consequently, $\widehat u\in C^1([0,T];L^2(dr))$ and $\widehat u_t(\cdot,t)
=
v(\widehat u(\cdot,t),t)$ in $L^2(dr)$
for every $t\in[0,T]$. Moreover, $\partial_t \widehat u(\cdot, t)$ admits the continuous representative  $\mathbf{v}(t)(\xi) := v(\widehat u(\xi,t),t)$ $r$-a.e. $\xi \in I$.
\end{corollary}

\begin{proof}
The first part is a consequence of Lemma \ref{lem: only cont} with $\gamma =\widehat u$ and $h=v$, using Lemma \ref{lem:santambrogio-lagrangian-first-derivative} part \eqref{item: u_hat abs cont rep}.
For the last part, from the Bochner integral identity \eqref{eq: boch} in the proof of Lemma \ref{lem:santambrogio-lagrangian-first-derivative}
and the fact that $\mathbf{v}\in C([0,T];L^2(dr))$, the fundamental theorem of
calculus for Bochner integrals (Lemma \ref{lem: Bochner}) implies that
$\widehat u\in C^1([0,T];L^2(dr))$ 
and
$\widehat u_t(\cdot,t)
=
\mathbf{v}(t)
=
v(\widehat u(\cdot,t),t)$ in $L^2(dr)$
for every $t\in[0,T]$. 
The moreover follows from hypothesis and  Lemma \ref{lem:santambrogio-lagrangian-first-derivative} part \eqref{item: u_hat abs cont rep}.
\end{proof}

\begin{lemma}\label{lem:hyp-a-cont}
Assume the hypotheses of Lemma~\ref{lem:santambrogio-lagrangian-first-derivative}.
Assume, in addition, that $v\in C^1(I\times[0,T])$, and define
\begin{equation*}
\mathfrak a(x,t):=v_t(x,t)+v_x(x,t)v(x,t), \quad \text{and} \quad  \mathbf{a}:[0,T]\to L^2(dr), \quad
\mathbf{a}:=\mathfrak a(\widehat u(\cdot,t),t).
\end{equation*}
Suppose that $a$ is bounded and continuous.
Then $\mathbf{a}\in C([0,T];L^2(dr))$.
Moreover, if $\mathbf{v}(t):=v(\widehat u(\cdot,t),t)$,
then $\mathbf{v}\in C^1([0,T];L^2(dr))$ and
\begin{equation*}
\mathbf{v}'(t)=\mathbf{a}(t)
=
\bigl(v_t+v_xv\bigr)(\widehat u(\cdot,t),t)
\qquad\text{in }L^2(dr).
\end{equation*}
Consequently, $\widehat u\in C^2([0,T];L^2(dr))$ and $\widehat u_{tt}(\cdot,t)
=
\bigl(v_t+v_xv\bigr)(\widehat u(\cdot,t),t)$
in $L^2(dr)$. Moreover, $\partial_{tt} \widehat u(\cdot, t)$ admits the continuous representative  $\mathbf{a}(t)(\xi) := \bigl(v_t+v_xv\bigr)(\widehat u(\xi,t),t)$ $r$-a.e. $\xi \in I$.
\end{lemma}

\begin{proof}
    By Lemma \ref{lem: hyp v cont}, we know that $\widehat u\in C^1([0,T];L^2(dr))$.
    Also, as an application of Lemma \ref{lem: only cont} using $\gamma=\widehat u$ and $h=a$, we have that $\mathbf{a}\in C([0,T];L^2(dr))$ with representative $\mathbf{a}(t)(\xi) = \mathfrak a(\widehat u(\xi,t),t)$. Indeed, by Lemma~\ref{lem:santambrogio-lagrangian-first-derivative}, we may
    choose the representative
    \begin{equation*}
    \widehat u(\xi,t)=\Phi^{(t)}(\widehat\init(\xi))
    \qquad r\text{-a.e. }\xi,
    \end{equation*}
    where $\Phi^{(t)}$ is the flow generated by $v$. Hence, for $r$-a.e. $\xi$, the map $t\mapsto \widehat u(\xi,t)$
    is absolutely continuous on $[0,T]$. Moreover, by Corollary \ref{lem: hyp v cont}, it satisfies
    \begin{equation*}
    \widehat u_t(\xi,t)
    =
    v(\widehat u(\xi,t),t)
    \qquad \text{ for a.e. }  t.
    \end{equation*}
    We will use this fact to show that $\mathbf{a}$ is the actual derivative of $\mathbf{v}$. 
    Since $v\in C^1(I\times[0,T])$, the chain rule for absolutely
    continuous curves gives, for $r$-a.e. $\xi$,
    \begin{align*}
    \frac{d}{dt}v(\widehat u(\xi,t),t)
    &=
    v_x(\widehat u(\xi,t),t)\widehat u_t(\xi,t)
    +
    v_t(\widehat u(\xi,t),t)=
    \bigl(v_t+v_xv\bigr)(\widehat u(\xi,t),t)=
    \mathbf{a}(t)(\xi),
    \end{align*}
    which is continuous as a function of $t\in [0,T]$. Therefore, for $0\le s\le t\le T$,
    \begin{equation*}
    \mathbf{v}(t)(\xi)-\mathbf{v}(s)(\xi)
    =
    \int_s^t \mathbf{a}(\tau)(\xi)\,d\tau
    \qquad\text{for }r\text{-a.e. }\xi.
    \end{equation*}
    Since $\mathbf{a}\in C([0,T];L^2(dr))$, this identity holds as a Bochner integral
    identity in $L^2(dr)$:
    \begin{equation*}
    \mathbf{v}(t)-\mathbf{v}(s)
    =
    \int_s^t \mathbf{a}(\tau)\,d\tau.
    \end{equation*}
    The fundamental theorem of calculus for Bochner integrals (Lemma \ref{lem: Bochner})  implies that $\mathbf{v}\in C^1([0,T];L^2(dr))$, with 
    $\mathbf{v}'(t)=\mathbf{a}(t)$ in $L^2(dr)$.
    Finally, from Lemma~\ref{lem:santambrogio-lagrangian-first-derivative},
    \begin{equation*}
    \widehat u(\cdot,t)
    =
    \widehat \init+\int_0^t \mathbf{v}(\tau)\,d\tau
    \qquad\text{in }L^2(dr).
    \end{equation*}
    Since $\mathbf{v}\in C^1([0,T];L^2(dr))$, it follows that $\widehat u\in C^2([0,T];L^2(dr))$
    and
    \begin{equation*}
    \widehat u_{tt}(\cdot,t)
    =
    \mathbf{v}'(t)
    =
    \mathbf{a}(t)
    =
    \bigl(v_t+v_xv\bigr)(\widehat u(\cdot,t),t)
    \qquad\text{in }L^2(dr).
    \end{equation*}
\end{proof}

\subsection{Extra Regularity Results}\label{sec: extra reg theory}

    In some cases, the CDT transform $\widehat u $ of $u$ can be completely characterized by \eqref{eq:hat-equivalent-r-general}. When that is possible, the regularity of $\widehat u$ can be proved with the aid of simple tools like the Implicit Function Theorem. In this section we take this route: we first introduce some notation, then show that $\widehat u$ is well-defined by  \eqref{eq:hat-equivalent-r-general} and finally explore the regularity of the maps $t\mapsto \widehat u(\xi,t)$ for each fixed $\xi\in \RR$.

    \begin{definition}
    Let $u:\RR\times [t_0,T]\to \RR$, with $u\geq0$, and $u(\cdot,t) \in L^1(\RR)$ for every $t$. We define its cumulative function with respect to $x$, denoted $F_u$, by
    \begin{equation}\label{eq: Fu as a func of 2 variables}
        F_u(x,t) = F_{u(\cdot,t)}(x) = \int_{-\infty}^x u(y,t)\, dy.
    \end{equation}
    Notice that if $u$ does not depend on $t$ the definition coincides with the usual cumulative function. 
\end{definition}

    \begin{lemma} \label{lem: u_hat_well_defined}
        Let $r\in\mathcal{P}(\RR)$, with $F_r(\xi)\in (0,1)$ for every $\xi\in \RR$, and $u\in C(\RR\times[t_0,T])$ such that 
        \begin{equation*}
            u>0, \qquad \text{ and }  \qquad \int u(x,t) \, dx = 1 \text{ for every }t. 
        \end{equation*}
        Then, $\widehat{u}:\RR\times[t_0,T]\to\RR$ defined by 
        \begin{equation}\label{eq: implicit def of hat u}
            F_u(\widehat{u}(\xi,t),t) = F_r(\xi),
        \end{equation}
        is well-defined, that is, for every $(\xi,t)$ the value $\widehat{u}(\xi,t)$ is uniquely determined. Moreover, it coincides with the $\widehat u$ given in Definition \ref{def: CDT}.
    \end{lemma}
    
    \begin{proof}
        Given a fixed $t$, the function $x \mapsto F_u(x,t)$ is differentiable with derivative $u(x,t)>0$. Thus, it is strictly increasing and continuous. Since $u(\cdot,t)$ is a probability density, it satisfies $\lim_{x\to-\infty}F_u(x,t)=0$ and $\lim_{x\to+\infty}F_u(x,t)=1$. Then  $x \mapsto F_u(x,t)$ is a bijection from $\RR$ to $(0,1)$ and has an inverse $(F_{u(\cdot,t)})^{-1}$ such that $(F_{u(\cdot,t)})^{-1}(F_u(x,t)) = x$. Now, for fixed $(\xi,t)$, $F_r(\xi)\in (0,1)$. Thus, applying $(F_{u(\cdot,t)})^{-1}$ on both sides of \eqref{eq: implicit def of hat u} we get that there exists a solution $\widehat u (\xi,t)$ solving the equation. Moreover, if there were $x,y$ such that $F_u(x,t)=F_r(\xi) =F_u(y,t)$, then applying $F_{u(\cdot,t)}^{-1}$ to each term we get $x=y$, providing the uniqueness.  
    \end{proof}

\begin{lemma} \label{lem: differentiability of F_u(x,t)}
    Let $u \in C(\RR\times [t_0,T])$ such that 
    \begin{enumerate}
        \item For every $t\in[t_0,T]$, $u(\cdot,t) \in L^1(\RR)$,
        \item For every $x$, $u(x,\cdot) \in C^1([t_0,T])$,
        \item \label{item: hyp u_t L^1 cont}
        $t\mapsto u_t(\cdot,t) \in C([t_0,T];L^1(\RR))$. 
    \end{enumerate}
    Then, $F_u\in C^1(\mathbb R\times[t_0,T])$, with 
    \begin{equation*}
        \partial_xF_u(x,t) = u(x,t), \qquad \qquad \partial_t F_u(x,t) = \int_{-\infty}^x u_t(y,t)\,dy.
    \end{equation*}
\end{lemma}

\begin{proof}
    The formula for the partial derivative with respect to $x$ follows from the Fundamental Theorem of Calculus. For the partial derivative with respect to $t$, take $\tau \in [t_0, T]$, and $h$ small enough. Then, compute 
    \[
    \begin{aligned}
        \frac{F_u(x,\tau+h)-F_u(x,\tau)}{h}
        &=
        \frac{1}{h}
        \int_{-\infty}^x
        \bigl(u(y,\tau+h)-u(y,\tau)\bigr)\,dy
        \\
        &=
        \frac{1}{h}
        \int_{-\infty}^x
        \int_{\tau}^{\tau+h}
        u_t(y,s)\,ds\,dy
        \\
        &=
        \frac{1}{h}
        \int_{\tau}^{\tau+h}
        \int_{-\infty}^x
        u_t(y,s)\,dy\,ds,
    \end{aligned}
    \label{eq: incremental_quotient}
    \]
    where the use of Fubini's theorem is justified because $u_t(\cdot,s)\in L^1(\mathbb R)$, uniformly for $s$ in a compact subinterval of $[t_0,T]$. 
    Now, since $t\mapsto u_t(\cdot,t)$ is continuous as an $L^1(\mathbb R)$-valued map, we obtain
    \[
    \begin{aligned}
        &\left| \frac{1}{h}
        \int_{\tau}^{\tau+h}
        \int_{-\infty}^x
        u_t(y,s)\,dy\,ds - \frac{1}{h}
        \int_{\tau}^{\tau+h}
        \int_{-\infty}^x
        u_t(y,\tau)\,dy\,ds \right| \\
        &\qquad= 
        \frac{1}{h}
        \int_{\tau}^{\tau+h}
        \left|\int_{-\infty}^x
        u_t(y,s)\,dy - 
        \int_{-\infty}^x
        u_t(y,\tau)\,dy\right|\,ds\\
        &\qquad\leq
        \frac{1}{h}
        \int_{\tau}^{\tau+h}
        \int_{-\infty}^x
        |u_t(y,s)-u_t(y,\tau)|\,dy
        \,ds \\
        &\qquad\leq
         \frac{1}{h}
        \int_{\tau}^{\tau+h}
         \|u_t(\cdot,s)-u_t(\cdot,\tau)\|_{L^1(\mathbb R)}
        \,ds \to 0 \qquad \text{when } h\to 0.\\
    \end{aligned}
    \]
    Therefore, the limit for $h\to 0$ in the preceding difference quotient,  gives us
    \begin{equation*}
        \partial_t F_u(x,\tau)
        =
        \int_{-\infty}^x u_t(y,\tau)\,dy.
    \end{equation*}
To show that $\partial_tF_u$ is continuous, take $(x_n,t_n)\to (x_0,\tau)$ and compute
\begin{align*}
    \left|
    \partial_tF_u(x_n,t_n)-\partial_tF_u(x_0,\tau)
    \right|
    &\le
    \left|
    \int_{-\infty}^{x_n}
    \bigl(u_t(y,t_n)-u_t(y,\tau)\bigr)\,dy
    \right| +
    \left|
    \int_{x_0}^{x_n} u_t(y,\tau)\,dy
    \right|\\
     &\le
     \|u_t(\cdot,t_n)-u_t(\cdot,\tau)\|_{L^1(\mathbb R)} + \int_{\min\{x_0,x_n\}}^{\max\{x_0,x_n\}} |u_t(y,\tau)|\,dy \to 0,
\end{align*}
where convergence of the first term follows from Item \ref{item: hyp u_t L^1 cont}, and the second term tends to $0$ due to the absolute continuity of the Lebesgue integral, since $u_t(\cdot,\tau)\in L^1(\RR)$.
\end{proof}

\begin{lemma}\label{lem: F_u deriv wrt t}
    Let $u \in C(\RR\times [t_0,T])$ satisfying the hypotheses of Lemmas \ref{lem: u_hat_well_defined} and \ref{lem: differentiability of F_u(x,t)}. Then, $\widehat{u}:\RR\times [t_0,T] \to \RR$ is well-defined and for every $\xi$, and the function $t\mapsto \widehat u(\xi,t)$ is differentiable for $t\in[t_0,T]$, with 
    \begin{equation*}
        \widehat{u}_t(\xi,t) = -\frac{\int_{-\infty}^x u_t(y,t)\, dy}{u(x,t)}\Bigg|_{(x,t) = (\widehat{u}(\xi,t),t)}
    \end{equation*}
\end{lemma}

\begin{proof}
    By Lemma \ref{lem: u_hat_well_defined}, $\widehat u$ is well-defined on
    $\RR\times [t_0,T]$ and is the unique function satisfying $F_u(\widehat u(\xi,t),t)=F_r(\xi)$. Moreover, by Lemma \ref{lem: differentiability of F_u(x,t)}, $F_u$ is continuously differentiable and $\partial_x F_u(x,t)=u(x,t)>0$.
    Therefore, for each fixed $\xi$, the implicit function theorem applied to $F_u(x,t)=F_r(\xi)$ shows that $t\mapsto x(t)$ (that satisfies $x(t) = \widehat u(\xi,t)$ by uniqueness) is differentiable. Differentiating the identity
    $F_u(\widehat u(\xi,t),t)=F_r(\xi)$ with respect to $t$ gives
    \begin{equation*}
        \widehat u_t(\xi,t)
        =
        -\frac{\partial_t F_u(x,t)}{\partial_x F_u(x,t)}
        \Bigg|_{(x,t)=(\widehat u(\xi,t),t)}
        =
        -\frac{\int_{-\infty}^x u_t(y,t)\,dy}{u(x,t)}
        \Bigg|_{(x,t)=(\widehat u(\xi,t),t)}.
    \end{equation*}

\end{proof}

\section{Hyperbolic Equations}\label{sec: hyperbolic case}

In this section, we study hyperbolic equations in CDT space. Our proofs involve a more rigorous verification of how their dynamics can be represented through a semi-Lagrangian form as in \eqref{eq:ut-hat-v-form}. This structure allows us to estimate the Kolmogorov $n$-widths of the associated solution manifolds $\widehat{\mathcal M}_{[0,T]}$ in CDT space. 
Essentially, we control $\|\widehat u(\cdot,t)\|_{L^2(dr)}$, $\|\widehat u_t(\cdot,t)\|_{L^2(dr)}$ and $\|\widehat u_{tt}(\cdot,t)\|_{L^2(dr)}$ to provide bounds for $d_n(\widehat{\mathcal M}_{[0,T]},L^2(dr))$.
The same strategy can also be applied in the native space, providing comparable estimates that help assess the improvement obtained by using the CDT.

\subsection{Kolmogorov Width Bounds in CDT Space and Comparisons with Native Space}\label{sec: kolm hyp}

Let $T>0$ and let $I\subseteq \RR$ be a closed interval or the whole real line. 
Consider the conservation law
\begin{equation}\label{eq: hyp}
    \begin{cases}
        &u_t + (f(u))_x = 0 \qquad \text{for } (x,t) \in I\times[0,T]\\
        & u(\cdot,0)=\init,   
    \end{cases}
\end{equation}
with zero outward flux. That is, $f(u(x,t)) = 0$ for $x $ on the boundary of $I$. We assume that $\init: I\to \RR$ is a probability density function with finite second moment if $I=\RR$.  
In particular, $\init\geq 0$.

Through this section let $r$ be a probability density supported on $I$ with finite second moment if $I=\RR$. Let $\widehat \init\in L^2(dr)$ be the monotone map pushing $r$ forward to the initial density $\init$.

Let us denote the CDT solution space by
    $$\widehat{\mathcal M}_{[0,T]}:=\{\widehat u(\cdot,t):t\in[0,T]\}$$

 In order to properly state our first result (Theorem \ref{thm:hyperbolic-cdt-lipschitz}), we need the following definition of  entropy solutions by S.N. Kružkov \cite{Kruzhkov1969Continuity,kruvzkov1970first}. Additionally, in Theorem \ref{thm:hyperbolic-cdt-lipschitz} we restrict to the Cauchy problem on the whole line
$I=\mathbb R$. The bounded-interval case requires entropy boundary
conditions and is not treated here.
 
\begin{definition}[Entropy solution]\label{def:entropy-solution-scalar}
Let $T>0$, let $f\in C^1_{\mathrm{loc}}(\mathbb R)$, and let
$\init\in L^\infty(\mathbb R)$. A function $u\in L^\infty(\mathbb R\times(0,T))
\cap C([0,T];L^1_{\mathrm{loc}}(\mathbb R))$
is called an entropy solution of \eqref{eq: hyp} on $\mathbb R\times[0,T]$ if $u(\cdot,0)=\init \in L^1_{\mathrm{loc}}(\mathbb R)$, and for every
$k\in\mathbb R$ and every nonnegative test function
$\varphi\in C_c^\infty(\mathbb R\times[0,T))$,
\begin{equation}\label{eq: entropy cond}
   \int_0^T\int_{\mathbb R}
\left[
|u-k|\varphi_t
+
\operatorname{sgn}(u-k)\big(f(u)-f(k)\big)\varphi_x
\right]\,dx\,dt
+
\int_{\mathbb R}|\init-k|\varphi(x,0)\,dx
\ge 0, 
\end{equation}
where $\operatorname{sgn}$ denotes the sign function with the convention $\operatorname{sgn}(0)=0$.
\end{definition}

Notice that since $u$ and $\init$ in Definition \ref{def:entropy-solution-scalar} are essentially bounded, choosing $k$ larger
than their essential ranges and then choosing $k$ smaller than their
essential ranges in \eqref{eq: entropy cond} gives the usual distributional
formulation
\begin{equation*}
\int_0^T\int_{\mathbb R}
\big(u\phi_t+f(u)\phi_x\big)\,dx\,dt
+
\int_{\mathbb R}\init(x)\phi(x,0)\,dx
=0
\end{equation*}
for every $\phi\in C_c^\infty(\mathbb R\times[0,T))$. Thus every entropy
solution in the sense of Definition \ref{def:entropy-solution-scalar} is also
a distributional weak solution.
Informally speaking, an entropy solution is the unique physically meaningful weak solution selected by entropy inequalities (\eqref{eq: entropy cond} in Definition \ref{def:entropy-solution-scalar}).

\begin{theorem}\label{thm:hyperbolic-cdt-lipschitz}
Let $T>0$. Assume $f\in C^1_{\mathrm{loc}}(\mathbb R)$, $f(0)=0$, and let
\begin{equation*}
\init\in L^1(\mathbb R)\cap L^\infty(\mathbb R),
\qquad
\init\ge 0,
\qquad
\int_{\mathbb R}\init(x)\,dx=1,
\qquad
\int_{\mathbb R}x^2\init(x)\,dx<\infty .
\end{equation*}
Set $M:=\|\init\|_{L^\infty}$ and
$B_M:=\sup_{0<z\le M}\left|\tfrac{f(z)}{z}\right|$.
Then there exists a unique entropy solution $u$, in the sense of
Definition~\ref{def:entropy-solution-scalar}, of  \eqref{eq: hyp}
on $\RR\times[0,T]$ with initial condition $\init$. 
Furthermore, $u(\cdot,t)$ is a probability density in $\mathcal P_2(\mathbb R)$ for every
$t\in[0,T]$, and the CDT trajectory satisfies
\begin{equation}\label{eq: nodes}
    \|\widehat u(\cdot,t)-\widehat u(\cdot,s)\|_{L^2(dr)}
=
W_2(u(\cdot,t),u(\cdot,s))
\le B_M|t-s|.
\end{equation}
Moreover,
\begin{equation*}
d_n\big(\widehat{\mathcal M}_{[0,T]};L^2(dr)\big)
\le
\frac{B_MT}{2n}.
\end{equation*}
\end{theorem}

\begin{proof}
By Kružkov's well-posedness theorem \cite{kruvzkov1970first} (applied to scalar conservation laws), there exists a unique
entropy solution
$u\in C([0,T];L^1(\mathbb R))
\cap L^\infty(\mathbb R\times(0,T))$ for \eqref{eq: hyp}.
By the comparison principle for Kružkov entropy solutions \cite{kruvzkov1970first}, applied to the
constant entropy solutions $0$ and $M:=\|\init\|_{L^\infty}$, and using
$0\le \init\le M$, we obtain
$0\le u(x,t)\le M$ for a.e. $(x,t)\in\mathbb R\times(0,T)$. (Indeed, since the solution remains in the interval $[0,M]$, one may first extend
$f|_{[0,M]}$ to a globally Lipschitz $C^1$ flux without changing the
 solution.)
Moreover, by the $L^1$-contraction property \cite{kruvzkov1970first}, applied with the zero solution,
$
\|u(\cdot,t)\|_{L^1(\mathbb R)}
\le
\|\init\|_{L^1(\mathbb R)}
=
1$.
Therefore $u(\cdot,t)\in L^1(\mathbb R)\cap L^\infty(\mathbb R)$ for every
$t$, and since $0\le u\le M$, and $|f(u(x,t))|\le B_M u(x,t)$.
In particular $f(u(\cdot,t))\in L^1(\mathbb R)$. Notice  that
$B_M<\infty$, since $f\in C^1_{\mathrm{loc}}(\mathbb R)$ and $f(0)=0$.

We next prove that $u(\cdot,t)$ has unit mass. Since entropy solutions are
distributional weak solutions, we may use the weak formulation. Let $\varphi_R\in C_c^\infty(\mathbb R)$
satisfy $0\le \varphi_R\le 1$, $\varphi_R\equiv 1$ on $[-R,R]$, and
$|\varphi_R'|\le C/R$. 
Using a temporal cutoff approximating $\mathbf 1_{[0,t]}$ in the weak
formulation and testing in space with $\varphi_R$, we obtain
\begin{equation*}
\int_{\mathbb R} u(x,t)\varphi_R(x)\,dx
-
\int_{\mathbb R} \init(x)\varphi_R(x)\,dx
=
\int_0^t\int_{\mathbb R} f(u(x,\tau))\varphi_R'(x)\,dx\,d\tau .
\end{equation*}
Using $|f(u)|\le B_Mu$ and the $L^1$-contraction property \cite{kruvzkov1970first}, applied with the zero solution, i.e., $\|u(\cdot,t)\|_{L^1(\mathbb R)}
\le
\|\init\|_{L^1(\mathbb R)}
=1$, we have
\begin{equation*}
\left|
\int_0^t\int_{\mathbb R} f(u)\varphi_R'\,dx\,d\tau
\right|
\le
\frac{CB_Mt}{R}.
\end{equation*}
Letting $R\to\infty$ gives
\begin{equation*}
\int_{\mathbb R}u(x,t)\,dx=\int_{\mathbb R}\init(x)\,dx=1.
\end{equation*}

Now, let us verify finite second moment. Let $\psi_R$ be a smooth bounded approximation of
$\min\{x^2,R^2\}$, chosen so that
$0\le \psi_R\le x^2$, $\psi_R(x)\rightarrow x^2$ as $R\to \infty$, and
$|\psi_R'|\le 2\sqrt{\psi_R}$.
Set
\begin{equation*}
m_R(t):=\int_{\mathbb R}\psi_R(x)u(x,t)\,dx .
\end{equation*}
Testing the equation with $\psi_R$, justified by approximation, yields
\begin{equation*}
\frac{d}{dt}m_R(t)
=
\int_{\mathbb R}\psi_R'(x)f(u(x,t))\,dx
\end{equation*}
in the sense of distributions in time. Therefore,
\begin{equation*}
\left|\frac{d}{dt}m_R(t)\right|
\le
2B_M\int_{\mathbb R}\sqrt{\psi_R(x)}\,u(x,t)\,dx
\le
2B_M\,m_R(t)^{1/2},
\end{equation*}
where we used $\int u(x,t)\,dx=1$. Hence
\begin{equation*}
m_R(t)^{1/2}
\le
m_R(0)^{1/2}+B_Mt
\le
\left(\int_{\mathbb R}x^2\init(x)\,dx\right)^{1/2}+B_Mt .
\end{equation*}
Letting $R\to\infty$, by monotone convergence,
\begin{equation*}
\int_{\mathbb R}x^2u(x,t)\,dx
\le
\left[
\left(\int_{\mathbb R}x^2\init(x)\,dx\right)^{1/2}+B_Mt
\right]^2
<\infty .
\end{equation*}
Thus $u(\cdot,t)\in\mathcal P_2(\mathbb R)$ for every $t\in[0,T]$ (where one is using the $L^1$-continuous representative).

Now, since $f(0)=0$, the equation \eqref{eq: hyp} can be written as a continuity equation
\begin{equation}\label{eq: cont eq}
    u_t+(uv)_x=0,
\qquad
v(x,t)=
\begin{cases}
\dfrac{f(u(x,t))}{u(x,t)}, & u(x,t)>0,\\[1.2ex]
0, & u(x,t)=0.
\end{cases} 
\end{equation}
The assumption gives that  $\|v(\cdot,t)\|_{L^\infty(\RR)}\le B_M$. 

Since $uv=f(u)$ a.e., the conservation law \eqref{eq: hyp} is exactly the continuity
equation \eqref{eq: cont eq}
(we do not need $u$ to be classical solution). Assume without loss of generality that $s\leq t$. 
Using that $\int u(x,\tau)\,dx=1$, the dynamic formulation of the Wasserstein distance yields
\begin{equation*}
W_2(u(\cdot,t),u(\cdot,s))
\le
\int_s^t
\left(\int |v(x,\tau)|^2u(x,\tau)\,dx\right)^{1/2}d\tau\leq B_M|t-s|.
\end{equation*}
Note that this computation is valid since we already know that the pair $u,v$ solves the continuity equation $u_t + (uv)_x = 0$ 
in the sense of distributions and we have $\int_0^T \int |v(x,t)|^2 u(x,t)\,dx\,dt < \infty$, then 
these hypotheses imply that $t \mapsto u(\cdot,t)$ is an absolutely continuous curve in $(\mathcal{P}_2,W_2)$ \cite{Santambrogio-OTAM}.
Finally, the CDT isometry $W_2(u(\cdot,t),u(\cdot,s))=\|\widehat{u}(\cdot,t)-\widehat{u}(\cdot,s)\|_{L^2(dr)}$ (see \eqref{eq:cdt-isometry}) gives the first claim. 

To obtain the Kolmogorov width bound, given $n\in \mathbb N$, let us choose $2n + 1$ equally spaced time nodes $t_j=\frac{jT}{2n}$, $j=0,\dots, 2n$.
Let 
$$V_n:=\mathrm{span}\{\widehat u(\cdot,t_{2k+1})\}_{k=0}^{n-1}$$ be the span of the corresponding odd-time CDT snapshots, and approximate each point of the curve by the nearest node. Then $\mathrm{dim}(V_n)\leq n$, and using \eqref{eq: nodes} we have that the $L^2(dr)$-distance of an element $\widehat u(\cdot,t)$ to the linear subspace $V_n$ is upper bounded by  
\begin{equation*}
    \min_{k\in\{0,\dots,n-1\}} \|\widehat u(\cdot,t) - \widehat u(\cdot,t_{2k+1})\|_{L^2(dr)} \leq \min_{k\in\{0,\dots,n-1\}} B_M |t-t_{2k+1}| \leq B_M T/2n.  
\end{equation*}
Since the Kolmogorov $n$ width is the infimum over all $n$-dimensional subspaces,
we conclude that
$d_n\big(\widehat{\mathcal M}_{[0,T]};L^2(dr)\big)
\le
\tfrac{B_MT}{2n}$.
\end{proof}

\begin{proposition}\label{prop: u, u_t, u_tt}
    Let $\init \in C^1(I)$ with $\|\init\|_\infty < M$, $\|\init'\|_\infty < N$. Let $f\in C^2([-M,M])$ and assume $G(x):=f(x)/x$ extends to a function of class $C^1([-M,M])$. Let $u$ be a classical solution of the equation \eqref{eq: hyp} with initial condition $\init$, strictly before shock formation  (i.e., assume $1+tf''(\init(\xi))\init'(\xi)\geq \alpha>0$ for all $t\in[0,T]$, $\xi\in I$). 
    If $I=(a,b)$, assume in addition that the velocity $v(x,t)=G(u(x,t))$ satisfies the no-outflow condition $v(a,t)=v(b,t)=0$
    \footnote{(this is implied by the zero-flux condition $f(u)=0$ at the boundary
    provided $G(u)=0$ at the corresponding boundary states; in particular,
     if the boundary state is $u=0$, this requires $G(0)=f'(0)=0$)}. 
    Then, 
    \begin{enumerate}[(i)]
        \item $\widehat u \in C^2([0,T];L^2(dr))$;
        \item 
            $\widehat u_t(\cdot,t) = v(\widehat u(\cdot,t),t) 
            = G(u(\widehat u(\cdot,t),t))$ with  
            $$\|\widehat u_t(\cdot,t)\|_{L^2(dr)}^2 
            =\int \init(\xi) (G(\init(\xi)))^2(1+tf''(\init(\xi))\init '(\xi))d\xi;$$ 
        \item 
            $\widehat u_{tt}(\cdot,t) = -\left.\left((G'\circ u)^2 \, u \, u_x\right)\right| _{(\widehat u(\cdot,t),t)}$ with 
            $$\|\widehat u_{tt}(\cdot,t)\|_{L^2(dr)}^2 
            =\int\frac{\init(\xi)^3 (G'(\init(\xi)))^4(\init'(\xi))^2}{1+tf''(\init(\xi))\init '(\xi)}d\xi.
            $$ 
    \end{enumerate}
\end{proposition}

\begin{remark}
    In some cases it will not be simultaneously satisfied that the characteristic flow and the flow given by $v$ will leave the interval $I$ invariant for any initial condition. In other words, it is not true that simultaneously $f'(u_0(b)) = 0 = v(b,t)$. In these cases, the theorem will be valid for the initial conditions that guarantee this. For example, if $u_0\in C_c((a,b))$, then $f(u)=uv$ will automatically be zero at the boundary, and extending $u_0\equiv 0$ outside of $I$ will allow us to work on $\mathbb{R}$ where the flow invariance is trivially satisfied.  
   
\end{remark}

\begin{proof}
    Let $u$ be the solution of the hyperbolic equation \eqref{eq: hyp} with initial condition $u(\cdot,0)=\init$ obtained by the method of characteristics. That is, let $g^{(t)}$ be the characteristic flow generated by $\dot{\zeta}(t)=f'(\init (\xi))$, $\zeta(0) = \xi$. Then, $g^{(t)}(\xi)=\xi+tf'(\init(\xi))$ and $\xi(x,t):=(g^{(t)})^{-1}(x)$ is well defined and continuously differentiable as a function of $(x,t)$ (see Appendix \ref{sec: app Invertibility of Characteristics Before Shock}) and $u(x,t) = \init((g^{(t)})^{-1}(x))$.
    Indeed, we are considering the solution $u$
    strictly before shock formation, in the sense that the characteristic map
    $g^{(t)}(\xi)$ satisfies
    \begin{equation*}
    \inf_{(\xi,t)\in I\times[0,T]}
    \partial_\xi g^{(t)}(\xi)
    =
    \inf_{(\xi,t)\in I\times[0,T]}
    \left(1+t f''(\init(\xi))\init'(\xi)\right)
    \ge \alpha>0.
    \end{equation*}
    Hence, $u\in C^1(I \times [0,T])$.
    Now, take $v(x,t) = G(u(x,t))$. Then, $v\in C^1(I\times [0,T])\cap L^\infty(I\times [0,T])$ and 
    \begin{equation*}
        v(g^{(t)}(\xi),t)=G(\init (\xi)) \quad \Longrightarrow \quad v_x(g^{(t)}(\xi),t)=\frac{G'(\init(\xi))\init'(\xi)}{1+tf''(\init(\xi))\init'(\xi)}.
    \end{equation*}
    Therefore, before characteristic crossing $1+tf''(\init(\xi))\init'(\xi)\geq \alpha>0$, and using our hypotheses together with the fact that $g^{(t)}$ is a diffeomorphism, we have that $v_x$ is uniformly bounded in $(x,t)$. In particular, $v(\cdot,t)$ is Lipschitz in $x$, uniformly in $t$.

    For such velocity, consider the equation
\begin{equation}\label{eq: bs}
    \begin{cases}
        u_t+(vu)_x=0\\
        u(\cdot,0)=\init
    \end{cases}
\end{equation}
In particular, $u$ is also solution of \eqref{eq: bs}.
Since, $v$ satisfies the hypotheses of the Lemma \ref{lem:santambrogio-lagrangian-first-derivative} then, by \cite[Theorem 4.4]{Santambrogio-OTAM},  $u$ is the unique solution of  \eqref{eq: bs}. 

By Lemma \ref{lem: hyp v cont}, we know that $\widehat u\in C^1([0,T];L^2(dr))$ and $\widehat u_t(\cdot,t)=v(\widehat{u}(\cdot,t),t)$ in $L^2(dr)$ for all $t\in [0,T]$.

Since $v_x=(G'\circ u) \, u_x$ and $v_t=-(G'\circ u) \, (f'\circ u) \, u_x=-(G'\circ u) \, (G\circ u  +(G'\circ u) \, u) \, u_x$,
we obtain
\begin{align*}
    \mathfrak a(x,t)&:=v_t(x,t)+v_x(x,t)v(x,t)\\
    &=-\frac{\left(u(x,t)f'(u(x,t))-f(u(x,t)) \right)^2}{u^3(x,t)}u_x(x,t)\\
    &=-(G'(u(x,t)))^2  u(x,t)  u_x(x,t).
\end{align*}
Therefore,
\begin{equation*}
    \mathfrak a(g^{(t)}(\xi),t)=-\frac{(G'(\init(\xi)))^2 \,   \init(\xi) \, \init'(\xi)}{1+tf''(\init(\xi))\init'(\xi)},
\end{equation*}
where we used that we know how to express the solution $u(x,t)$ in terms of characteristics, before shock formation (see \eqref{eq: implicit characteristic flow}), i.e., $u(t,x)=\init(x-tf'(u(t,x)))$
and therefore
     \begin{align*}
         &u_x(x,t)=\init'(x-tf'(u(x,t)))(1-tf''(u(x,t))u_x(x,t)) \quad \Rightarrow\\
         &u_x(x,t)=\frac{\init'(x-tf'(u(x,t)))}{1+tf''(u(x,t))\init'(x-tf'(u(x,t)))}.
     \end{align*}
So, by our hypotheses, $\mathfrak a\in  C(I\times [0,T])\cap L^\infty(I\times [0,T])$. 
Let $\mathbf{a}(t):=\mathfrak a(\widehat u(\cdot,t),t)$. 
Then, Lemma \ref{lem:hyp-a-cont} applies and we have that $\widehat u\in C^2([0,T];L^2(dr))$ with $\widehat u_{tt}(\cdot,t)
=
\mathbf{a}(t)$
in $L^2(dr)$.

Finally, let us compute the $L^2(dr)$-norm values of  $\widehat{u}_t(\cdot,t)$ and $\widehat{u}_{tt}(\cdot,t)$. For the hyperbolic case  $v=\frac{f\circ u}{u}$, \eqref{eq:norm-v} and \eqref{eq:norm-a} read as 
    \begin{align*}
        \|\widehat{u}_t(\cdot,t)\|_{L^2(dr)}^2&=\int \frac{f(u(x,t))^2}{u(x,t)}dx,\\
\|\widehat{u}_{tt}(\cdot,t)\|_{L^2(dr)}^2&=\int\frac{\left(u(x,t)f'(u(x,t))-f(u(x,t)) \right)^4}{u^5(x,t)}u_x^2(x,t)dx,
    \end{align*}
     where we have used that $u(x,t)\geq0$.
     Using again that we know how to express the solution $u(x,t)$ in terms of characteristics, before shock formation, one can use the substitution
     \begin{equation*}
         \begin{cases}
         \xi=x-tf'(u(x,t))\\
         \left(1+tf''(\init(\xi))\init  '(\xi)\right)d\xi=dx,
         \end{cases}
     \end{equation*}
     obtaining
     \begin{align*}
         \|\widehat{u}_t(\cdot,t)\|_{L^2(dr)}^2&=\int\frac{f(\init(\xi))^2}{\init(\xi)}(1+tf''(\init(\xi))\init '(\xi))d\xi\\
         &=\int \init(\xi) (G(\init(\xi)))^2(1+tf''(\init(\xi))\init '(\xi))d\xi\\
         \|\widehat{u}_{tt}(\cdot,t)\|_{L^2(dr)}^2&=\int \frac{(\init(\xi) f'(\init(\xi))-f(\init(\xi)))^4}{\init(\xi)^5} \frac{\init '(\xi)^2}{1+tf''(\init(\xi))\init  '(\xi)}d\xi\\
         &=\int\frac{\init(\xi)^3 (G'(\init(\xi)))^4(\init'(\xi))^2}{1+tf''(\init(\xi))\init '(\xi)}d\xi.
     \end{align*}
    Here the change of variables is performed using that $g^{(t)}$ is a diffeomorphism on the  whole line. Since the integrands vanish outside the support
of $\init$, which remains inside $I$ throughout $[0,T]$, the resulting
whole-line integrals reduce to the displayed integrals over $I$.
\end{proof}

In the next result, under the hypothesis of Proposition \ref{prop: u, u_t, u_tt}, we provide a precise pre-shock CDT Kolmogorov $n$-width bound.

\begin{theorem}\label{thm:hyperbolic-cdt-width}
    Let $\init \in C^1(I)$ with $\|\init\|_\infty < M$, $\|\init'\|_\infty < N$. Let $f\in C^2([-M,M])$ and assume $G(x):=f(x)/x$ extends to a function of class $C^1([-M,M])$. Let $u$ be the classical solution of \eqref{eq: hyp}
    on $I\times[0,T]$, with zero-flux boundary conditions, and strictly before characteristic crossing, i.e., assume $$1+tf''(\init(\xi))\init'(\xi)\geq \alpha>0 \qquad \text{for all } t\in[0,T], \,  \xi\in I.$$ 
    If $I=(a,b)$, assume in addition that the velocity $v(x,t)=G(u(x,t))$ satisfies the no-outflow condition $v(a,t)=v(b,t)=0$.
    Then $\widehat u \in C^2([0,T];L^2(dr))$ and
    \begin{equation*}
    d_n\big(\widehat{\mathcal M}_{[0,T]};L^2(dr)\big)
    \le
    \frac{C \, T^2}{8 \, \sqrt{\alpha} \, (n-1)^2} \qquad\text{for every } n\ge2,
    \end{equation*}
    where $C^2:=\int \tfrac{(\init(\xi) f'(\init(\xi))-f(\init(\xi)))^4 \init '(\xi)^2 }{\init(\xi)^5 } d\xi=\int \init(\xi)^3 (G'(\init(\xi)))^4(\init'(\xi))^2d\xi.$
\end{theorem}

Notice that, as a direct consequence of this result, if in particular $f(u)=Au$, then $\init f'(\init)-f(\init)=0$ and so $\|\widehat u_{tt}(\cdot,t)\|_{L^2(dr)}=0$ obtaining 
\begin{equation*}
d_n\big(\widehat{\mathcal M}_{[0,T]};L^2(dr)\big)=0
\qquad\text{for every } n\ge2,
\end{equation*}
which coincides with \eqref{eq:kol-2-adv} in Section \ref{sec:linear-transport-eq}.

\begin{proof}
The proof is a direct consequence of Proposition \ref{prop: u, u_t, u_tt} 
in combination with part \ref{item:2-curve} from Lemma \ref{lem:curve-width} to obtain the Kolmogorov $n$-width estimate. Indeed, from Proposition \ref{prop: u, u_t, u_tt}, $\widehat u \in C^2([0,T];L^2(dr))$
with $$\|\widehat u_{tt}(\cdot,t)\|_{L^2(dr)}^2\leq \frac{C^2}{\alpha} \qquad \forall t\in [0,T]$$
for the constant $C$ given in the statement of the theorem. Therefore, by applying Lemma \ref{lem:curve-width}, we have, for every $n\ge2$,
\begin{equation*}
    d_n\big(\widehat{\mathcal M}_{[0,T]};L^2(dr)\big)\leq \frac{T^2}{8  (n-1)^2}\sup_{t\in[0,T]}\|\widehat u_{tt}(\cdot,t)\|_{L^2(dr)}
    \leq
    \frac{C \, T^2}{8 \, \sqrt{\alpha} \, (n-1)^2}.
\end{equation*}
\end{proof}

We remark that Theorems \ref{thm:hyperbolic-cdt-lipschitz} and \ref{thm:hyperbolic-cdt-width} should be read as upper bounds, not as exact width identities.

\paragraph{Formal comparison with native-space bounds.} 
Using the same Hilbert-valued curve estimates as in CDT space, we can
formally derive corresponding bounds in native space. Throughout this
comparison, we assume sufficient regularity and integrability for all the
derivatives below to belong to the stated $L^2$ spaces. We also assume dynamic before shock formation  (i.e., assume $1+tf''(\init(\xi))\init'(\xi)\geq \alpha>0$ for all $t\in[0,T]$, $\xi\in I$).

Consider the set of solutions $\mathcal{M}_{[0,T]}$ of equation \eqref{eq: hyp} with initial condition $\init$ as in \eqref{eq: sol manifold}, i.e., $
    \mathcal{M}_{[0,T]}:=\{u(\cdot,t): t\in [0,T]\}$.
Using that $u(g^{(t)}(\xi),t) = \init(\xi)$ and differentiating with respect to $\xi$ we get: 
\begin{align*}
    & u_x(g^{(t)}(\xi),t)  
    =
    \frac{\init'(\xi)}{g^{(t)}_\xi(\xi)}
    =
    \frac{\init'(\xi)}{1+tf''(\init(\xi))\init'(\xi)}
    , \\
    & u_{xx}(g^{(t)}(\xi),t) 
    = 
    \frac{\init''(\xi)g_\xi^{(t)}(\xi) - \init'(\xi) g_{\xi\xi}^{(t)}(\xi)}{[g_\xi^{(t)}(\xi)]^3} 
    =
    \frac{\init''(\xi) - tf'''(\init(\xi))\init'(\xi)^3}{[1+tf''(\init(\xi))\init'(\xi)]^3}.
\end{align*}
Since $u_t=-f'(u)u_x$,
a change of variables $x=g^{(t)}(\xi)$ gives
\begin{align*}
    \|u_t(\cdot,t)\|_{L^2(I)}^2 
    &= 
    \int \left(f'(u(x,t))u_x(x,t) \right)^2 \, dx= 
    \int \frac{\left(f'(\init(\xi))\init'(\xi) \right)^2}{1+tf''(\init(\xi))\init'(\xi)} \, d\xi\leq 
    \frac{C_0^2}{\alpha},
    \\
    \|u_{tt}(\cdot,t)\|_{L^2(I)}^2
    &=
\int \left[f'(u)\right]^4 u_{xx}^2 \, dx
+ 4 \int \left[f'(u)\right]^3 f''(u)\,u_{xx}u_x^2 \, dx 
+ 4 \int \left[f'(u)\right]^2
\left[f''(u)\right]^2 u_x^4 \, dx\\
&=\int
\frac{
\Bigl(
\left[f'(u_0)\right]^2
\left[
u_0''
-
t f'''(u_0)(u_0')^3
\right]
+
2 f'(u_0)f''(u_0)(u_0')^2
\left[
1+t f''(u_0)u_0'
\right]
\Bigr)^2
}{
\left[
1+t f''(u_0)u_0'
\right]^5
}
\, d\xi\\ 
    &\leq \frac{C_1 + C_2 t + C_3 t^2}{\alpha^5},
\end{align*}
where $C_0,C_1,C_2,C_3$ are nonnegative constants depending only on $f, \init$ and its derivatives. 
Consequently, Lemma~\ref{lem:curve-width} gives the following formal
native-space estimates:
\begin{enumerate}
    \item \label{item: usual 1} If $\init\in C^1(I)$ with $\|\init\|_\infty < M$, $\|\init'\|_\infty < N$, and $f\in C^2([-M,M])$, then, for every $n\geq 1$,
    \begin{equation}
        d_n\big(\mathcal M_{[0,T]};L^2(I)\big)\leq \frac{T}{2n}\sup_{t\in[0,T]}\|u_{t}(\cdot,t)\|_{L^2(I)}
        \leq
        \frac{C_0T}{  \sqrt{\alpha} 2 n},
    \end{equation}
    where $C_0$ depends only on $f,\init$ and its derivatives.
    \item \label{item: usual 2} If $\init\in C^2(I)$ with $\|\init\|_\infty < M$, $\|\init'\|_\infty < N_1$, $\|\init''\|_\infty < N_2$, and $f\in C^3([-M,M])$, then, for every $n\geq 2$, 
    \begin{equation}
        d_n\big(\mathcal M_{[0,T]};L^2(I)\big)\leq \frac{T^2}{8(n-1)^2}\sup_{t\in[0,T]}\|u_{tt}(\cdot,t)\|_{L^2(I)}
        \leq
        \frac{\sqrt{C_1+C_2T+C_3T^2} \, T^2}{8\, \alpha^{5/2} \, (n-1)^2}.
    \end{equation}
    where $C_1,C_2,C_3$ depend only on $f,\init$ and its derivatives.
\end{enumerate}

These estimates, as well as those in
Theorems~\ref{thm:hyperbolic-cdt-lipschitz}
and~\ref{thm:hyperbolic-cdt-width}, are upper bounds and need not be
sharp. Nevertheless, applying the same curve-approximation principle in
the two coordinate systems reveals two useful distinctions.

First, at the regularity level
$\init\in C^1$ and $f\in C^2$, compare
Theorem~\ref{thm:hyperbolic-cdt-width} with
Item~\ref{item: usual 1}. The transformed trajectory admits an order
$(n-1)^{-2}$ bound (under appropriate assumptions), whereas the corresponding native-space estimate
obtained from first temporal regularity has order $n^{-1}$.

Second, when second temporal derivatives are compared directly,
Theorem~\ref{thm:hyperbolic-cdt-width} and
Item~\ref{item: usual 2} both yield order $(n-1)^{-2}$ decay. Their
dependence on the pre-shock Jacobian is, however, different. Consequently, the derived transformed-space upper bound deteriorates like
$\alpha^{-1/2}$ as $\alpha\downarrow0$, while the corresponding
native-space $C^2$ bound deteriorates like $\alpha^{-5/2}$. This is a comparison of the available upper bounds in their respective
ambient norms, not a claim about the exact Kolmogorov widths.

The first-order estimate uses nearest-node approximation of the
trajectory, while the second-order estimate uses piecewise linear
interpolation. Additional temporal regularity alone does not improve these
rates if the same approximants are retained; higher-order estimates would
require higher-degree interpolation together with the corresponding
dimension count.

Suppressing constants depending on $T$, $f$, and $\init$, the
comparison can be summarized schematically as:
\begin{equation*}
    \begin{array}{c|c|c}
        &
        \init \in C^1,\ f\in C^2
        &
        \init \in C^2,\ f\in C^3
        \\
        \hline
        \text{Native}
        &
        \displaystyle O\left(\frac{1}{\sqrt{\alpha}n}\right)
        &
        \displaystyle O\left(\frac{1}{\alpha^{5/2}n^2}\right)
        \\
        \hline
        \text{CDT}
        &
        \multicolumn{2}{c}{
            \displaystyle
            \min\left\{
                O\left(\frac{1}{n}\right),
                O\left(\frac{1}{\sqrt{\alpha}n^2}\right)
            \right\}
        }
    \end{array}.
\]

In the CDT row, the minimum reflects the fact that two different estimates are available. The first estimate, of order $O(n^{-1})$, comes from the Lipschitz-in-time control of the CDT trajectory in Theorem~\ref{thm:hyperbolic-cdt-lipschitz} and remains meaningful through shock formation. The second estimate, of order $O(\alpha^{-1/2}n^{-2})$, is only a
pre-shock estimate from Theorem~\ref{thm:hyperbolic-cdt-width} assuming $u_0\in C^1$, $f\in C^2$. Whenever both hypotheses are satisfied, both bounds hold simultaneously, and therefore the sharper available upper bound is their minimum. Thus the $O(n^{-1})$ estimate may be preferable close to shock formation, when $\alpha$ is small, while the $O(\alpha^{-1/2}n^{-2})$ estimate gives faster decay in $n$ away from shock formation.

Overall, comparing the Kolmogorov bounds in Native vs CDT space based on regularity of $\init$ and $f$ we get faster decay when $n \to \infty$ and slower growth when $\alpha \to 0$ (i.e. near shock times) and with less regularity assumptions when using CDT coordinates.

\clearpage
\subsection{Numerical examples}\label{sec: numerics hyperbolic case}

In this section we perform the numerical experiments described at the end of Section \ref{sec:CDT-POD-method} for several hyperbolic equations.

\paragraph{Inviscid Burgers.}

    {The inviscid Burgers equation
    \begin{align}\label{eq:burgers}
        u_t + \left(\frac{u^2}{2}\right)_x = 0
    \end{align}
    models nonlinear self-transport of a scalar quantity $u$: the quantity is carried by a flow whose speed is the quantity itself. So where $u$ is larger, that part of the profile moves faster; where $u$ is smaller, it moves slower. This makes the front of a wave steepen, and if the slope is negative somewhere initially, characteristics intersect and a shock forms in finite time. Rather than modeling a specific physical problem, this equation was introduced as the simplest prototype capturing key features of hyperbolic equations, both theoretically and numerically.}
    
    \paragraph{High-Fidelity Model  (HFM).}To solve \eqref{eq:burgers}, a hyperbolic finite-volume HFM was implemented using the exact first-order Godunov flux, periodic boundary conditions, and CFL-controlled adaptive time stepping, with CFL number $0.95$. The computational domain was $(x,t)\in[0,30]\times[0,0.5]$, discretized using 1500 spatial points and 200 output times. The initial condition was constructed from a function $g$ made from a  reproducible mixture of five Gaussian functions, tapered and buffered to localize its spatial variation.  After normalization, it was scaled and shifted according to
    \begin{equation}\label{eq: init burgers}
        u_0(x)=A g(x)+0.1.
    \end{equation}
    The experiment was restricted to times before the solution interacted with the boundary. Consequently, mass was conserved, and the dynamics could be regarded as effectively occurring on the full line.

    \paragraph{Experimental setup.}Results of numerical experiments for  initial condition \eqref{eq: init burgers} with  $A=10$ are shown in Fig. $\ref{fig: burgers_three_rows}$. In Fig. $\ref{fig: burgers_more_transport}$, we compare the evolution corresponding to two scaled initial conditions, with $A=3.3$ and $A=10$. This changes the maximum characteristic speed,
    \begin{equation*}
    \max_x f'(u_0)=\|u_0\|_\infty,
    \end{equation*}and therefore the amount of mass displacement over the prescribed time interval. Increasing $A$ also strengthens nonlinear steepening and reduces the shock-formation time, given for smooth initial data by
    \begin{equation*}
    t_s=-\frac{1}{\min_x u_0'(x)},
    \end{equation*}provided that $\min_x u_0'(x)<0$. These cases illustrate how increased transport and earlier shock formation affect the Eulerian and CDT reduced-order models.

    \begin{figure}[ht!]
        \centering
    
        \begin{subfigure}{\textwidth}
            \centering
            \begin{minipage}{0.01\textwidth}
                \centering
                \subcaption{}
            \end{minipage}
            \begin{minipage}{0.98\textwidth}
                \centering
                \includegraphics[width=\figwidth]{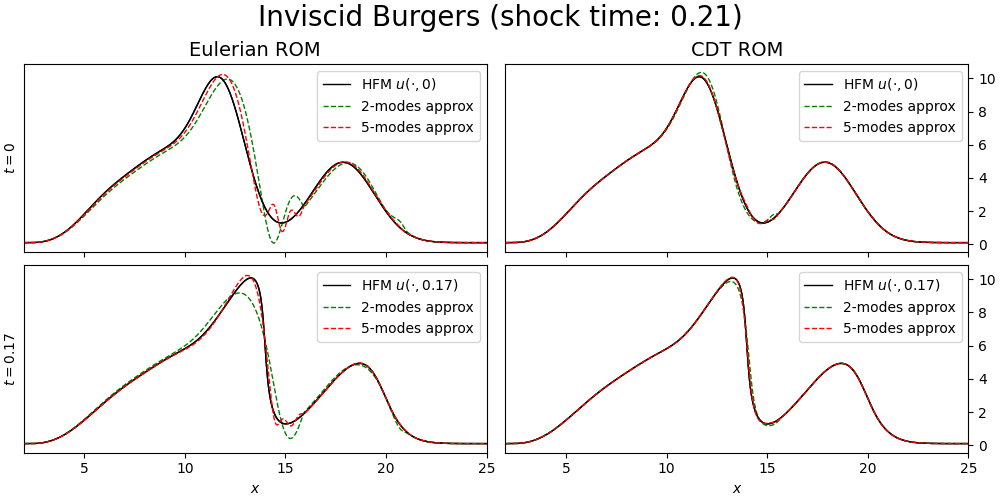}
            \end{minipage}
        \end{subfigure}
    
        \vspace{0.3cm}
    
        \begin{subfigure}{\textwidth}
            \centering
            \begin{minipage}{0.01\textwidth}
                \centering
                \subcaption{}
            \end{minipage}
            \begin{minipage}{0.98\textwidth}
                \centering
                \includegraphics[width=\figwidth]{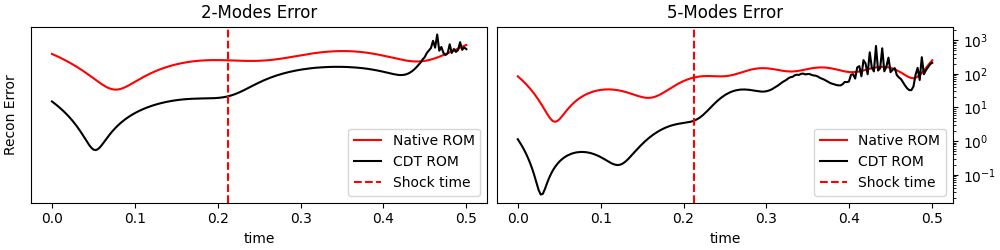}
            \end{minipage}
        \end{subfigure}
    
        \vspace{0.3cm}
    
        \begin{subfigure}{\textwidth}
            \centering
            \begin{minipage}{0.01\textwidth}
                \centering
                \subcaption{}
            \end{minipage}
            \begin{minipage}{0.98\textwidth}
                \centering
                \includegraphics[width=\figwidth]{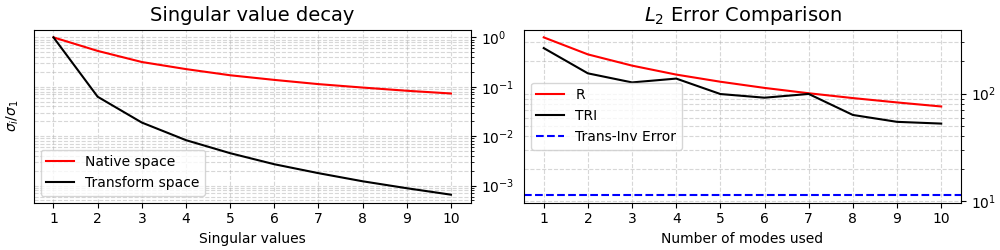}
            \end{minipage}
        \end{subfigure}
    
        \caption{%
            Inviscid Burgers equation: comparison between native Euclidean ROM and CDT-ROM. (a) The top panels show high-fidelity solutions and rank-$r$ (for $r=2$ and $r=5$) reconstructions at selected times. 
            (b) The middle panels show the time-dependent $L^2$ reconstruction error for $r=2$ and $r=5$, together with the shock time. 
            (c) The bottom-left panel compares the singular value decay of the native and transformed snapshot matrices. The bottom-right panel compares the global $L^2$ reconstruction error as a function of the number of retained modes and the transform/inverse-transform error floor (independent of the number of modes).
        }
        \label{fig: burgers_three_rows}
    \end{figure}

        \begin{figure}[ht!]
        \centering
    
        \begin{subfigure}{\textwidth}
            \centering
            \begin{minipage}{0.01\textwidth}
                \centering
                \subcaption{}
            \end{minipage}
            \begin{minipage}{0.98\textwidth}
                \centering
                \includegraphics[width=\figwidth]{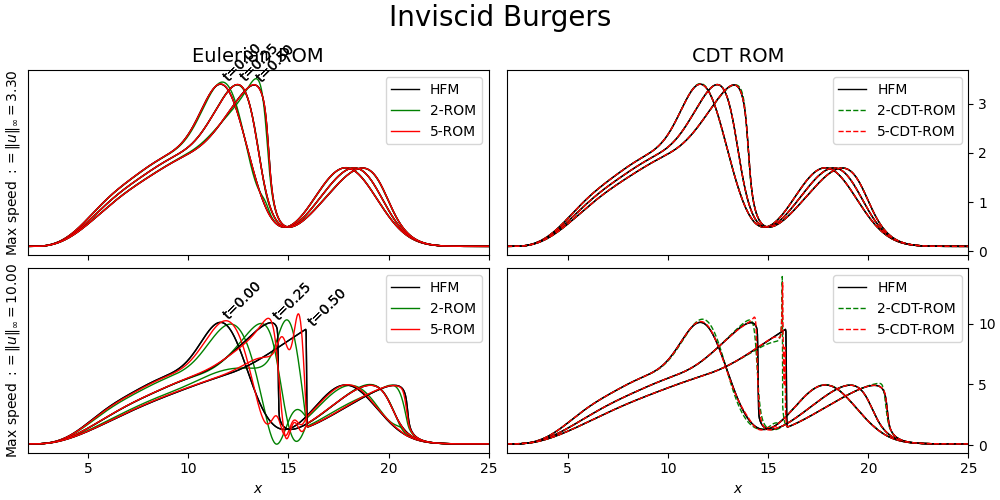}
            \end{minipage}
        \end{subfigure}
    
        \vspace{0.3cm}
    
        \begin{subfigure}{\textwidth}
            \centering
            \begin{minipage}{0.01\textwidth}
                \centering
                \subcaption{}
            \end{minipage}
            \begin{minipage}{0.98\textwidth}
                \centering
                \includegraphics[width=\figwidth]{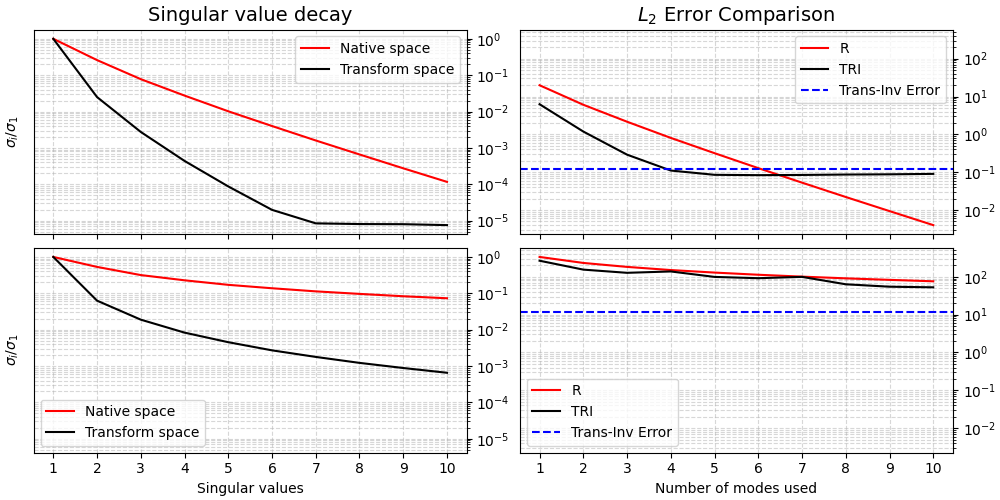}
            \end{minipage}
        \end{subfigure}
        
        \caption{%
            Inviscid Burgers equation for two regimes with different maximum characteristic speeds. Each row corresponds to one regime. The top panels compare representative reconstructions, while the bottom panels compare singular value decay and rank-dependent $L^2$ reconstruction errors. The results show that the CDT representation becomes increasingly effective as displacement-dominated dynamics become stronger.
        }
        \label{fig: burgers_more_transport}
    \end{figure}
    \clearpage

\paragraph{Traffic-Flow.}
    Traffic flow can be viewed as the transport of a conserved number of cars along a road. As density increases, drivers slow down, so the flow is small in both nearly empty and heavily congested traffic, and largest at intermediate density. Under this assumption, the traffic density $u(x,t)$ is modeled by
    \begin{align} \label{eq:traffic}
        u_t + \left(u(1-u)\right)_x = 0 
    \end{align}

    \paragraph{High-Fidelity Model (HFM).}To solve \eqref{eq:traffic}, a hyperbolic finite-volume HFM was implemented using a first-order Godunov solver with periodic numerical fluxes and CFL-controlled adaptive time stepping with CFL number $0.45$. The computational domain was $(x,t)\in[0,10]\times[0,1]$, discretized using 4500 spatial points and 1500 output times. The initial condition was generated as a reproducible mixture of five Gaussian functions and tapered and buffered so that its support remained away from the periodic boundary. Thus, mass was conserved, and the dynamics could be regarded as effectively occurring on the full line.

    \paragraph{Experimental setup.}Results of numerical experiments for a normalized initial profile scaled to satisfy $\|u_0\|_\infty=0.25$ are shown in Fig. $\ref{fig: traffic_three_rows}$. In Fig. $\ref{fig: traffic_more_transport}$, we compare the evolution corresponding to two scaled versions of the same initial profile. The amplitudes were selected so that the minimum characteristic speeds were
    \begin{equation*}
    \min_x f'(u_0)
    =1-2\|u_0\|_\infty
    \in\{0.25,0.75\}.
    \end{equation*}Equivalently, the corresponding initial amplitudes were $\|u_0\|_\infty=0.375$ and $0.125$, respectively. Since the characteristic speed is
    \begin{equation*}
    f'(u)=1-2u,
    \end{equation*}the larger-amplitude initial condition contains slower-moving high-density regions and produces greater deformation and relative displacement across the profile. These cases illustrate how changes in transport speed and nonlinear wave evolution affect the Eulerian and CDT reduced-order models.

    \begin{figure}[ht!]
        \centering
    
        \begin{subfigure}[ht!]{\textwidth}
            \centering
            \begin{minipage}{0.01\textwidth}
                \centering
                \subcaption{}
            \end{minipage}
            \begin{minipage}{0.98\textwidth}
                \centering
                 \includegraphics[width=\figwidth]{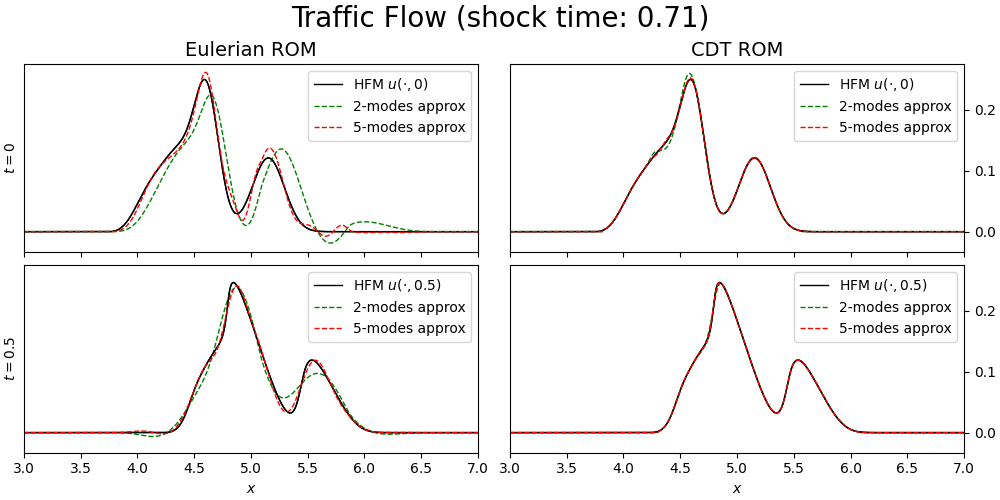}
            \end{minipage}
        \end{subfigure}
    
        \vspace{0.3cm}
    
        \begin{subfigure}[ht!]{\textwidth}
            \centering
            \begin{minipage}{0.01\textwidth}
                \centering
                \subcaption{}
            \end{minipage}
            \begin{minipage}{0.98\textwidth}
                \centering
                \includegraphics[width=\figwidth]{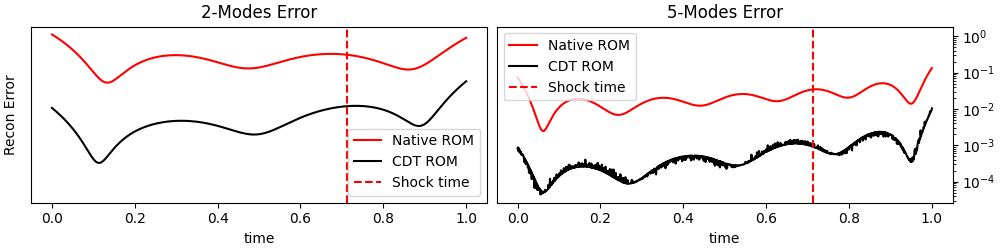}
            \end{minipage}
        \end{subfigure}
    
        \vspace{0.3cm}
    
        \begin{subfigure}[ht!]{\textwidth}
            \centering
            \begin{minipage}{0.01\textwidth}
                \centering
                \subcaption{}
            \end{minipage}
            \begin{minipage}{0.98\textwidth}
                \centering
                 \includegraphics[width=\figwidth]{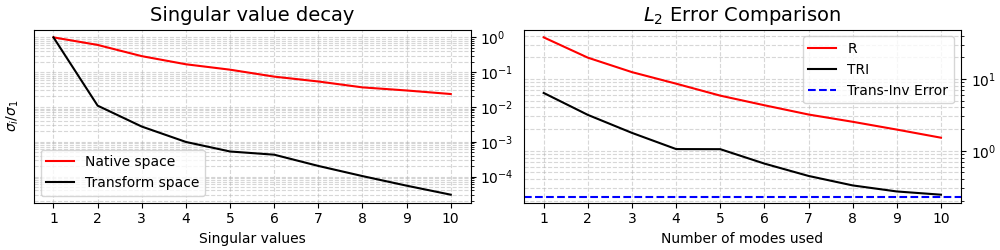}
            \end{minipage}
        \end{subfigure}
    
        \caption{%
            Traffic Flow equation: comparison between native Euclidean ROM and CDT-ROM. (a) The top panels show high-fidelity solutions and rank-$r$ (for $r=2$ and $r=5$) reconstructions at selected times. 
            (b) The middle panels show the time-dependent $L^2$ reconstruction error for $r=2$ and $r=5$, together with the shock time. 
            (c) The bottom-left panel compares the singular value decay of the native and transformed snapshot matrices. The bottom-right panel compares the global $L^2$ reconstruction error as a function of the number of retained modes and the transform/inverse-transform error floor (independent of the number of modes).
        }
        \label{fig: traffic_three_rows}
    \end{figure}

        \begin{figure}[ht!]
        \centering
    
        \begin{subfigure}{\textwidth}
            \centering
            \begin{minipage}{0.01\textwidth}
                \centering
                \subcaption{}
            \end{minipage}
            \begin{minipage}{0.98\textwidth}
                \centering
                \includegraphics[width=\figwidth]{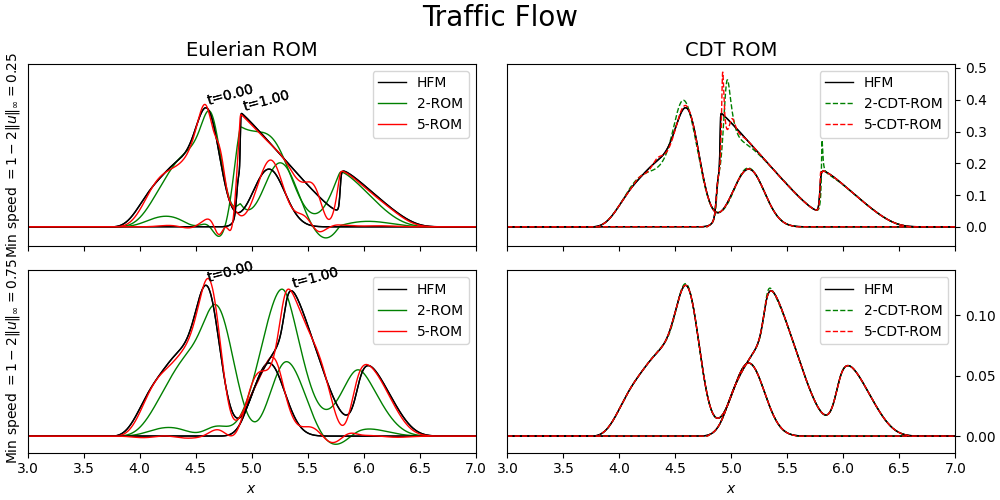}
            \end{minipage}
        \end{subfigure}
    
        \vspace{0.3cm}
    
        \begin{subfigure}[ht!]{\textwidth}
            \centering
            \begin{minipage}{0.01\textwidth}
                \centering
                \subcaption{}
            \end{minipage}
            \begin{minipage}{0.98\textwidth}
                \centering
               \includegraphics[width=\figwidth]{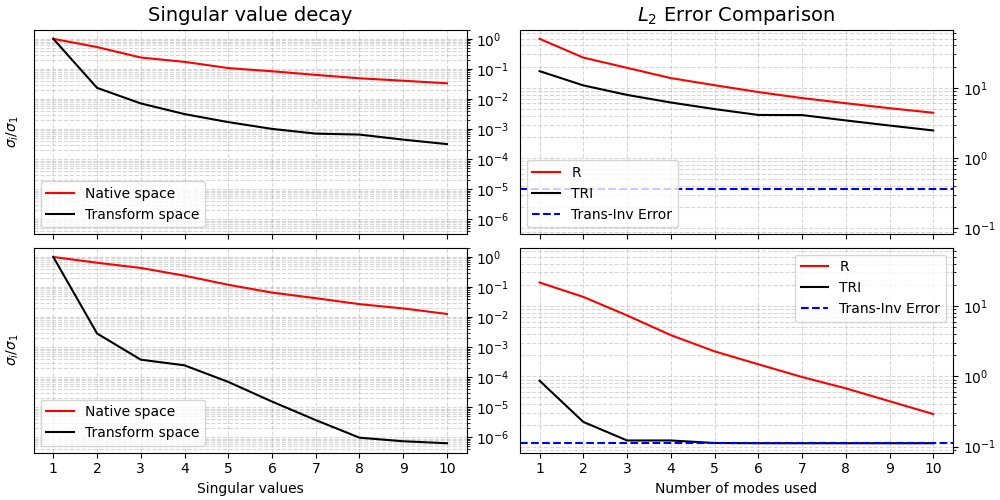}
            \end{minipage}
        \end{subfigure}
        
        \caption{%
            Traffic Flow equation for two regimes with different minimum characteristic speeds. Each row corresponds to one regime. The top panels compare representative reconstructions, while the bottom panels compare singular value decay and rank-dependent $L^2$ reconstruction errors. The results show that the CDT representation becomes increasingly effective as displacement-dominated dynamics become stronger.
        }
        \label{fig: traffic_more_transport}
    \end{figure}
    \clearpage

\paragraph{Buckley-Leverett.}

     {The Buckley-Leverett model describes the displacement of one incompressible, immiscible fluid by another in a porous medium, such as water pushing oil through reservoir rock. The unknown $u(x,t)$ represents the wetting-phase saturation, and its evolution is governed by mass conservation together with a nonlinear fractional-flow law. Under standard simplifying assumptions, the problem reduces to a scalar conservation law:
        \begin{align}\label{eq: buckley_leverett}
            \begin{cases}
                u_t + \left(\frac{u^2}{u^2+(1-u)^2}\right)_x = 0 \\
                 u_x(a,t) = 0 = u_x(b,t),\,\,\,\,  
            \end{cases}
        \end{align}
     whose nonlinear flux determines the propagation of saturation fronts, including shocks and rarefactions. The homogeneous Neumann boundary conditions $u_x(a,t) = 0 = u_x(b,t)$ impose zero saturation gradient at the two boundaries, so no artificial boundary variation is introduced there. In this sense, they serve as a simple numerical closure for the finite interval.}

     \paragraph{High-Fidelity Model (HFM).}To solve \eqref{eq: buckley_leverett}, a hyperbolic finite-volume HFM was implemented using a first-order Godunov-type numerical flux and CFL-controlled time stepping with CFL number $0.45$. The boundary numerical fluxes were set to zero. This numerical condition is not equivalent to a homogeneous Neumann condition in general. In the present experiments, however, the initial condition was compactly supported away from the boundaries, and the computational domain and final time were selected so that the solution remained negligible near them. Consequently, $u=0$, $u_x=0$, and the physical flux $f(u)=0$ were simultaneously satisfied at the boundaries throughout the simulated interval, making the numerical treatment consistent with the theoretical homogeneous Neumann conditions for this setup.
    The computational domain was $(x,t)\in[0,30]\times[0,1]$, discretized using 1500 spatial points and 200 output times. Initial saturations were generated as a reproducible mixture of five Gaussian functions, tapered and buffered to obtain compactly supported data, and subsequently normalized.

    \paragraph{Experimental setup.}
    Results for the normalized initial condition, satisfying $\|u_0\|_\infty=1$, are shown in Fig. $\ref{fig: buc_lev_three_rows}$. 
    In Fig. $\ref{fig: buc_lev_more_transport}$, we compare the evolution corresponding to two scaled versions of the same initial profile, with $\|u_0\|_\infty=0.33$ and $0.66$. Changing the amplitude modifies the characteristic velocity
    \begin{equation*}
    f'(u)=
    \frac{2u(1-u)}
    {\left(u^2+(1-u)^2\right)^2}.
    \end{equation*}
    These examples illustrate how differences in nonlinear wave propagation and mass displacement affect the Eulerian and CDT reduced-order models.

    \begin{figure}[ht!]
        \centering
    
        \begin{subfigure}{\textwidth}
            \centering
            \begin{minipage}{0.01\textwidth}
                \centering
                \subcaption{}
            \end{minipage}
            \begin{minipage}{0.98\textwidth}
                \centering
                 \includegraphics[width=\figwidth]{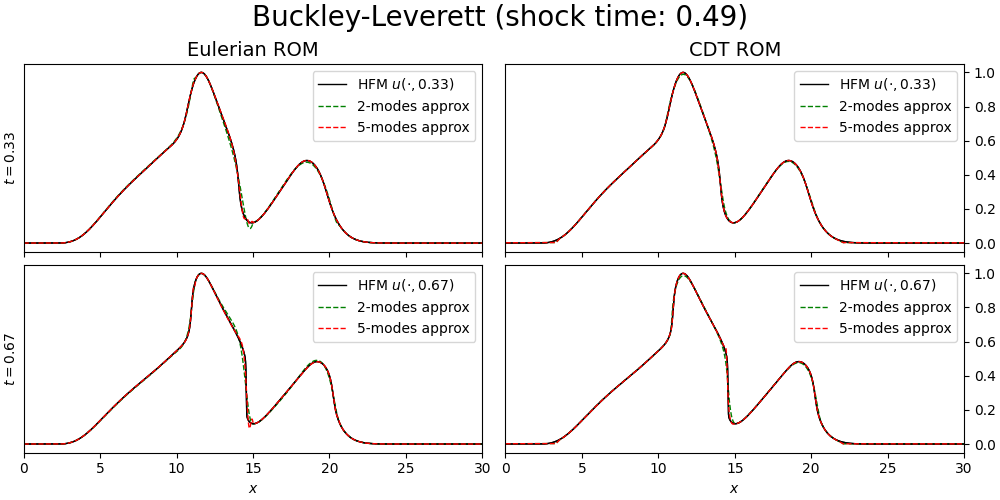}
            \end{minipage}
        \end{subfigure}
    
        \vspace{0.3cm}
    
        \begin{subfigure}{\textwidth}
            \centering
            \begin{minipage}{0.01\textwidth}
                \centering
                \subcaption{}
            \end{minipage}
            \begin{minipage}{0.98\textwidth}
                \centering
                \includegraphics[width=\figwidth]{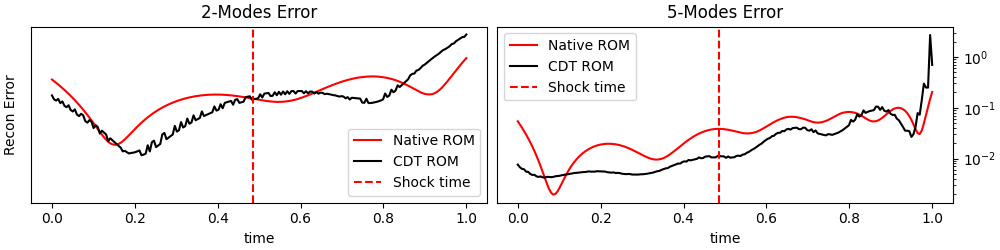}
            \end{minipage}
        \end{subfigure}
    
        \vspace{0.3cm}
    
        \begin{subfigure}{\textwidth}
            \centering
            \begin{minipage}{0.01\textwidth}
                \centering
                \subcaption{}
            \end{minipage}
            \begin{minipage}{0.98\textwidth}
                \centering
                \includegraphics[width=\figwidth]{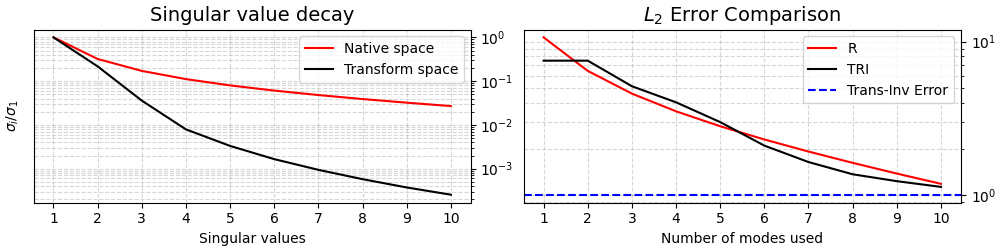}
            \end{minipage}
        \end{subfigure}
    
        \caption{%
            Buckley-Leverett equation: comparison between native Euclidean ROM and CDT-ROM. (a) The top panels show high-fidelity solutions and rank-$r$ (for $r=2$ and $r=5$) reconstructions at selected times. 
            (b) The middle panels show the time-dependent $L^2$ reconstruction error for $r=2$ and $r=5$, together with the shock time. 
            (c) The bottom-left panel compares the singular value decay of the native and transformed snapshot matrices. The bottom-right panel compares the global $L^2$ reconstruction error as a function of the number of retained modes and the transform/inverse-transform error floor (independent of the number of modes).
        }
        \label{fig: buc_lev_three_rows}
    \end{figure}

        \begin{figure}[ht!]
        \centering
    
        \begin{subfigure}{\textwidth}
            \centering
            \begin{minipage}{0.01\textwidth}
                \centering
                \subcaption{}
            \end{minipage}
            \begin{minipage}{0.98\textwidth}
                \centering
                 \includegraphics[width=\figwidth]{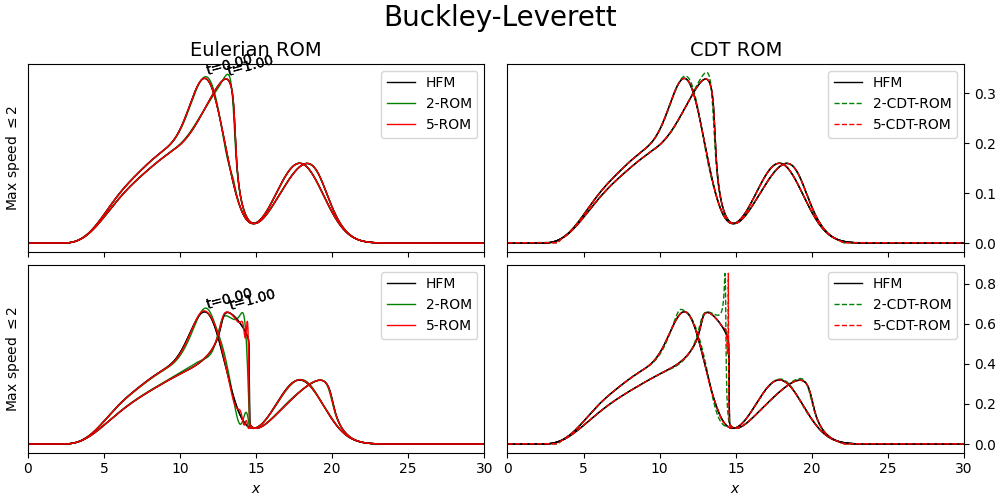}
            \end{minipage}
        \end{subfigure}
    
        \vspace{0.3cm}
    
        \begin{subfigure}{\textwidth}
            \centering
            \begin{minipage}{0.01\textwidth}
                \centering
                \subcaption{}
            \end{minipage}
            \begin{minipage}{0.98\textwidth}
                \centering
              \includegraphics[width=\figwidth]{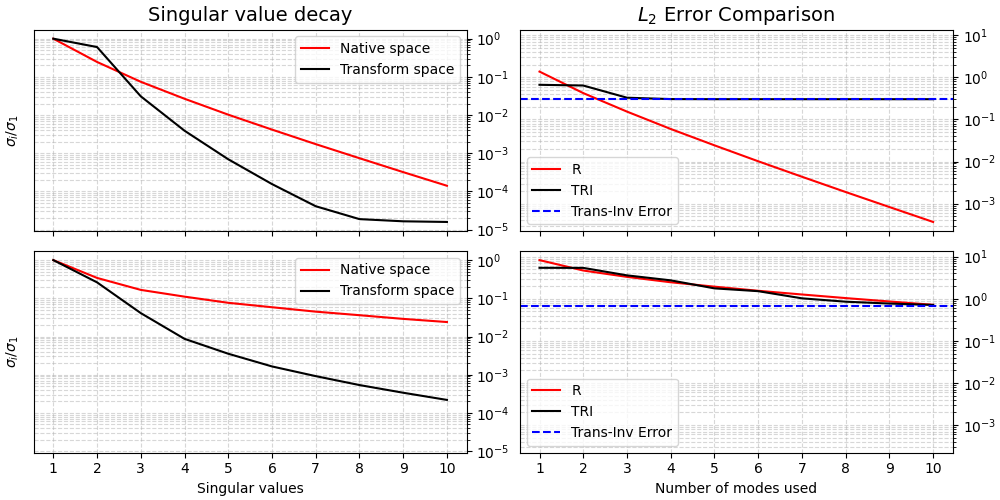}
            \end{minipage}
        \end{subfigure}
        
        \caption{%
            Buckley-Leverett equation for two regimes with different maximum characteristic speeds. Each row corresponds to one regime. The top panels compare representative reconstructions, while the bottom panels compare singular value decay and rank-dependent $L^2$ reconstruction errors. The results show that the CDT compression becomes increasingly effective as displacement-dominated dynamics become stronger. Nevertheless, the transform-inverse reconstruction error generates a bottleneck that prevents the effectiveness to be directly translated to native domain. 
        }
        \label{fig: buc_lev_more_transport}
    \end{figure}
    
    \clearpage

\section{Advection-Diffusion} \label{sec:adr-widths}

\subsection{Kolmogorov Width Bounds in CDT Space}\label{sec: kol adr}

Throughout this section we consider the linear advection-diffusion equation. Given $D>0$ and $A\in\RR$, let
\begin{equation}\label{eq:AD}
    u_t + A\,u_x \;=\; D\,u_{xx}, \qquad u(\cdot,0)=\init, 
\end{equation}
for $x \in I$ and $t\in (0,T]$, where $I\subset\mathbb R$ is an
interval, possibly $I=\mathbb R$. We  assume that the initial data $u(\cdot,0)=\init$ is a probability density with finite second moment, that is, $\init\ge0$ with total mass $\int \init=1$ and $\int x^2\init(x)dx<\infty$. 
When working in a bounded interval $I=(a,b)$, to guarantee mass-conservation at all times, we assume zero-flux boundary conditions given by $Au(a,t) - Du_{x}(a,t) = Au(b,t) - Du_{x}(b,t) = 0$ for all $t$.

For a fixed reference density $r$,  consider the set
$$\widehat \CalM_{[t_0,T]} = \{\widehat u(\cdot, t):\, t\in [t_0,T]\}\subset L^2(dr),$$
where $u$ is a solution of \eqref{eq:AD}, on the time interval $[t_0,T]$ for some $t_0\geq 0$, with initial state $\init$.
Our following results provide  Kolmogorov $2$-width bounds for $\widehat \CalM_{[t_0,T]}$. We start with Theorem \ref{thm:ADR-CDT-widths} and then add stronger regularity assumptions to obtain Theorem \ref{thm:ADR-CDT-widths-R}.

\begin{theorem}\label{thm:ADR-CDT-widths}
Assume $I=\RR$ and let $r>0$ be a fixed reference density on $\RR$, and let $u(\cdot,t)$ be the solution to \eqref{eq:AD}. 
Then,
$$d_2 (\widehat{\mathcal M}_{[0,T]};L^2(dr))\ \le\ \sqrt{2DT} . $$
\end{theorem}

Notice that as $D\to 0$ we recover the 2-Kolmogorov width $d_2 = 0$ of the transport equation. 

\begin{proof}
    Recall that
    \begin{align*}
        d_2(\widehat{\mathcal{M}}_{[0,T]};L^2(dr)) := \inf_{\substack{V\subset L^2(dr))\\ \dim V\le 2}} \ \sup_{t\in[0,T]}\inf_{\widehat v\in V} \|\widehat u(\cdot,t)-\widehat v\|_{L^2(dr)}.
    \end{align*}
    Take a fixed two-dimensional subspace $V := \mathrm{span}\{\widehat \init,\mathbf{1}\}$. Then, for every $t\in [0,T]$, $\widehat v(\cdot, t) := \widehat \init + At \in V$, and 
    \begin{align} \label{eq:d2-CDT-bound}
        d_2(\widehat{\mathcal{M}}_{[0,T]} ;L^2(dr)) 
        &\leq \sup_{t\in[0,T]} \|\widehat u(\cdot,t)-(\widehat \init + At )\|_{L^2(dr))} = \sup_{t\in[0,T]}  W_2\left(u(\cdot,t),\init(\cdot-At)\right).
    \end{align}
    Now take a random variable $X_0$ with law $X_0\sim \init$ and let $(\beta_t)_{t\geq 0}$ be a standard one-dimensional Brownian motion independent of $X_0$. We define the processes $X_t := X_0 +At +\sqrt{2D}\beta_t$, where $\beta_t$ is a standard one-dimensional Brownian motion ($\beta_t\sim \mathcal{N}(0,t)$), $Y_t := X_0 + At$
    with probability laws $X_t \sim u(\cdot,t)$, $Y_t \sim \init(\cdot-At)$. 
    Thus we can bound
    \begin{align*}
        W_2^2\left(u(\cdot,t),\init(\cdot-At)\right) 
        &\leq \mathbb{E}_{(X_t,Y_t)} \left((X_t-Y_t)^2\right)=\mathbb{E}((\sqrt{2D}\beta_t)^2) = 2D t.
    \end{align*}
    Combining this with \eqref{eq:d2-CDT-bound} we obtain $d_2(\widehat {\mathcal{M}}_{[0,T]};L^2(dr)) \leq \sqrt{2DT}$.    
\end{proof}

\begin{lemma}\label{lem: hat tt C2}
    Let $D>0$, $A\in\mathbb R$, and let $u$ solve the advection-diffusion equation \eqref{eq:AD} on $\mathbb R\times(0,T]$, with
    $u(\cdot,0)=u_0$.
    Let $r>0$ be a reference density on $\RR$ and consider $\widehat u$ the CDT with respect to $r$. If $u_0\in L^1(\RR)$, then $\widehat u (\xi,t)$ is well defined for $(\xi,t) \in \RR\times[0,T]$. Moreover, for every $\xi$, the application $t\mapsto \widehat u(\xi,t) \in C^2([t_0,T])$ for every $t_0>0$, with 
    \begin{equation}
\widehat u_t(\xi,t)
    =
    \left.\left(
        A-D\frac{u_x}{u}
    \right)\right|_{(\widehat u(\xi,t),t)}
\quad \text{and} \quad
    \widehat u_{tt}(\xi,t)
    =
    -D^2
    \left.\left(
        \frac{u_{xxx}}{u}
        -
        2\frac{u_xu_{xx}}{u^2}
        +
        \frac{u_x^3}{u^3}
    \right)\right|_{(\widehat u(\xi,t),t)}.
    \end{equation}
    If, additionally, $u_0$ satisfies the hypotheses of Lemma \ref{lem: extra regularity diffusion}, then we can include $t_0=0$.
\end{lemma}

\begin{proof}
    By Lemma \ref{lem:heat-kernel-diffusion}, $u\in C(\RR\times [t_0,T])$, and satisfies the hypotheses of Lemma \ref{lem: F_u deriv wrt t} for $t_0>0$. Thus, for every $\xi\in \RR$, $t\mapsto \widehat u(\xi,t)$ is differentiable with
    \begin{align*}
        \widehat u_t &=
        -\frac{\int_{-\infty}^x u_t(y,t)\,dy}{u(x,t)}
        \Bigg|_{(x,t)=(\widehat u(\xi,t),t)}=\left.\left(
        A-D\frac{u_x}{u}
    \right)\right|_{(\widehat u(\xi,t),t)},
    \end{align*}
    since $u\in C^\infty(\RR\times(0,T])$ solves \eqref{eq:AD}, and $u(\cdot,t) \in C_0(\RR)$ for every $t>0$. This gives $\widehat u(\xi,\cdot) \in C^1([t_0,T])$ (by composition of continuous functions). Differentiating again and using the regularity of $u$ we get the formula for $\widehat u_{tt}$ that is also continuous for $t>t_0$. 
    Indeed, defining   $v(x,t):=A-D\frac{u_x(x,t)}{u(x,t)} $then
    $\widehat u_t(\xi,t)=v(\widehat u(\xi,t),t)$, and 
    we get
    \begin{align*}
        \widehat u_{tt}(\cdot,t) 
        &= v_x(\widehat u(\cdot,t),t) \widehat u_t(\cdot,t) + v_t(\widehat u(\cdot,t),t) = \left.(v_x v +v_t)\right|_{(\widehat u(\cdot,t),t)}\\
        &= \left[-D\left(\frac{u_x}{u}\right)_x\left( A-D\frac{u_x}{u}\right) -D\left(\frac{u_x}{u}\right)_t\right]\Bigg|_{(\widehat u(\cdot,t),t)} \\
        & =  -D^2\left(\frac{u_{xxx}}{u} - 2\frac{u_xu_{xx}}{u^2} + \frac{u_x^3}{u^3} \right)\Bigg|_{(\widehat u(\cdot,t),t)}, 
    \end{align*}
    where the equality is valid for  $t \in [t_0,T]$. Thus, $\widehat{u}(\xi,\cdot) \in C^2([t_0,T])$, for every $t_0>0,$ because it is a composition of continuous functions.

    If, additionally, $u_0$ satisfies the hypotheses of Lemma \ref{lem: extra regularity diffusion}, we get $u\in C(\RR\times [0,T])$ and we can also apply Lemma \ref{lem: F_u deriv wrt t}  with $t_0 = 0$: the same steps as before prove regularity of $t\mapsto \widehat u(\xi,t)$ up to $t_0 = 0$ included. 

\end{proof}

\begin{proposition}\label{prop: taylor approx adv-diff}
    Let $D>0$, $A\in\mathbb R$, and let $u$ solve the advection-diffusion equation \eqref{eq:AD} on $\mathbb R\times(0,T]$, with
    $u(\cdot,0)=u_0$.
    Let $r>0$ be a reference density on $\RR$, and consider $\widehat u $ the CDT of $u$ with respect to $r$. Then, for every $t_0>0$,
     \begin{align}\label{eq: adv-diff taylors_general_bound}
        \left\|
            \widehat u(\xi,t)
            -
            \widehat u(\xi,t_0)
            -
            (t-t_0)\,\widehat u_t(\xi,t_0)
        \right\|_{L^2(dr)}
        \le
        \int_{t_0}^t
        \|\widehat u_{ss}(\cdot,s)\|_{L^2(dr)}
        (t-s)\,ds.
    \end{align}
    If, additionally, $u_0$ satisfies the hypotheses of Lemma \ref{lem: extra regularity diffusion}, then we can include $t_0=0$.
\end{proposition}
\begin{proof}
    Since $t\mapsto \widehat u_{tt}(\xi,t)$ is continuous for every $t>0$, 
    Taylor's theorem with integral remainder gives
    \begin{equation*}
        \label{eq:taylors-remainder-pointwise}
        \widehat u(\xi,t)
        -
        \widehat u(\xi,t_0)
        -
        (t-t_0)\,\widehat u_t(\xi,t_0)
        =
        \int_{t_0}^t
        \widehat u_{ss}(\xi,s)
        (t-s)\,ds.
    \end{equation*}
    Taking $L^2(dr)$ norms on each side and applying Minkowski's inequality we get \eqref{eq: adv-diff taylors_general_bound}. Under the extra hypothesis on $u_0$, then $t\mapsto \widehat u_{tt}(\xi,t)$ is continuous on $[0,t]$ and the conclusion follows as before.

\end{proof}
\begin{corollary}\label{cor: bound-hat-utt-diff}
        Let $D>0$, $A\in\mathbb R$, and let $u$ solve the advection-diffusion equation \eqref{eq:AD} on $\mathbb R\times(0,T]$, with
    $u(\cdot,0)=u_0$.
    Let $r>0$ be a reference density on $\RR$, and consider $\widehat u $ the CDT of $u$ with respect to $r$. Then, for every $t_0>0$,
    \begin{equation}\label{eq: sup widehat u_ss bound}
        \sup_{s\in[t_0,t]} \|\widehat u_{ss}(\cdot,s) \|_{L^2(dr)} < D^2 C(u(\cdot,t_0)),
    \end{equation}
    where $C$ is a constant depending on $u(\cdot,t_0)$ and its spatial derivatives. Thus,
    \begin{align}\label{eq: taylor full bound}
        \left\|
            \widehat u(\xi,t)
            -
            \widehat u(\xi,t_0)
            -
            (t-t_0)\,\widehat u_t(\xi,t_0)
        \right\|_{L^2(dr)}
        \le D^2 C(u(\cdot,t_0)) (t-t_0)^2/2.
    \end{align}
    If, additionally, $u_0$ satisfies the hypotheses of Lemma \ref{lem: extra regularity diffusion}, then we can include $t_0=0$.
\end{corollary}
\begin{proof}
     Using \eqref{eq: push-forward id}, by the change of variable formula for pushforward measures we obtain $r(\xi)\, d\xi = u(x,t)\, dx$ when $x = \widehat u(\xi,t)$. Thus, 
     \begin{align*}
        \|\widehat u_{ss}(\cdot,s)\|_{L^2(dr)}^2
        &=
        D^4
        \int_{\mathbb R}
        \left|
            \frac{u_{xxx}(x,s)}{u(x,s)}
            -
            2\frac{u_xu_{xx}(x,s)}{u^2(x,s)}
            +
            \frac{u_x^3(x,s)}{u^3(x,s)}
        \right|^2
        u(x,s)\,dx.
    \end{align*}
    Now, naming $a=u_x/u$, $b=u_{xx}/u$, and $c=u_{xxx}/u$ and using that $|c-2ab+a^3|^2\le 3|c|^2 + 12|a|^2|b|^2 + 3|a|^6$, 
    \begin{align*}
        \frac{\|\widehat u_{ss}(\cdot,s)\|_{L^2(dr)}^2}{D^4}
        &\le
        3
        \int
        \left|
            \frac{u_{xxx}}{u}
        \right|^2
        u\,dx 
        +
        12
        \left(
            \int
            \left|
                \frac{u_x}{u}
            \right|^4
            u\,dx
        \right)^{1/2}
        \left(
            \int
            \left|
                \frac{u_{xx}}{u}
            \right|^4
            u\,dx
        \right)^{1/2}
        +
        3
        \int
        \left|
            \frac{u_x}{u}
        \right|^6
        u\,dx.
    \end{align*}
   Since the above integrals are translation invariant with respect to $x$, we can consider  $u(x,s)$ and its derivatives as convolutions of $u(\cdot,t_0)$ and its derivatives against  the heat kernel $G_{D(s-t_0)}$. Now,  we will use Lemma \ref{lem: convolution ambrosio} with $\rho = G_{D(s-t_0)}$, $\mu=u(\cdot,t_0)$, and $\Theta $ equal to $u_x(\cdot,t_0)/u(\cdot,t_0)$, $u_{xx}(\cdot,t_0)/u(\cdot,t_0)$ or $u_{xxx}(\cdot,t_0)/u(\cdot,t_0)$. The right hand side can be bounded by a constant that depends on $u(\cdot,t_0)$ and its spatial derivatives: 
    \begin{align*}
        C^2(u(\cdot,t_0)) &:= 3
        \int
        \left|
            \tfrac{u_{xxx}(x,t_0)}{u(x,t_0)}
        \right|^2
        u(x,t_0)\,dx +
        3
        \int
        \left|
            \tfrac{u_x(x,t_0)}{u(x,t_0)}
        \right|^6
        u(x,t_0)\,dx
        \\&\qquad +
        12
        \left(
            \int
            \left|
                \tfrac{u_x(x,t_0)}{u(x,t_0)}
            \right|^4
            u(x,t_0)\,dx 
            \right)^{1/2}
            \left(
            \int
            \left|
                \tfrac{u_{xx}(x,t_0)}{u(x,t_0)}
            \right|^4
            u(x,t_0)\,dx
        \right)^{1/2}
    \end{align*}
    Thus, we have proved $\eqref{eq: sup widehat u_ss bound}$. The bound $\eqref{eq: taylor full bound}$ follows from Proposition \ref{prop: taylor approx adv-diff} and $\eqref{eq: sup widehat u_ss bound}$.

    Under the extra regularity hypothesis on $u_0$, the constant $C(u(\cdot,0))=C(u_0)$ is calculated by writing $u(x,s)$ and its derivatives as convolutions of $u_0$ and its derivatives against  the heat kernel $G_{Ds}$:
    \begin{align}\label{eq: C of u0}
        C^2(u_0) &:= 3
        \int
        \left|
            \tfrac{u'''_0}{u_0}
        \right|^2
        u_0\,dx +        12
        \left(
            \int
            \left|
                \tfrac{u'_0}{u_0}
            \right|^4
            u_0\,dx 
            \right)^{1/2}
            \left(
            \int
            \left|
                \tfrac{u''_0}{u_0}
            \right|^4
            u_0\,dx
        \right)^{1/2}+
        3
        \int
        \left|
            \tfrac{u'_0}{u_0}
        \right|^6
        u_0\,dx
    \end{align}
   
\end{proof}

\begin{remark}\label{rem: conv lemma again}
        In the previous corollary it is not clear, a priori, that the constant $C(u(\cdot,t_0))<\infty$. Nevertheless, one can write $u(\cdot,t_0)$ and its derivatives as convolutions of $u_0$ against the heat kernel $G_{Dt_0}$ and its derivatives. Thus, using Lemma \ref{lem: convolution ambrosio} again, but 
        with $\rho = u_0$, $\mu=G_{Dt_0}$, and $\Theta $ equal to $G_{Dt_0}'/G_{Dt_0}$, $G_{Dt_0}''/G_{Dt_0}$ or $G_{Dt_0}'''/G_{Dt_0}$ (i.e., polynomials on $x$ and $1/Dt_0$) one obtains 
        \begin{small}
            \begin{align*}
            C^2(u(\cdot,t_0))&\leq3
            \int
            \left|
                \tfrac{G_{Dt_0}'''}{G_{Dt_0}}
            \right|^2
            G_{Dt_0}dx 
            +
            12
            \left(
                \int
                \left|
                    \tfrac{G_{Dt_0}'}{G_{Dt_0}}
                \right|^4
                G_{Dt_0}dx
                \int
                \left|
                    \tfrac{G_{Dt_0}''}{G_{Dt_0}}
                \right|^4
                G_{Dt_0}dx
            \right)^{1/2}
            +
            3
            \int
            \left|
                \tfrac{G_{Dt_0}'}{G_{Dt_0}}
            \right|^6
            G_{Dt_0}dx\\
            &=3\int \left(
            -\frac{x^3}{8D^3t_0^3}
            +
            \frac{3x}{4D^2t_0^2}
            \right)^2G_{Dt_0}(x)\,dx+3\int\left(-\frac{x}{2Dt_0}\right)^6G_{Dt_0}(x)\,dx \\
            &\qquad + 12\left(\int\left(-\frac{x}{2Dt_0}\right)^4G_{Dt_0}(x)\,dx\right)^{1/2}\left(\int \left(
            \frac{x^2}{4D^2t_0^2}
            -
            \frac{1}{2Dt_0}
            \right)^4G_{Dt_0}(x)\,dx\right)^{1/2}\\
            & = \frac{63+72\sqrt{5}}{8(Dt_0)^3}<\infty.
            \end{align*}
        \end{small}

\end{remark}

\begin{theorem}\label{thm:ADR-CDT-widths-R}
  Let $D>0$, $A\in\mathbb R$, and let $u$ solve the advection-diffusion equation \eqref{eq:AD} on $\mathbb R\times(0,T]$, with
    $u(\cdot,0)=u_0$.
    Let $r>0$ be a reference density on $\RR$, and consider $\widehat u $ the CDT of $u$ with respect to $r$. Then, for every $t_0>0$, 
    \begin{equation}\label{eq: kolm bound 2 AD in R}
        d_2\big(\widehat{\mathcal M}_{[t_0,T]};L^2(dr)\big)
        \leq
        D^2 \frac{(T-t_0)^2}{2}
        C(u(\cdot,t_0)).
    \end{equation}
    If, additionally, $u_0$ satisfies the hypotheses of Lemma \ref{lem: extra regularity diffusion}, then we can include $t_0=0$.
\end{theorem}

\begin{proof}
    The bound \eqref{eq: kolm bound 2 AD in R} follows from considering the subspace $V_2$ of $L^2(dr)$ spanned by $\widehat{u}(\cdot,t_0)$ and $\widehat{u}_t(\cdot,t_0)= \left.(A-D\frac{u_x}{u})\right|_{(\widehat{u}(\xi,t_0),t_0)}$ and applying Proposition \ref{prop: taylor approx adv-diff} and equation \eqref{eq: taylor full bound} to obtain,
    \begin{align*}
        d_2\big(\widehat{\mathcal M}_{[t_0,T]};L^2(dr)\big)
        &\leq
        \sup_{t\in [t_0,T]} \mathrm{dist}(\widehat u(\cdot,t),V_2)\\
        &\leq 
        \left\|
            \widehat u(\cdot,t)
            -
            \widehat u(\cdot,t_0)
            -
            (t-t_0),\widehat u_t(\cdot,t_0)
        \right\|_{L^2(dr)}\\
        &\le D^2 C(u(\cdot,t_0)) (T-t_0)^2/2.
    \end{align*}
    
    Under the extra regularity hypothesis on $u_0$, the same proof follows for $t_0=0$, getting  $d_2\big(\widehat{\mathcal M}_{[0,T]};L^2(dr)\big)
        \leq
        D^2 \frac{T^2}{2}
        C(u_0)$, where $C^2(u_0)$ is as in \eqref{eq: C of u0}.
\end{proof}

\begin{remark}
     Whereas in the case $t_0>0$ the constant $C(u(\cdot,t_0))$ may depend indirectly on the diffusion coefficient $D$, in the particular case $t_0=0$, it becomes clear that the bound $C^2(u_0)$ given in \eqref{eq: C of u0} depends only on the initial condition.
    Under this circumstance, we can take the limit $D\to 0$ in \eqref{eq: kolm bound 2 AD in R} and recover the pure advection 2-Kolmogorov width $d_2\big(\widehat{\mathcal M}_{[t_0,T]};L^2(dr)\big) =  0$.  
\end{remark}

\subsection{Numerical examples}\label{sec: num adv-diff}
    We perform the experiments described in Section~\ref{sec:CDT-POD-method} for the one-dimensional advection-diffusion equation \eqref{eq:AD}, with advection and diffusion coefficients $A$ and $D$, respectively. To generate the required HFMs, we use the computational domain $(x,t)\in[0,30]\times[0,1.5]$, discretized with 1500 spatial points and 200 output times. The spatial derivatives are approximated using centered finite differences, while time integration is performed using an implicit first-order method with homogeneous Dirichlet boundary conditions. The initial condition is generated as a reproducible mixture of five Gaussian functions, tapered and buffered so that it vanishes near the boundaries. Thus, mass is numerically preserved. 
    For Fig.\ref{fig: adv_diff_three_rows}, we use $A=5$ and $D=0.5$, resulting in an advection-dominated evolution with moderate diffusive spreading. For Fig.\ref{fig: adv_diff_more_transport}, we consider three cases: pure advection, with $(A,D)=(5,0)$; pure diffusion, with $(A,D)=(0,1)$; and combined advection and diffusion, with $(A,D)=(5,1)$. The same type of initial condition is used in all cases. 
    
    \begin{figure}[ht!]
        \centering
    
        \begin{subfigure}{\textwidth}
            \centering
            \begin{minipage}{0.01\textwidth}
                \centering
                \subcaption{}
            \end{minipage}
            \begin{minipage}{0.98\textwidth}
                \centering
                \includegraphics[width=\figwidth]{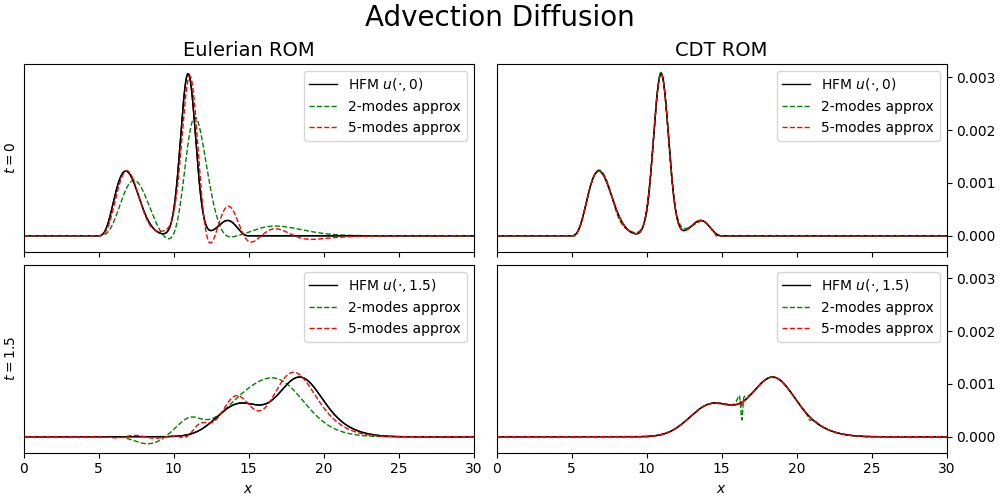}
            \end{minipage}
        \end{subfigure}
    
        \vspace{0.3cm}
    
        \begin{subfigure}{\textwidth}
            \centering
            \begin{minipage}{0.01\textwidth}
                \centering
                \subcaption{}
            \end{minipage}
            \begin{minipage}{0.98\textwidth}
                \centering
                \includegraphics[width=\figwidth]{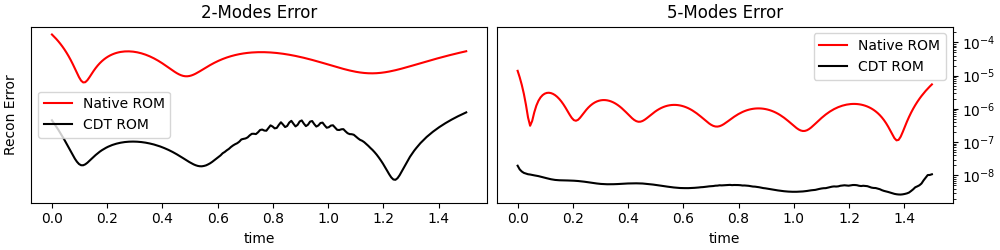}
            \end{minipage}
        \end{subfigure}
    
        \vspace{0.3cm}
    
        \begin{subfigure}{\textwidth}
            \centering
            \begin{minipage}{0.01\textwidth}
                \centering
                \subcaption{}
            \end{minipage}
            \begin{minipage}{0.98\textwidth}
                \centering
                \includegraphics[width=\figwidth]{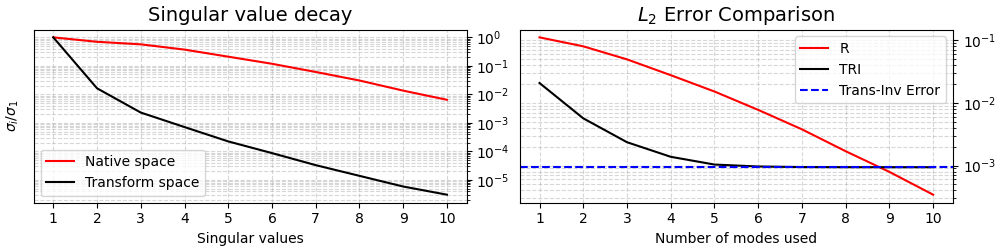}
            \end{minipage}
        \end{subfigure}
    
        \caption{%
            Advection-Diffusion equation: comparison between native Euclidean ROM and CDT-ROM. (a) The top panels show high-fidelity solutions and rank-$r$ (for $r=2$ and $r=5$) reconstructions at selected times. 
            (b) The middle panels show the time-dependent $L^2$ reconstruction error for $r=2$ and $r=5$. 
            (c) The bottom-left panel compares the singular value decay of the native and transformed snapshot matrices. The bottom-right panel compares the global $L^2$ reconstruction error as a function of the number of retained modes and the transform/inverse-transform error floor (independent of the number of modes).
        }
        \label{fig: adv_diff_three_rows}
    \end{figure}

    \begin{figure}[ht!]
        \centering
    
        \begin{subfigure}{\textwidth}
            \centering
            \begin{minipage}{0.01\textwidth}
                \centering
                \subcaption{}
            \end{minipage}
            \begin{minipage}{0.98\textwidth}
                \centering
                \includegraphics[width=\figwidth]{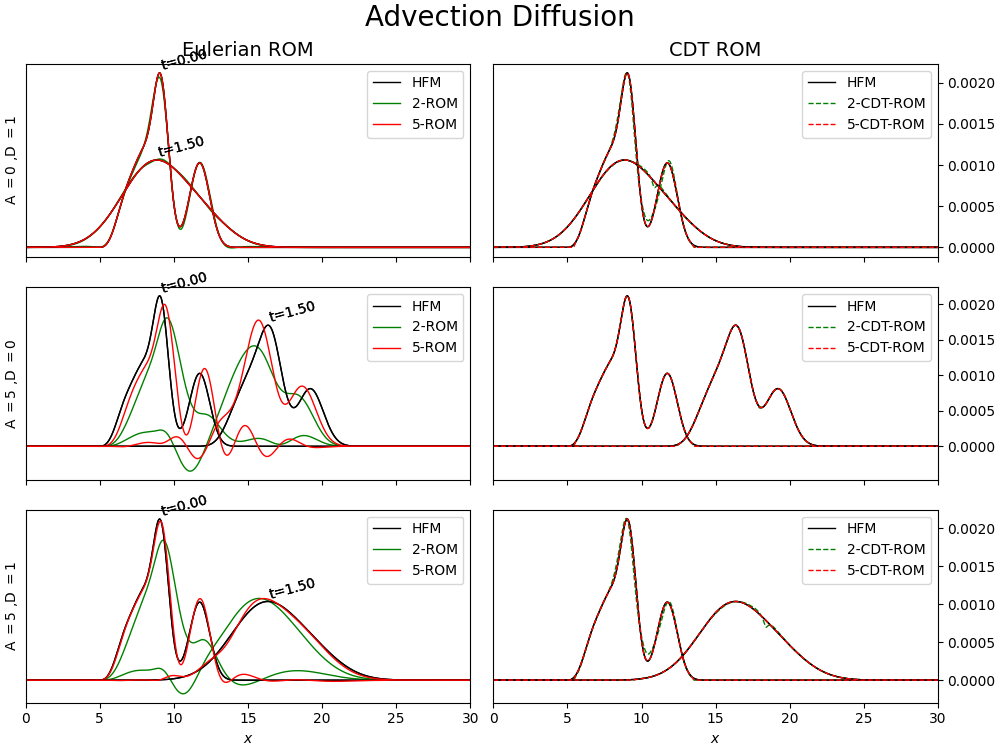}
            \end{minipage}
        \end{subfigure}
    
        \vspace{0.3cm}
    
        \begin{subfigure}{\textwidth}
            \centering
            \begin{minipage}{0.01\textwidth}
                \centering
                \subcaption{}
            \end{minipage}
            \begin{minipage}{0.98\textwidth}
                \centering
               \includegraphics[width=\figwidth]{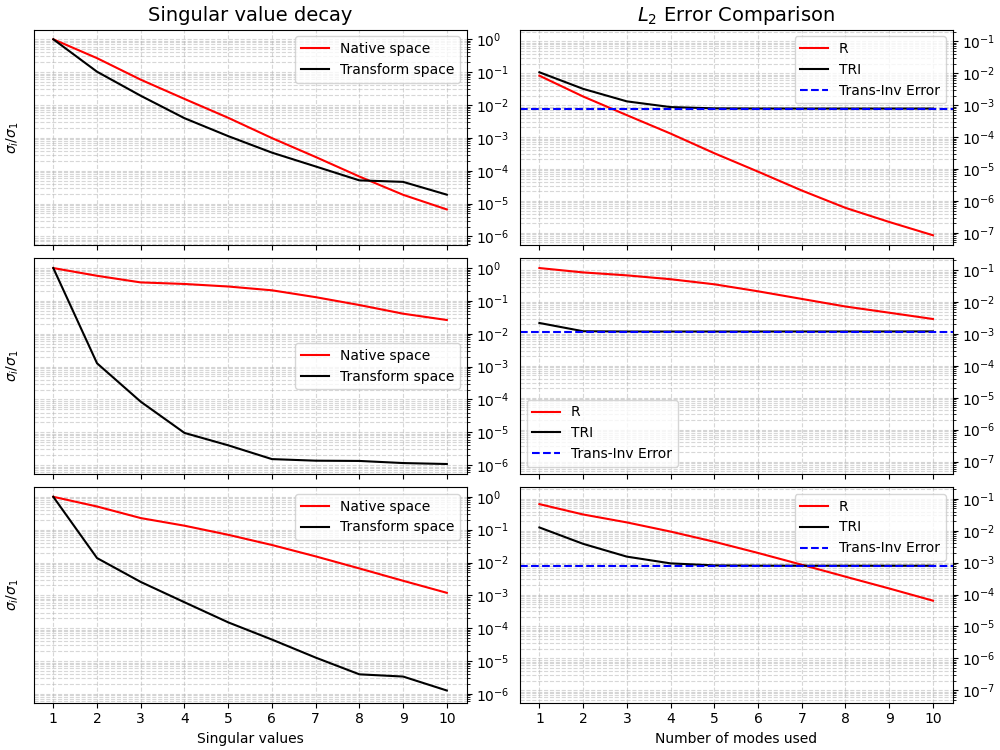}
            \end{minipage}
        \end{subfigure}
        
        \caption{%
            Advection-Diffusion equation for three regimes with different advection and diffusion coefficients. Each row corresponds to one regime. The top panels compare representative reconstructions, while the bottom panels compare singular value decay and rank-dependent $L^2$ reconstruction errors. The results show that reduced order modeling in CDT space is on par with native space for pure diffusion and becomes increasingly effective as advection-dominated dynamics become stronger. The transform-inverse bottleneck remains present for pure diffusion.
        }
        \label{fig: adv_diff_more_transport}
    \end{figure}
    \clearpage

\section{Conclusions}\label{sec: conclusions}
In this work, we developed a Reduced Order Modeling (ROM) framework for one-dimensional mass-preserving scalar conservative PDEs based on the Cumulative Distribution Transform (CDT). The proposed CDT-POD method maps
nonnegative, equal-mass snapshots to CDT space, performs Proper Orthogonal Decomposition (POD) in the transformed coordinates, and reconstructs physical-space approximations through the inverse CDT. This yields a nonlinear approximation method in Eulerian variables while retaining the computational simplicity of linear POD in a Hilbert space (precisely, in $L^2(dr)$ where the weight is given by a fixed reference probability density $r$ chosen by the user).

The motivation for this approach lies in the transport geometry encoded by the CDT. By identifying probability densities with monotone transport maps, the CDT converts the $2$-Wasserstein distance into an $L^2$-distance and represents spatial translations as affine trajectories. In this form, standard linear methods can be applied after preprocessing the data, while the underlying Wasserstein geometry remains well suited to conservation laws governed by mass displacement. For pure linear transport, this effect is exact: the transformed solution manifold lies in a two-dimensional linear space, so its Kolmogorov n-width in CDT-space vanishes for every $n\geq 2$.

The theoretical results developed in this paper quantify how this exact
transport structure persists, or is perturbed, for more general conservative
dynamics. For first-order hyperbolic conservation laws, we obtained two
complementary types of estimates. Robust transport-metric estimates give
$O(n^{-1})$ width bounds under low regularity and remain meaningful for
entropy solutions beyond shock formation. In smooth pre-shock regimes, where
the CDT dynamics can be differentiated more precisely, curvature-based
arguments yield sharper $O(n^{-2})$ bounds. For advection-diffusion, the
advective component retains the affine CDT structure, while diffusion produces
a controlled departure from the pure transport plane. 

The numerical experiments support this theoretical picture. Across inviscid Burgers' equation, traffic flow, 
Buckley-Leverett, and advection-diffusion dynamics, CDT-POD often produces faster singular-value
decay and lower reconstruction errors than Eulerian POD when considering a small number of
modes.

Our analysis also identifies several limitations and directions for future work. Although the CDT maps densities into $L^2(dr)$, its image is not the entire Hilbert space: admissible elements must correspond to monotone transport maps. Therefore, a linear POD approximation in CDT space may leave the range of the CDT if the reduced modes and coefficients do not preserve monotonicity. In such cases, the inverse transform may introduce artifacts or require additional post-processing. The mass preservation and non-negativity properties of the reconstructed density depend on the admissibility of the reduced CDT map. A natural next step is to develop constrained or structure-preserving variants of CDT-POD, for instance by enforcing monotonicity in the reduced coordinates, projecting onto the cone of admissible transport maps, or using adaptive reference densities.

There is also room to improve the theoretical results. The Kolmogorov width bounds obtained in this work rely on particular representations of the solutions under consideration (i.e., characteristic line solutions for hyperbolic equations, or convolution formulas for advection-diffusion).
This leaves two important directions only partially addressed by the present
analysis: nonlinear viscous conservation laws, where diffusion and nonlinear
transport interact, and sharper post-shock estimates for hyperbolic
conservation laws. Although
Theorem~\ref{thm:hyperbolic-cdt-lipschitz} already provides a robust
$\mathcal{O}(n^{-1})$ bound for entropy solutions beyond shock formation,
obtaining a higher-regularity post-shock counterpart to the pre-shock
$\mathcal{O}((n-1)^{-2})$ estimate of
Theorem~\ref{thm:hyperbolic-cdt-width} remains open.
We believe that the regularity results for the CDT obtained via the implicit function theorem  may provide a path toward treating these cases. In the viscous case, this could be done by combining the smoothing effects of diffusion with regularity arguments for the CDT. In the post-shock hyperbolic case, one could couple implicit solution formulas for the dynamics with the implicit definition of the CDT, and then apply suitable implicit function theorem arguments.

Overall, this work contributes to a more precise understanding of why transport-based coordinates can improve reduced models for one-dimensional conservative dynamics. It connects the observed performance of CDT-POD to concrete approximation-theoretic mechanisms and highlights several complementary ways of interpreting the CDT: as a map into Wasserstein space, as a global alternative to characteristic curves, and as an implicitly defined coordinate system. These perspectives provide a foundation for developing new methods that further integrate transport theory with reduced order modeling.

\begin{small}
\begin{center}
    \textbf{Acknowledgments}
   \vspace{-0.1in}
\end{center}
Harbir Antil, Aryan Saxena, and Sarswati Shah are partially supported by the Office of Naval Research (ONR) under Award NO: N00014-24-1-2147, NSF grant DMS-2408877, and the Air Force Office of Scientific Research (AFOSR) under Award NO: FA9550-25-1-0231. Gustavo K. Rohde, Ivan V. Medri, and Kristofor E. Pas are partially supported by NIH  Award GM130825, and ONR Award N000142212505.
\end{small}

\appendix

\section{Relations Between Kolmogorov Width and POD/SVD}\label{sec: kolmogorov width and SVD}

Consider, in general, a matrix  $X\in\mathbb R^{m\times n}$. In our work, we take $X$ to arise from solution snapshots of a PDE, either as $\mathbb U$ or $\widehat{\mathbb U}$ (see \eqref{eq:U} and \eqref{eq:hat-U}).
Proper Orthogonal Decomposition (POD) finds the best rank-$k$ approximation of a ``snapshot'' matrix $X$.  More precisely, it solves
\begin{equation}\label{eq:pod_frob}
    X_k = \arg\min_{\operatorname{rank}(Y)\le k} \|X-Y\|_F^2,
\end{equation}
where $\|\cdot\|_F$ denotes the Frobenius norm. By the Eckart-Young-Mirsky theorem, the minimizer is given by the truncated singular value decomposition (SVD) of $X$: If
$
X = U\Sigma V^T
$
is the singular value decomposition of $X$, then
$
X_k = U_k\Sigma_k V_k^T,
$
where $U_k\in\mathbb R^{m\times k}$, $\Sigma_k\in\mathbb R^{k\times k}$, and $V_k\in\mathbb R^{n\times k}$ contain the first $k$ singular vectors and singular values.
This matrix approximation problem has a natural interpretation in terms of approximation by subspaces. Let
$
\mathcal U_k := \operatorname{span}\{u_1,\dots,u_k\}
$
be the subspace spanned by the first $k$ left singular vectors, and let $P_V$ denote the orthogonal projection onto a subspace $V\subset \mathbb R^m$. Then $\mathcal U_k$ solves
\begin{equation*}
\mathcal U_k
=
\arg\min_{\dim(V)=k}\sum_{j=1}^n \|x_j-P_Vx_j\|_2^2=\arg\min_{\dim(V)=k}\sum_{j=1}^n \min_{v\in V}\|x_j-v\|_2^2,
\end{equation*}
where $x_1,\dots,x_n$ are the columns of $X$. 
Indeed, since $P_{\mathcal U_k}X = U_kU_k^TX = X_k$, the projection error and the rank-$k$ approximation error coincide:
\begin{equation*}
\|(I-P_{\mathcal U_k})X\|_F^2
=
\|X-X_k\|_F^2.
\end{equation*}
Therefore, the POD error is given by
\begin{equation*}
\|X-X_k\|_F^2 = \sum_{j>k}\sigma_j^2.
\end{equation*}

The previous formulation is naturally associated with the Frobenius norm, which measures the aggregate squared approximation error over the given snapshots. If instead one replaces the Frobenius norm by the operator norm, one obtains a related but conceptually different approximation problem:
\begin{equation*}\label{eq:pod_op}
    X_k = \arg\min_{\operatorname{rank}(Y)\le k}\|X-Y\|_2.
\end{equation*}
The minimizer is still given by the truncated SVD, but the error is now
\begin{equation*}
\|X-X_k\|_2 = \sigma_{k+1}.
\end{equation*}
In contrast with the Frobenius norm, the operator norm measures the worst-case error over all unit vectors. This leads to a formulation in terms of worst-case approximation over the image of the unit ball under $X$. More precisely,
\begin{equation*}
\|(I-P_V)X\|_2
=
\sup_{\|z\|_2\le 1}\|(I-P_V)Xz\|_2
=
\sup_{y\in X(B(0,1))}\|(I-P_V)y\|_2=\min_{\dim(V)=k}\sup_{y\in X(B(0,1))}\inf_{v\in V}\|y-v\|_2,
\end{equation*}
where $B(0,1)\subset \mathbb R^n$ denotes the Euclidean unit ball. 
Thus, if we define
$
K := X(B(0,1))\subset \mathbb R^m$, i.e., as the image of the unit ball by the linear transformation with matrix representation $X$,
then
\begin{equation*}
d_k(K;\mathbb R^m)
=
\min_{\dim(V)=k}\sup_{y\in K}\inf_{v\in V}\|y-v\|_2
=
\min_{\operatorname{rank}(Y)\le k}\|X-Y\|_2
=
\sigma_{k+1}.
\end{equation*}

The two viewpoints may be summarized as follows: the Frobenius norm leads to an optimal approximation problem for the given collection of snapshots in an average squared-error sense, whereas the operator norm leads to an optimal worst-case approximation problem, which is exactly the Kolmogorov $k$-width of the set $X(B(0,1))$. In particular, the decay of the singular values of the snapshot matrix quantifies how well the data can be approximated by low-dimensional linear spaces, a perspective that will be useful when studying the effect of the CDT on the underlying solution manifold.

In this work, this connection between singular value decay and approximation by low-dimensional spaces serves as a guiding principle in the numerical section. There, we compare the spectra of $\mathbb U$ and $\widehat{\mathbb U}$ and use these comparisons to illustrate the extent to which the CDT enhances the compressibility of the solution snapshots.

\section{Scalar Conservative Laws}\label{sec: scalar conservative laws}

\subsection{Vector-Valued Curves}\label{app: vector-val}

In this section we provide proofs for Lemmas \ref{lem:curve-width} and \ref{lem: Bochner}.

\begin{proof}[Proof of Lemma \ref{lem: Bochner}]
    Since $[0,T]$ is compact, $\gamma$ is bounded and uniformly continuous. Hence the Bochner integral
    $\Gamma(t)=\int_0^t \gamma(s)\,ds$
    is well-defined for every $t\in[0,T]$.
    Let us first show that $\Gamma$ is continuous. If $0\leq t_0<t_1\leq T$, then
    \begin{equation*}
        \|\Gamma(t_1)-\Gamma(t_0)\|_{L^2(dr)} 
        =
        \left\|\int_{t_0}^{t_1} \gamma(s)\,ds\right\|_{L^2(dr)} 
        \leq
        \int_{t_0}^{t_1} \|\gamma(s)\|_{L^2(dr)}\,ds.
    \end{equation*}
    Since $\gamma$ is bounded on $[0,T]$, there exists $M>0$ such that $\|\gamma(s)\|_{L^2(dr)}\leq M$ for all  $s\in[0,T]$. Thus,
    \begin{equation*}
        \|\Gamma(t_1)-\Gamma(t_0)\|_{L^2(dr)}
        \leq
        M|t_1-t_0|,   
    \end{equation*}
    which proves that $\Gamma\in C([0,T];L^2(dr))$.   
 Now, let us show that $\Gamma'(t) = \gamma(t)$ in the $L^2(dr)$ sense. Fix $t\in(0,T)$. For $h\neq 0$ sufficiently small so that
    $t+h\in[0,T]$, we have
    \begin{equation*}
        \frac{\Gamma(t+h)-\Gamma(t)}{h}-\gamma(t)
        =
        \frac{1}{h}\int_t^{t+h}\big(\gamma(s)-\gamma(t)\big)\,ds.
    \end{equation*}
    Using the Bochner integral estimate,
    \begin{align*}
        \left\|
        \frac{\Gamma(t+h)-\Gamma(t)}{h}-\gamma(t)
        \right\|_{L^2(dr)}
        &\leq
        \frac{1}{|h|}
        \int_{\min\{t,t+h\}}^{\max\{t,t+h\}}
        \|\gamma(s)-\gamma(t)\|_{L^2(dr)}\,ds\\
        &\leq 
        \sup_{s\in [t-|h|,t+|h|]\cap[0,T]}
        \|\gamma(s)-\gamma(t)\|_{L^2(dr)}.
    \end{align*}
    Since $\gamma$ is continuous as a map from $[0,T]$ into $L^2(dr)$, the right-hand
    side tends to $0$ as $h\to 0$. Therefore, $\Gamma'(t)=\gamma(t)$ for every $t\in(0,T)$. 
    At the endpoints, the same argument gives the one-sided derivatives $\Gamma'_+(0)=\gamma(0)$ and $\Gamma'_-(T)=\gamma(T)$. Since $\Gamma'=\gamma \in C([0,T];L^2(dr))$, then $ \Gamma\in C^1([0,T];L^2(dr))$.
\end{proof}

\begin{proof}[Proof of Lemma \ref{lem:curve-width}]
For the first part, assuming only $\gamma\in C^1([t_0,T];\mathcal{H})$, partition the interval $[t_0, T]$ using $2n+1$ points $t_0, \dots, t_{2n}=T$. Choose the space $V_n$ generated by $\{\gamma(t_{2j-1})\}_{j=1}^n$. Then, for every $t$ there is a $j^*$ such that $|t-t_{2j^*-1}|\leq L/2n$. Using the fundamental theorem of calculus for Bochner integrals (see Lemma \ref{lem: Bochner}), we get that
\begin{align*}
    \|\gamma(t) - \gamma(t_{2j^*-1})\|_{\mathcal H} 
    = 
    \left\|\int \gamma'(t) dt\right\|_{\mathcal H}
    \leq 
    \int \left\| \gamma'(t) \right\|_{\mathcal H}\, dt 
    \leq 
    M_1 L/2n.
\end{align*}

If instead we have $\gamma\in C^2([t_0,T];\mathcal{H})$, let $n\geq 2$. Partition $[t_0,T]$ into $n-1$ equal subintervals with nodes $t_j=t_0+j\frac{L}{n-1}$, $j=0,\ldots,n-1$,
and consider linear subspace of $\mathcal H$
\begin{equation*}
V_n:=\operatorname{span}\{\gamma(t_0),\ldots,\gamma(t_{n-1})\}. 
\end{equation*}
Then $\dim V_n\le n$. On each subinterval $[t_j,t_{j+1}]$, approximate $\gamma(t)$ by the chord joining $\gamma(t_j)$ and $\gamma(t_{j+1})$, which belongs to $V_n$. That is, consider the approximation given by
linear interpolant of $\gamma$ at the endpoints 
\begin{equation*}
(I_j\gamma)(t)
=
\frac{t_{j+1}-t}{t_{j+1}-t_j}\gamma(t_j)
+
\frac{t-t_j}{t_{j+1}-t_j}\gamma(t_{j+1}),
\qquad t\in[t_j,t_{j+1}].
\end{equation*}
By applying the scalar interpolation remainder formula to
$\ell\circ\gamma$, for each continuous linear functional $\ell\in\mathcal H^*$ with $\|\ell\|\le 1$ (where $\mathcal H^*$ denotes the dual space of $\mathcal{H}$),
and then taking the supremum over such $\ell$, we obtain
\begin{equation*}
    \|\gamma(t)-I_j\gamma(t)\|_{\mathcal H}
\le \frac{1}{2} (t-t_j)(t_{j+1}-t)
\sup_{s\in[t_j,t_{j+1}]}\|\gamma''(s)\|_{\mathcal H}\le
\frac{1}{8}\left(\frac{L}{n-1}\right)^2M_2.
\end{equation*}
Therefore, since $(I_j\gamma)(t)\in V_n$, we have $\operatorname{dist}(\gamma(t),V_n)
\le \|\gamma(t)-I_j\gamma(t)\|_{\mathcal H}
\le
\frac{1}{8}\left(\frac{L}{n-1}\right)^2M_2$, obtaining \eqref{eq:bound-M2} for all $n\geq 2$. 
\end{proof}

\subsection{Invertibility of Characteristics Before Shock}\label{sec: app Invertibility of Characteristics Before Shock}

Consider the scalar conservation law \eqref{eq: hyp}.
The characteristic curve starting from $\xi$ is by
$    \bigl(\Phi(\xi,t),t\bigr)
    =
    \bigl(\xi+t f'(\init(\xi)),t\bigr)$.
    For completeness we include the standard result on invertibility of characteristics before characteristic crossing/shock formation \cite[Sections 3.2 and 3.4]{evans2010partial}.
\begin{proposition}
Assume $f\in C^2(\mathbb R)$, $\init\in C^1(\mathbb R)$, and there exists
$\alpha>0$ such that
\begin{equation*}
    \partial_\xi \Phi(\xi,t)
    =
    1+t f''(\init(\xi))\init'(\xi)
    \ge \alpha
    \qquad
    \text{for all }(\xi,t)\in \mathbb R\times[0,T].
\end{equation*}
Then, for every $t\in[0,T]$, the map $    \Phi(\cdot,t):\mathbb R\to\mathbb R$
is one-to-one and onto. Consequently, the spacetime map
$\Psi:\mathbb R\times[0,T]\to \mathbb R\times[0,T]$, $\Psi(\xi,t)=\bigl(\Phi(\xi,t),t\bigr),$
is invertible, with inverse $\Psi^{-1}(x,t)=\bigl(\xi(x,t),t\bigr)$ where
    $\xi(x,t)=\Phi(\cdot,t)^{-1}(x)$.
Moreover, $\Psi^{-1}$ is continuous. If $f$ and $\init$ are smooth, then
$\Psi$ is a smooth diffeomorphism.
\end{proposition}

\begin{proof}
For each fixed $t\in[0,T]$, by hypothesis we have
that $\Phi(\cdot,t)$ is strictly increasing, and hence one-to-one.
We now show that it is also onto. Since
\begin{equation*}
    \Phi(\xi,t)-\Phi(0,t)
    =
    \int_0^\xi \partial_\xi\Phi(s,t)\,ds,
\end{equation*}
and $\partial_\xi\Phi(s,t)\ge \alpha$, it follows that
\begin{equation*}
    \Phi(\xi,t)\to+\infty
    \quad\text{as }\xi\to+\infty,
    \qquad
    \Phi(\xi,t)\to-\infty
    \quad\text{as }\xi\to-\infty.
\end{equation*}
Thus $\Phi(\cdot,t):\mathbb R\to\mathbb R$ is onto. Hence, for every fixed
$t\in[0,T]$, the map $\Phi(\cdot,t)$ is bijective.

It follows that the spacetime map $\Psi(\xi,t)=\bigl(\Phi(\xi,t),t\bigr)$
is also bijective. Indeed, if
$\Psi(\xi_0,t_0)=\Psi(\xi_1,t_1)$,
then the second coordinate gives $t_0=t_1$. Writing this common value as
$t$, the first coordinate gives $\Phi(\xi_0,t)=\Phi(\xi_1,t)$.
Since $\Phi(\cdot,t)$ is one-to-one, $\xi_0=\xi_1$. Hence $\Psi$ is injective.
Surjectivity follows because, for every $(x,t)\in\mathbb R\times[0,T]$,
there exists a unique $\xi=\Phi(\cdot,t)^{-1}(x)$ such that $\Psi(\xi,t)=(x,t)$.
Therefore, $\Psi^{-1}(x,t)=\bigl(\Phi(\cdot,t)^{-1}(x),t\bigr)$.

It remains to prove that $\Psi^{-1}$ is continuous. Let
$(x_n,t_n)\to(x,t)$
    in $\mathbb R\times[0,T]$
and define
\begin{equation*}
    \xi_n:=\Phi(\cdot,{t_n})^{-1}(x_n),
    \qquad
    \xi:=\Phi(\cdot,t)^{-1}(x).
\end{equation*}
Since $\partial_\xi\Phi(\cdot,{t_n})\ge\alpha$, we have $\alpha |\xi_n-\xi|
    \le
    |\Phi_{t_n}(\xi_n)-\Phi_{t_n}(\xi)|$.
Using $\Phi(\xi_n,t_n)=x_n$, we obtain
\begin{equation*}
    \alpha |\xi_n-\xi|
    \le
    |x_n-\Phi(\xi,{t_n})|
    \le
    |x_n-x|+|\Phi(\xi,t)-\Phi(\xi,t_n)|.
\end{equation*}
The right-hand side tends to zero, because $x_n\to x$, $t_n\to t$, and
$\Phi$ is continuous. Hence $\xi_n\to\xi$, proving that $\Psi^{-1}$ is
continuous.

Finally, if $f$ and $\init$ are smooth, then $\Psi$ is smooth. Its
Jacobian matrix is
\begin{equation*}
    D\Psi(\xi,t)
    =
    \begin{bmatrix}
        \partial_\xi\Phi(\xi,t) & \partial_t\Phi(\xi,t) \\
        0 & 1
    \end{bmatrix}
    =
    \begin{bmatrix}
        1+t f''(\init(\xi))\init'(\xi) & f'(\init(\xi)) \\
        0 & 1
    \end{bmatrix},
\end{equation*}
and therefore
\begin{equation*}
    \det D\Psi(\xi,t)
    =
    1+t f''(\init(\xi))\init'(\xi)
    \ge \alpha>0.
\end{equation*}
Thus the inverse function theorem gives smooth local inverses. Since $\Psi$
is globally bijective, these local inverses agree with the global inverse.
Therefore $\Psi^{-1}$ is smooth, and $\Psi$ is a smooth diffeomorphism.
\end{proof}

\begin{remark}[Bounded intervals]
If the initial labels are restricted to a bounded interval $[a,b]$, then the
same fixed-time argument shows that $\Phi_t$ is a bijection from $[a,b]$
onto its image $[\Phi_t(a),\Phi_t(b)]$, provided $\partial_\xi \Phi(\xi,t)>0$
    for all $(\xi,t)\in [a,b]\times[0,T]$.
Thus the spacetime map
$\Psi(\xi,t)=\bigl(\Phi(\xi,t),t\bigr)$
is invertible onto its image. In general this image is a moving strip rather
than the fixed rectangle $[a,b]\times[0,T]$, because the endpoints of the
interval are transported by the characteristic flow.
\end{remark}

\section{Advection-Diffusion Equation}

We recall general facts about solutions of the diffusion equation. Although these results are known, it was not trivial to find the exact references for our assumptions. For completeness in the exposition we provide the necessary adaptations needed starting from similar results in the area. We start by recalling the definition of the one-dimensional heat kernel for $s>0$, 
\begin{equation}\label{eq: the heat kernel}
        G_s(x)
    :=
    \tfrac{1}{\sqrt{4\pi s}}
    e^{-\tfrac{x^2}{4s}},
    \qquad x\in\RR.
\end{equation}

\begin{lemma}
\label{lem:heat-kernel-diffusion}
Let $D>0$, $A\in \RR$, and assume $u_0$ is a probability density (in particular, $0\leq u_0\in L^1(\RR)$).
Set
\begin{equation}
    \label{eq: convolution_diffusion_solution}
        \begin{cases}
            u(x,t)
            :=
            (G_{Dt}*u_0)(x-At)
            =
            \int_{\RR}G_{Dt}(y)u_0(x-At-y)\,dy,  \qquad \text{ for } t>0,\\
            u(x,0) = u_0
        \end{cases}
\end{equation}
Then, $u\in C^\infty(\RR\times(0,T])$, $u>0$ on $\RR\times(0,T]$, and it is the solution of the diffusion equation in the sense that 
    \begin{align*}
        &u_t(x,t) = Du_{xx}(x,t) - Au_x(x,t), \qquad \text{for } t>0\\
        &u(\cdot,t) \to u_0 \text{ in } L^1(\RR) \qquad \text{as } t\to 0.
    \end{align*}
    Moreover, $t\mapsto u_t(\cdot,t) \in C((0,T];L^1(\RR))$.
\end{lemma}

\begin{proof}
    Without loss of generality we assume the drift term $A=0$.
    Following the proof of \cite[Section 2.3, Theorem 1]{evans2010partial}, since $G_{Dt}$ is infinitely differentiable with uniformly bounded derivatives of all orders, on $\RR\times [t_0,\infty)$ for any $t_0>0$,  $u$ satisfies 
    \begin{enumerate}
        \item 
            $u\in C^\infty(\RR\times(0,T])$, with $\partial^i_t \partial^j_x u(x,t) = ((\partial^i_t\partial^j_x G_{Dt}) * u_0)(x)$, for $t>0$ (for all $i,j\in \mathbb N$).
        \item 
        $u$ is a solution of the diffusion equation $u_t=Du_{xx}$ for $x\in \RR$ and $t>0$.
    \end{enumerate}

    Since $u_0\ge 0$ with $\int_{\mathbb R}u_0(y)\,dy=1$ (in particular, $u_0\not\equiv 0$)
we have
\begin{equation*}
u(x,t)=\int_{\mathbb R}G_{Dt}(x-y)u_0(y)\,dy>0 \qquad \forall x\in \RR, t>0.
\end{equation*}

    To prove convergence to the initial datum in $L^1(\RR)$, notice that $G_{Dt}$ is an approximate identity as $t\to0$, hence $G_{Dt}*u_0\to u_0$ in $L^1(\mathbb R)$.
Indeed, by Minkowski's integral inequality, and 
using the continuity of translations in $L^1(\mathbb R)$, for every $\varepsilon>0$ there exists $\delta>0$ such that 
\begin{align*}
    \|G_{Dt}*u_0-u_0\|_{L^1(\RR)}
&\le
\int_{\mathbb R}G_{Dt}(y)
\|u_0(\cdot-y)-u_0(\cdot)\|_{L^1(\RR)}\,dy\\
    &\leq \varepsilon \int_{|y|\leq\delta}G_{Dt}(y)\, dy + 2\|u_0\|_{L^1(\RR)} \int_{|y|\geq\delta}  G_{Dt}(y)\, dy.
\end{align*}
Taking the $\limsup_{t\to 0}$ on both sides of the inequality, we get $\limsup_{t\to 0} \|G_{Dt}*u_0-u_0\|_{L^1(\RR)} \leq \varepsilon$ and since $\varepsilon$ is arbitrary the limit must be zero.

Finally, since
$u_t(\cdot,t)=\partial_tG_{Dt}*u_0$, Young's convolution inequality gives, for $s,t>0$,
\begin{equation*}
\|u_t(\cdot,t)-u_t(\cdot,s)\|_{L^1(\RR)}
\le
\|\partial_tG_{Dt}-\partial_tG_{Ds}\|_{L^1(\RR)}\|u_0\|_{L^1(\RR)}.
\end{equation*}
But $\|u_0\|_{L^1(\RR)}=1$, and
$t\mapsto \partial_tG_{Dt}$ is continuous from $(0,T]$ into $L^1(\mathbb R)$
(which follows from dominated  convergence theorem). Therefore, $t\mapsto u_t(\cdot,t)
\in C((0,T];L^1(\mathbb R))$.
\end{proof}

\begin{lemma} \label{lem: extra regularity diffusion}
    Let $u_0\in C^3(\RR)$, $u_0>0$, with $u_0^{(j)}\in L^\infty(\RR)\cap L^1(\RR)$, $j=0,1,2,3$, and let $u$ be defined by \eqref{eq: convolution_diffusion_solution}. Then, $u, u_x, u_t, u_{xx}, u_{xt}=u_{tx}, u_{xxx} \in C(\RR\times [0,T])$. Also, $u$ is a solution of the advection-diffusion equation in the classical sense up to $t=0$ (included). Moreover, $t \mapsto u_t(\cdot,t) \in C([0,T];L^1(\RR))$.
\end{lemma}

\begin{proof}
As in the previous proof,  without loss of generality we assume the drift term $A=0$.
    In \cite[Section 2.3, Theorem 1]{evans2010partial} it is proved that if $u_0\in C(\RR)\cap L^\infty(\RR)$, then $u$ defined by \eqref{eq: convolution_diffusion_solution}, satisfies
    \begin{enumerate}
        \item \label{item: open_cont} 
            $u\in C^\infty(\RR\times(0,T])$, with $\partial^i_t \partial^j_x u(x,t) = ((\partial^i_t\partial^j_x G_{Dt}) * u_0)(x)$, for $t>0$ (for all $i,j\in \mathbb N$).
        \item \label{item: solution_of_diffusion} 
            $u$ is a solution of the diffusion equation $u_t=Du_{xx}$ for $x\in \RR$ and $t>0$.
        \item \label{item: boundary_cont}
            $u(x,t) \to u_0(x_0)$ as $(x,t)\to(x_0,0)$ (which, as we defined $u(\cdot,0)=u_0$, means $u\in C(\RR\times[0,T])$).
    \end{enumerate}
    To prove that the spatial derivatives $u_x,u_{xx},u_{xxx} \in C(\RR\times[0,T])$  we apply Item \ref{item: boundary_cont} for $u_0^{(j)} \in C(\RR)\cap L^\infty(\RR)$, $j=1,2,3$, in place of $u_0$ in formula \eqref{eq: convolution_diffusion_solution}. The only caveat is to show that the partial derivatives can be passed to  $u_0$ through the convolution. This is done, for $t>0$, using integration by parts and Item \ref{item: open_cont},
    \begin{equation}
        \partial^j_x u = ((\partial^j_x G_{Dt}) * u_0)(x) =  G_{Dt} * u_0^{(j)}\qquad \text{ for } j=0, 1,2,3.
    \end{equation}

    Now we have to prove the continuity of $u_t,u_{xt},u_{tx}$. We start with $u_t$. By Item \ref{item: solution_of_diffusion}, $u_t = Du_{xx}$  for $t>0$. Since $u_{xx} \in C(\RR\times[0,T])$, with $u_{xx}(x,0) =u_0''(x)$, we only need to show that $u_t(x,0) = Du_0''(x)$ for every $x\in\RR$.  For $0<\varepsilon<h$, the fundamental theorem of calculus gives
    \begin{equation*}
        u(x,h) - u(x,\varepsilon) = 
        \int_\varepsilon^h
            u_s(x,s)\,ds
        =
        \int_\varepsilon^h
            Du_{xx}(x,s)\,ds.
    \end{equation*}
    Letting $\varepsilon\downarrow0$ we obtain
    \begin{equation*}
        u(x,h) -u(x,0)
        =
        \int_0^h Du_{xx}(x,s)\,ds.
    \end{equation*}
    Hence, as  $h\to 0$, 
    \[
    \begin{aligned}
            \frac{u(x,h)-u(x,0)}{h}-Du_0''(x)
        &=
        \frac{D}{h}
        \int_0^h
            u_{xx}(x,s)-u_0''(x)\,ds
        \longrightarrow0. 
    \end{aligned}
    \]
    Thus, $u_t(x,0) =  Du_0''(x) = Du_{xx}(x,0)$ as we wanted to prove. Now, the continuity of $u_{tx}$ is trivial since $u_{tx} = Du_{xxx}$. To prove that $u_{xt}$ is continuous we just need to do as with $u_t$ but with $u_x$ in place of $u$ and $u_{xxx}$ in place of $u_{xx}$.

    Finally, the $L^1$ claim follows from the expression
$u_t
=
D\,G_{Dt} * u_0''
-
A\,G_{Dt} * u_0'$.

\end{proof}

The next lemma is not exactly about the solution of the advection-diffusion equation, but will be used in the main text to find bounds of certain integrals associated with it. It is an adaptation of \cite[Lemma 8.1.10]{ambrosio2005gradient} from general measures to probability densities on the real line.

\begin{lemma}\label{lem: convolution ambrosio}
Let $p\geq 1$, let $\mu$ be a probability density on $\mathbb R$, and let
$\Theta:\mathbb R\to\mathbb R$ satisfy
\begin{equation*}
\int_{\mathbb R}|\Theta(x)|^p \mu(x)\,dx<\infty.
\end{equation*}
Then, for any nonnegative probability convolution kernel $\rho\in L^1(\RR)$ (i.e., $\rho\geq 0$ a.e. and $\int \rho=1$),
\begin{equation*}
\int_{\mathbb R}
\left|
\frac{(\Theta \, \mu)*\rho}{\mu*\rho}
\right|^p
(\mu*\rho)\,dx
\leq
\int_{\mathbb R}|\Theta|^p \mu\,dx.
\end{equation*}
Here the quotient is understood to be $0$ on the set where $\mu*\rho=0$.
\end{lemma}

\begin{proof}
    On $\mathbb R$ the proof is simple. Let
\begin{equation*}
q(x):=(\mu * \rho)(x)
=
\int_{\mathbb{R}} \mu(y)\rho(x-y)\,dy.
\end{equation*}
On the set where $q(x)>0$, define the probability measure
\begin{equation*}
d\nu_x(y)
:=
\frac{\mu(y)\rho(x-y)}{q(x)}\,dy.
\end{equation*}
Indeed,
\begin{equation*}
\int_{\mathbb{R}} d\nu_x(y)
=
\frac{1}{q(x)}
\int_{\mathbb{R}}\mu(y)\rho(x-y)\,dy
=1.
\end{equation*}
Moreover,
\begin{equation*}
\frac{((\Theta\mu)*\rho)(x)}{(\mu*\rho)(x)}
=
\frac{\int_{\mathbb{R}}\Theta(y)\mu(y)\rho(x-y)\,dy}{q(x)}
=
\int_{\mathbb{R}}\Theta(y)\,d\nu_x(y).
\end{equation*}
Since $p\geq 1$, the function $z\mapsto |z|^p$ is convex. Therefore,
by Jensen's inequality,
\begin{equation*}
\left|
\frac{((\Theta\mu)*\rho)(x)}{(\mu*\rho)(x)}
\right|^p
\leq
\int_{\mathbb{R}}|\Theta(y)|^p\,d\nu_x(y).
\end{equation*}
Multiplying by $q(x)=(\mu*\rho)(x)$, we obtain
\begin{equation*}
\left|
\frac{((\Theta\mu)*\rho)(x)}{(\mu*\rho)(x)}
\right|^p
(\mu*\rho)(x)
\leq
\int_{\mathbb{R}}
|\Theta(y)|^p\mu(y)\rho(x-y)\,dy.
\end{equation*}
On the set where $(\mu*\rho)(x)=0$, the left-hand side is defined to be
zero, so the same inequality remains valid.

Integrating with respect to $x$ and applying Tonelli's theorem gives
\begin{align*}
\int_{\mathbb{R}}
\left|
\frac{((\Theta\mu)*\rho)(x)}{(\mu*\rho)(x)}
\right|^p
(\mu*\rho)(x)\,dx
&\leq
\int_{\mathbb{R}}
\int_{\mathbb{R}}
|\Theta(y)|^p\mu(y)\rho(x-y)\,dy\,dx \\
&=
\int_{\mathbb{R}}
|\Theta(y)|^p\mu(y)
\left(
\int_{\mathbb{R}}\rho(x-y)\,dx
\right)\,dy.
\end{align*}
Since $\rho$ is a probability density,
\begin{equation*}
\int_{\mathbb{R}}\rho(x-y)\,dx
=
\int_{\mathbb{R}}\rho(z)\,dz
=1.
\end{equation*}
Hence
\begin{equation*}
\int_{\mathbb{R}}
\left|
\frac{((\Theta\mu)*\rho)(x)}{(\mu*\rho)(x)}
\right|^p
(\mu*\rho)(x)\,dx
\leq
\int_{\mathbb{R}}|\Theta(y)|^p\mu(y)\,dy,
\end{equation*}
which proves the claim.
\end{proof}

\bibliographystyle{IEEEtran}
\bibliography{references}

@article{mojgani2017lagrangian,
  title={Lagrangian basis method for dimensionality reduction of convection dominated nonlinear flows},
  author={Mojgani, Rambod and Balajewicz, Maciej}, 
  howpublished = "\url{http://aiweb.techfak.uni-bielefeld.de/content/bworld-robot-control-software/}",
  journal={arXiv preprint arXiv:1701.04343},
  year={2017}
}

@book{Santambrogio-OTAM,
  title     = {Optimal Transport for Applied Mathematicians},
  publisher = {Birkh\"auser Boston},
  year      = {2015},
  author    = {Santambrogio, Filippo},
  volume    = {87},
  series    = {PNLDE}}

@article{long2025reduced,
  title={A reduced-order model for advection-dominated problems based on the {Radon Cumulative Distribution Transform}},
  author={Long, Tobias and Barnett, Robert and Jefferson-Loveday, Richard and Stabile, Giovanni and Icardi, Matteo},
  journal={Advances in Computational Mathematics},
  volume={51},
  number={1},
  pages={5},
  year={2025},
  publisher={Springer}
}

@article{ren2021model,
  title={Model reduction of traveling-wave problems via {Radon} cumulative distribution transform},
  author={Ren, Jie and Wolf, William R and Mao, Xuerui},
  journal={Physical Review Fluids},
  volume={6},
  number={8},
  pages={L082501},
  year={2021},
  publisher={APS}
}

@article{rowley2017model,
  title={Model reduction for flow analysis and control},
  author={Rowley, Clarence W and Dawson, Scott TM},
  journal={Annual Review of Fluid Mechanics},
  volume={49},
  number={1},
  pages={387--417},
  year={2017},
  publisher={Annual Reviews}
}

@article{brunton2016discovering,
  title={Discovering governing equations from data by sparse identification of nonlinear dynamical systems},
  author={Brunton, Steven L and Proctor, Joshua L and Kutz, J Nathan},
  journal={Proceedings of the national academy of sciences},
  volume={113},
  number={15},
  pages={3932--3937},
  year={2016},
  publisher={National Academy of Sciences}
}

@book{evans2010partial,
  title={Partial differential equations},
  author={Evans, Lawrence C},
  volume={19},
  year={2010},
  publisher={American mathematical society}
}

@book{dafermos2005hyberbolic,
  title={Hyperbolic conservation laws in continuum physics},
  author={Dafermos, Constantine M},
  year={2005},
  publisher={Springer}
}

@incollection{cercignani1988boltzmann,
  title={The {Boltzmann} equation},
  author={Cercignani, Carlo},
  booktitle={The Boltzmann equation and its applications},
  pages={40--103},
  year={1988},
  publisher={Springer}
}

@article{morrison1998hamiltonian,
  title={Hamiltonian description of the ideal fluid},
  author={Morrison, Philip J},
  journal={Reviews of modern physics},
  volume={70},
  number={2},
  pages={467},
  year={1998},
  publisher={APS}
}

@article{akram22sign,
title = {The Signed Cumulative Distribution Transform for 1-D signal analysis and classification},
journal = {Foundations of Data Science},
volume = {4},
number = {1},
pages = {137-163},
year = {2022},
author = {Aldroubi, Akram  and Díaz Martín, Rocio  and Medri, Ivan and Rohde,  Gustavo K.  and Thareja, Sumati}
}

@article{schmid2022dynamic,
  title={Dynamic mode decomposition and its variants},
  author={Schmid, Peter J},
  journal={Annual Review of Fluid Mechanics},
  volume={54},
  number={1},
  pages={225--254},
  year={2022},
  publisher={Annual Reviews}
}

@phdthesis{tu2013dynamic,
  title={Dynamic mode decomposition: Theory and applications},
  author={Tu, Jonathan H},
  year={2013},
  school={Princeton University}
}

@article{berkooz1993proper,
  title={The proper orthogonal decomposition in the analysis of turbulent flows},
  author={Berkooz, Gal and Holmes, Philip and Lumley, John L},
  journal={Annual review of fluid mechanics},
  volume={25},
  number={1},
  pages={539--575},
  year={1993},
  publisher={Annual Reviews 4139 El Camino Way, PO Box 10139, Palo Alto, CA 94303-0139, USA}
}

@article{lu2020lagrangian,
  title={Lagrangian dynamic mode decomposition for construction of reduced-order models of advection-dominated phenomena},
  author={Lu, Hannah and Tartakovsky, Daniel M},
  journal={Journal of Computational Physics},
  volume={407},
  pages={109229},
  year={2020},
  publisher={Elsevier}
}

@article {FBlack_PSchulze_BUnger_2020a,
    AUTHOR = {F. Black and Ph. Schulze and B. Unger},
   FAUTHOR = {Black, Felix and Schulze, Philipp and Unger, Benjamin},
     TITLE = {Projection-based model reduction with dynamically transformed
              modes},
   JOURNAL = {ESAIM Math. Model. Numer. Anal.},
  FJOURNAL = {ESAIM. Mathematical Modelling and Numerical Analysis},
    VOLUME = {54},
      YEAR = {2020},
    NUMBER = {6},
     PAGES = {2011--2043}
     }

@article {BPeherstorfer_2020a,
    AUTHOR = {B. Peherstorfer},
   FAUTHOR = {Peherstorfer, Benjamin},
     TITLE = {Model {R}eduction for {T}ransport-{D}ominated {P}roblems via
              {O}nline {A}daptive {B}ases and {A}daptive {S}ampling},
   JOURNAL = {SIAM J. Sci. Comput.},
  FJOURNAL = {SIAM Journal on Scientific Computing},
    VOLUME = {42},
      YEAR = {2020},
    NUMBER = {5},
     PAGES = {A2803--A2836}
}

@article{amsallem2012nonlinear,
  title={Nonlinear model order reduction based on local reduced-order bases},
  author={Amsallem, David and Zahr, Matthew J and Farhat, Charbel},
  journal={International Journal for Numerical Methods in Engineering},
  volume={92},
  number={10},
  pages={891--916},
  year={2012},
  publisher={Wiley Online Library}
}

@book{pinkus2012n,
  title={${N}$-widths in Approximation Theory},
  author={Pinkus, Allan},
  volume={7},
  year={2012},
  publisher={Springer Science \& Business Media}
}

@article{arbes2025kolmogorov,
  title={The {Kolmogorov} {$n$}-width for linear transport: exact representation and the influence of the data},
  author={Arbes, Florian and Greif, Constantin and Urban, Karsten},
  journal={Advances in Computational Mathematics},
  volume={51},
  number={2},
  pages={13},
  year={2025},
  publisher={Springer}
}

@article{park2018cumulative,
  title={The cumulative distribution transform and linear pattern classification},
  author={Park, Se Rim and Kolouri, Soheil and Kundu, Shinjini and Rohde, Gustavo K},
  journal={Applied and computational harmonic analysis},
  volume={45},
  number={3},
  pages={616--641},
  year={2018},
  publisher={Elsevier}
}

@article{ivan2024data,
  title={Data representation with optimal transport},
  author={D{\'\i}az Mart{\'\i}n, Roc{\'\i}o  and Medri, Ivan V and Rohde, Gustavo Kunde},
  journal={arXiv preprint arXiv:2406.15503},
  year={2024}
}

@article{rubaiyat2024end,
  title={End-to-end signal classification in signed cumulative distribution transform space},
  author={Rubaiyat, Abu Hasnat Mohammad and Li, Shiying and Yin, Xuwang and Shifat-E-Rabbi, Mohammad and Zhuang, Yan and Rohde, Gustavo K},
  journal={IEEE Transactions on Pattern Analysis and Machine Intelligence},
  volume={46},
  number={9},
  pages={5936--5950},
  year={2024},
  publisher={IEEE}
}

@article{rubaiyat2024data,
  title={Data-driven identification of parametric governing equations of dynamical systems using the signed cumulative distribution transform},
  author={Rubaiyat, Abu Hasnat Mohammad and Thai, Duy H and Nichols, Jonathan M and Hutchinson, Meredith N and Wallen, Samuel P and Naify, Christina J and Geib, Nathan and Haberman, Michael R and Rohde, Gustavo K},
  journal={Computer methods in applied mechanics and engineering},
  volume={422},
  pages={116822},
  year={2024},
  publisher={Elsevier}
}

@article{rubaiyat2020parametric,
  title={Parametric signal estimation using the cumulative distribution transform},
  author={Rubaiyat, Abu Hasnat Mohammad and Hallam, Kyla M and Nichols, Jonathan M and Hutchinson, Meredith N and Li, Shiying and Rohde, Gustavo K},
  journal={IEEE Transactions on Signal Processing},
  volume={68},
  pages={3312--3324},
  year={2020},
  publisher={IEEE}
}

@inproceedings{rubaiyat2022nearest,
  title={Nearest subspace search in the signed cumulative distribution transform space for 1d signal classification},
  author={Rubaiyat, Abu Hasnat Mohammad and Shifat-E-Rabbi, Mohammad and Zhuang, Yan and Li, Shiying and Rohde, Gustavo K},
  booktitle={ICASSP 2022-2022 IEEE International Conference on Acoustics, Speech and Signal Processing (ICASSP)},
  pages={3508--3512},
  year={2022},
  organization={IEEE}
}

@misc{pytranskit,
  author = {G Rohde},
  title = {PyTransKit},
  year = {2021},
  publisher = {GitHub},
  journal = {GitHub repository},
  howpublished = {\url{https://github.com/rohdelab/PyTransKit}}
}

@book{hytonen2016analysis,
  title={Analysis in Banach spaces},
  author={Hyt{\"o}nen, Tuomas and Van Neerven, Jan and Veraar, Mark and Weis, Lutz},
  volume={1},
  year={2016},
  publisher={Springer}
}

@book{CoddingtonLevinson1955,
  author    = {Coddington, Earl A. and Levinson, Norman},
  title     = {Theory of Ordinary Differential Equations},
  publisher = {McGraw-Hill},
  address   = {New York},
  year      = {1955}
}

@article{Kruzhkov1969Continuity,
  author  = {Kru\v{z}kov, S. N.},
  title   = {Results concerning the nature of the continuity of solutions of parabolic equations and some of their applications},
  journal = {Mathematical Notes of the Academy of Sciences of the USSR},
  volume  = {6},
  number  = {1},
  pages   = {517--523},
  year    = {1969}
}

@article{algoritmy,
	author = {Mario Ohlberger and Stephan Rave},
	title = {Reduced Basis Methods: Success, Limitations and Future Challenges},
	journal = {Proceedings of the Conference Algoritmy},
	year = {2016},
	pages = {1--12}
}

@article{taddei2020registration,
  title={A registration method for model order reduction: data compression and geometry reduction},
  author={Taddei, Tommaso},
  journal={SIAM Journal on Scientific Computing},
  volume={42},
  number={2},
  pages={A997--A1027},
  year={2020},
  publisher={SIAM}
}

@article{rim2018model,
  title={Model reduction of a parametrized scalar hyperbolic conservation law using displacement interpolation},
  author={Rim, Donsub and Mandli, Kyle T},
  journal={arXiv preprint arXiv:1805.05938},
  year={2018}
}

@incollection{cagniart2018model,
  title={Model order reduction for problems with large convection effects},
  author={Cagniart, Nicolas and Maday, Yvon and Stamm, Benjamin},
  booktitle={Contributions to partial differential equations and applications},
  pages={131--150},
  year={2018},
  publisher={Springer}
}

@article{rim2018displacement,
  title={Displacement interpolation using monotone rearrangement},
  author={Rim, Donsub and Mandli, Kyle T},
  journal={SIAM/ASA Journal on Uncertainty Quantification},
  volume={6},
  number={4},
  pages={1503--1531},
  year={2018},
  publisher={SIAM}
}

@article{welper2017interpolation,
  title={Interpolation of functions with parameter dependent jumps by transformed snapshots},
  author={Welper, Gerrit},
  journal={SIAM Journal on Scientific Computing},
  volume={39},
  number={4},
  pages={A1225--A1250},
  year={2017},
  publisher={SIAM}
}

@article{rim2023manifold,
  title={Manifold approximations via transported subspaces: Model reduction for transport-dominated problems},
  author={Rim, Donsub and Peherstorfer, Benjamin and Mandli, Kyle T},
  journal={SIAM Journal on Scientific Computing},
  volume={45},
  number={1},
  pages={A170--A199},
  year={2023},
  publisher={SIAM}
}

@article{unger2019kolmogorov,
  author  = {Unger, Benjamin and Gugercin, Serkan},
  title   = {Kolmogorov {$n$}-widths for linear dynamical systems},
  journal = {Advances in Computational Mathematics},
  volume  = {45},
  pages   = {2273--2286},
  year    = {2019}
}

@article{mojgani2023kolmogorov,
  author  = {Mojgani, Rambod and Balajewicz, Maciej and Hassanzadeh, Pedram},
  title   = {Kolmogorov {$n$}-width and {Lagrangian} physics-informed neural networks:
             A causality-conforming manifold for convection-dominated {PDEs}},
  journal = {Computer Methods in Applied Mechanics and Engineering},
  volume  = {404},
  pages   = {115810},
  year    = {2023}
}

@article{wotte2026model,
  author  = {Wotte, Yannik P. and Buchfink, Patrick and Glas, Silke and
             Califano, Federico and Stramigioli, Stefano},
  title   = {Model order reduction via {Lie} groups},
  journal = {Communications in Nonlinear Science and Numerical Simulation},
  volume  = {160},
  pages   = {109973},
  year    = {2026}
}

@article{hesthaven2026nonlinear,
  author  = {Hesthaven, Jan S. and Peherstorfer, Benjamin and Unger, Benjamin},
  title   = {Nonlinear model reduction for transport-dominated problems},
  journal = {Acta Numerica},
  volume  = {35},
  year    = {2026},
  eprint  = {2602.01397},
  archivePrefix = {arXiv}
}

@incollection{gonnella2026nonlinear,
  author    = {Gonnella, Isabella Carla and Pichi, Federico and Rozza, Gianluigi},
  title     = {Nonlinear Reduction Strategies for Data Compression:
               A Comprehensive Comparison from Diffusion to Advection Problems},
  booktitle = {Scientific Machine Learning},
  series    = {SEMA SIMAI Springer Series},
  volume    = {42},
  pages     = {177--198},
  publisher = {Springer},
  address   = {Cham},
  year      = {2026}
}

@article{nonino2023overcoming,
  author  = {Nonino, Monica and Ballarin, Francesco and Rozza, Gianluigi and Maday, Yvon},
  title   = {Overcoming slowly decaying {Kolmogorov} {$n$}-width by transport maps:
             Application to model order reduction of fluid dynamics and
             fluid--structure interaction problems},
  journal = {Advances in Computational Science and Engineering},
  volume  = {1},
  number  = {1},
  pages   = {36--58},
  year    = {2023}
}

@article{thiffeault2003advection,
  title={Advection--diffusion in {Lagrangian} coordinates},
  author={Thiffeault, Jean-Luc},
  journal={Physics Letters A},
  volume={309},
  number={5-6},
  pages={415--422},
  year={2003},
  publisher={Elsevier}
}

@article{peherstorfer2022breaking,
  author  = {Peherstorfer, Benjamin},
  title   = {Breaking the {Kolmogorov} Barrier with Nonlinear Model Reduction},
  journal = {Notices of the American Mathematical Society},
  volume  = {69},
  number  = {5},
  pages   = {725--733},
  year    = {2022}
}

@article{greif2019decay,
  author  = {Greif, Constantin and Urban, Karsten},
  title   = {Decay of the {Kolmogorov} {$n$}-width for wave problems},
  journal = {Applied Mathematics Letters},
  volume  = {96},
  pages   = {216--222},
  year    = {2019}
}

@article{kruvzkov1970first,
  title={First order quasilinear equations in several independent variables},
  author={Kru{\v{z}}kov, Stanislav N},
  journal={Mathematics of the USSR-Sbornik},
  volume={10},
  number={2},
  pages={217--243},
  year={1970}
}

@article{benamou2000computational,
  title={A computational fluid mechanics solution to the {Monge-Kantorovich} mass transfer problem},
  author={Benamou, Jean-David and Brenier, Yann},
  journal={Numerische Mathematik},
  volume={84},
  number={3},
  pages={375--393},
  year={2000},
  publisher={Springer-Verlag Berlin/Heidelberg}
}

@article{otto2001geometry,
  title={The geometry of dissipative evolution equations: The porous medium equation},
  author={Otto, Felix},
  journal={Communications in Partial Differential Equations},
  volume={26},
  number={1--2},
  pages={101--174},
  year={2001},
  publisher={Taylor \& Francis}
}

@book{villani2003topics,
  title={Topics in optimal transportation},
  author={Villani, C{\'e}dric},
  number={58},
  year={2003},
  publisher={American Mathematical Soc.}
}

@book{ambrosio2005gradient,
  title={Gradient flows: in metric spaces and in the space of probability measures},
  author={Ambrosio, Luigi and Gigli, Nicola and Savar{\'e}, Giuseppe},
  year={2005},
  publisher={Springer Science \& Business Media}
}

@article{reiss2018shifted,
  author  = {Reiss, Julius and Schulze, Philipp and Sesterhenn, J{\"o}rn and Mehrmann, Volker},
  title   = {The Shifted Proper Orthogonal Decomposition: A Mode Decomposition for Multiple Transport Phenomena},
  journal = {SIAM Journal on Scientific Computing},
  volume  = {40},
  number  = {3},
  pages   = {A1322--A1344},
  year    = {2018}
}

@article{liu2022neural,
  author  = {Liu, Shu and Li, Wuchen and Zha, Hongyuan and Zhou, Haomin},
  title   = {Neural {Parametric Fokker--Planck Equation}},
  journal = {SIAM Journal on Numerical Analysis},
  volume  = {60},
  number  = {3},
  pages   = {1385--1449},
  year    = {2022}
}

@article{jin2025parameterized,
  author  = {Jin, Yijie and Liu, Shu and Wu, Hao and Ye, Xiaojing
             and Zhou, Haomin},
  title   = {Parameterized {Wasserstein Gradient Flow}},
  journal = {Journal of Computational Physics},
  volume  = {524},
  pages   = {113660},
  year    = {2025}
}

@article{zuo2026numerical,
  author  = {Zuo, Xinzhe and Zhao, Jiaxi and Liu, Shu and Osher, Stanley
             and Li, Wuchen},
  title   = {Numerical Analysis on Neural Network Projected Schemes for             Approximating One Dimensional {Wasserstein} Gradient Flows},
  journal = {Journal of Computational Physics},
  volume  = {546},
  pages   = {114501},
  year    = {2026}
}

@article{benner2015survey,
  author  = {Benner, Peter and Gugercin, Serkan and Willcox, Karen},
  title   = {A Survey of Projection-Based Model Reduction Methods for
             Parametric Dynamical Systems},
  journal = {SIAM Review},
  volume  = {57},
  number  = {4},
  pages   = {483--531},
  year    = {2015}
}

@article{antil2026discrete,
  title={A Discrete Cumulative Distribution Transform via Optimal Transport},
  author={Antil, Harbir and Rohde, Gustavo and Saxena, Aryan},
  journal={arXiv preprint arXiv:2606.12131},
  year={2026}
}

@inbook{HAntil_MHeinkenschloss_DCSorensen_2013a,
	address = {Cham},
	author = {Antil, Harbir and Heinkenschloss, Matthias and Sorensen, Danny C.},
	booktitle = {Reduced Order Methods for Modeling and Computational Reduction},
	editor = {Quarteroni, Alfio and Rozza, Gianluigi},
	isbn = {978-3-319-02090-7},
	pages = {101--136},
	publisher = {Springer International Publishing},
	title = {Application of the Discrete Empirical Interpolation Method to Reduced Order Modeling of Nonlinear and Parametric Systems},	
    year = {2014}
    }

@article{HAntil_MHeinkenschloss_RHoppe_2010c,
	author = {Antil, H. and Heinkenschloss, M. and Hoppe, R. H.W.},
	doi = {10.1080/10556781003767904},
	eprint = {http://www.tandfonline.com/doi/pdf/10.1080/10556781003767904},
	journal = {Optimization Methods and Software},
	number = {4-5},
	pages = {643-669},
	title = {Domain decomposition and balanced truncation model reduction for shape optimization of the Stokes system},
	volume = {26},
	year = {2011}}

@article{HAntil_MHeinkenschloss_RHoppe_DCSorensen_2010b,
	author = {Antil, H. and Heinkenschloss, M. and Hoppe, R.H.W. and Sorensen, D. C.},
	fjournal = {Computing and Visualization in Science},
	issn = {1432-9360},
	journal = {Comput. Vis. Sci.},
	mrclass = {49M25 (49Q10 65K10 65M55)},
	mrnumber = {2748450},
	number = {6},
	pages = {249--264},
	title = {Domain decomposition and model reduction for the numerical solution of {PDE} constrained optimization problems with localized optimization variables},
	volume = {13},
	year = {2010}}

@article{KCarlberg_MBarone_HAntil_2017a,
	author = {Carlberg, Kevin and Barone, Matthew and Antil, Harbir},
	fauthor = {Carlberg, Kevin and Barone, Matthew and Antil, Harbir},
	fjournal = {Journal of Computational Physics},
	issn = {0021-9991},
	journal = {J. Comput. Phys.},
	mrclass = {65L05 (65M20)},
	mrnumber = {3581488},
	pages = {693--734},
	title = {Galerkin v. least-squares {P}etrov-{G}alerkin projection in nonlinear model reduction},
	volume = {330},
	year = {2017}}

\end{document}